\documentclass[10pt,reqno]{amsart} 
\usepackage{amssymb,color,eucal,mathrsfs,cite}
\usepackage{mathptmx} 
\usepackage{enumitem} 
\usepackage[text={15cm,23cm},centering,foot=1cm]{geometry} 
\usepackage{comment}

\usepackage[normalem]{ulem} 

\usepackage{mathtools}
\usepackage{tikz}
\usepackage{tikz-cd}
\usepackage[matrix,arrow,curve,frame]{xy}

\usepackage{setspace}

\usepackage[capitalise]{cleveref} 

\setitemize{leftmargin=*, itemsep=3pt, topsep=3pt}   % Removes margin from itemized lists (enumitem package)
\setenumerate{leftmargin=*, itemsep=3pt, topsep=3pt, % Removes margin for enumerate lists (enumitem package)
              label=\textup{(\arabic*)}}             % and makes labels (1) in lists and refs

\DeclareSymbolFont{largesymbols}{OMX}{zplm}{m}{n} % Replaces summation/product symbols in mathptmx by the palatino ones...

\allowdisplaybreaks 

\newtheorem{thm}{Theorem}[section]
\newtheorem{cor}[thm]{Corollary}
\newtheorem{lemma}[thm]{Lemma}

\newtheorem{prop}[thm]{Proposition}

\theoremstyle{definition}
\newtheorem{defin}[thm]{Definition}
\newtheorem{remark}[thm]{Remark}
\newtheorem{example}{Example} 

\Crefname{thm}{Theorem}{Theorems}
\Crefname{prop}{Proposition}{Propositions}
\Crefname{lemma}{Lemma}{Lemmas}
\Crefname{cor}{Corollary}{Corollaries}
\Crefname{defin}{Definition}{Definitions}

\numberwithin{equation}{section}

\newcommand{\cW}{\VOA{W}}  

\newcommand{\alg}[1]{\mathfrak{#1}}  % Lie algebras
\newcommand{\Mod}[1]{\mathsf{#1}}    % modules
\newcommand{\VOA}[1]{\mathsf{#1}}    % VOAs (chosen to be same as modules)
\newcommand{\grp}[1]{\mathcal{#1}}   % groups, lattices and orbits
\newcommand{\fld}[1]{\mathbb{#1}}    % fields and so forth
\newcommand{\categ}[1]{\mathscr{#1}} % categories

\newcommand{\set}[1]{\left\{ #1 \right\}}
\newcommand{\st}{\mspace{3mu} : \mspace{3mu}} % "such that" in sets
\newcommand{\bilin}[2]{\left\langle #1 , #2 \right\rangle}
\newcommand{\abs}[1]{\left\lvert #1 \right\rvert}
\newcommand{\norm}[1]{\left\lVert #1 \right\rVert}

\newcommand{\NN}{\fld{N}}
\newcommand{\ZZ}{\fld{Z}}
\newcommand{\QQ}{\fld{Q}}
\newcommand{\RR}{\fld{R}}
\newcommand{\CC}{\fld{C}}

\DeclareMathOperator{\End}{End}
\DeclareMathOperator{\id}{id}
\DeclareMathOperator{\ptr}{ptr}

\newcommand{\voa}{vertex operator algebra}
\newcommand{\voas}{vertex operator algebras}
\newcommand{\vosa}{vertex operator subalgebra}
\newcommand{\vosas}{vertex operator subalgebras}

\newcommand{\opes}{operator product expansions}
\newcommand{\hw}{highest-weight}
\newcommand{\hwv}{\hw{} vector}
\newcommand{\hwvs}{\hw{} vectors}

\newcommand{\hwms}{\hw{} modules}

\newcommand{\fus}[1]{\mathbin{\boxtimes_{#1}}} % fusion
\newcommand{\fusco}[3]{\binom{#3}{#1\:#2}}     % fusion coefficient / intertwiner type
\newcommand{\intw}{\mathcal{Y}}                % intertwiner

\newcommand{\Com}[2]{\operatorname{Com} ( #1 , #2 )} % commutator of #1 in #2

\newcommand{\sfaut}{\sigma}                          % spectral flow
\newcommand{\sfmod}[2]{\sfaut^{#1} (#2)}             % apply spectral flow #1 times to #2

\newcommand{\affine}[1]{\widehat{#1}}
\newcommand{\SLA}[2]{\alg{#1}_{#2}}                             % Lie algebras like sl(2)
\newcommand{\SLSAp}[3]{\alg{#1}(#2 \vert #3)}                   % Lie superalgebras like gl(1|1), paranthesized
\newcommand{\AKMA}[2]{\affine{\alg{#1}}_{#2}}                   % Kac-Moody algebras
\newcommand{\MinMod}[2]{\VOA{M} \bigl( #1 , #2 \bigr)}          % Virasoro minimal models
\newcommand{\UAffVOA}[2]{\VOA{V}_{#1}(#2)}                      % Universal affine VOAs (#1 = level)
\newcommand{\AffVOA}[2]{\VOA{L}_{#1}(#2)}                       % Simple quotient affine VOAs
\newcommand{\Sing}[1]{\VOA{I}(#1)}                              % Singlet VOAs
\newcommand{\Trip}[1]{\VOA{W}(#1)}                              % Triplet VOAs

\DeclareMathOperator{\tr}{tr}
\newcommand{\traceover}[1]{\tr_{\raisebox{-3pt}{$\scriptstyle #1$}}} % for ch[M] = tr_M ...
\newcommand{\chmap}{\mathrm{ch}}
\newcommand{\Gr}[1]{\bigl[ #1 \bigr]}                                % element of a Grothendieck group/ring
\newcommand{\ch}[1]{\chmap \Gr{#1}}                                  % characters
\newcommand{\fch}[2]{\ch{#1} \bigl( #2 \bigr)}                       % characters as functions of q and ...

\newcommand{\blank}{{-}} % a minus sign that is not a binary operator (used for missing arguments)
\DeclareMathOperator{\rad}{rad}
\DeclareMathOperator{\soc}{soc}

\newcommand{\Ra}{\Rightarrow}
\newcommand{\ra}{\rightarrow}
\newcommand{\lra}{\longrightarrow}
\newcommand{\lsra}{\ensuremath{\relbar\joinrel\twoheadrightarrow}}        % long surjection

\newcommand{\ses}[3]{0 \ra #1 \ra #2 \ra #3 \ra 0}
\newcommand{\dses}[3]{0 \lra #1 \lra #2 \lra #3 \lra 0}
\newcommand{\drses}[3]{#1 \lra #2 \lra #3 \lra 0}

\newcommand{\vac}{{\mathbf{1}}}
\DeclareMathOperator{\Aut}{Aut}
\DeclareMathOperator{\Id}{Id}

\newcommand{\cY}{\mathcal Y}

\begin{document}

\title{Schur-Weyl Duality for Heisenberg Cosets}

\author{Thomas Creutzig}
\address{
	Department of Mathematical and Statistical Sciences \\
	University of Alberta\\
	Edmonton, Canada T6G 2G1
}
\email{creutzig@ualberta.ca}
\thanks{T.C. is supported by NSERC discovery grant  \#RES0020460}

\author{Shashank Kanade}
\address{
	Department of Mathematical and Statistical Sciences \\
	University of Alberta\\
	Edmonton, Canada T6G 2G1
}
\email{kanade@ualberta.ca}
\thanks{S.K. is supported by PIMS post-doctoral fellowship}

\author{Andrew R. Linshaw}
\address{
	Department of Mathematics \\
	University of Denver \\
	Denver, USA, 80208}
\email{andrew.linshaw@du.edu}
\thanks{A.R.L. is supported by Simons Foundation Grant \#318755}

\author{David Ridout}
\address{
School of Mathematics and Statistics \\
The University of Melbourne \\
Parkville, Australia, 3010}
\email{david.ridout@unimelb.edu.au}
\thanks{D.R.'s research is supported by the Australian Research Council Discovery Projects DP1093910 and DP160101520}


%%%%%%%%

\begin{abstract}
Let $\VOA{V}$ be a simple vertex operator algebra containing a rank $n$ Heisenberg vertex algebra $\VOA{H}$ and let $\VOA{C}=\text{Com}\left( \VOA{H}, \VOA{V}\right)$ be the coset of $\VOA{H}$ in $\VOA{V}$. Assuming that the representation categories of interest are vertex tensor categories in the sense of Huang, Lepowsky and Zhang, a Schur-Weyl type duality for both simple and indecomposable but reducible modules is proven. Families of vertex algebra extensions of $\VOA{C}$ are found and every simple $\VOA{C}$-module is shown to be contained in at least one $\VOA{V}$-module. A corollary of this is that if $\VOA{V}$ is rational and $C_2$-cofinite and  CFT-type, and $\text{Com}\left( \VOA{C}, \VOA{V}\right)$ is a rational lattice vertex operator algebra, then so is $\VOA{C}$. These results are illustrated with many examples and the $C_1$-cofiniteness of certain interesting classes of modules is established.
\end{abstract}

\maketitle

\onehalfspacing 

\section{Introduction}

Let $\VOA{V}$ be a 
\voa{}.\footnote{We mention that much of this discussion generalises immediately to vertex operator superalgebras.  However, we shall generally state results for \voas{} for simplicity, leaving explicit mention of the super-case to exceptions and examples.}  If $\grp{G}$ is a subgroup of the automorphism group of $\VOA{V}$, then the invariants $\VOA{V}^{\grp{G}}$ form a \vosa{} called the $\grp{G}$-orbifold of $\VOA{V}$. If $\VOA{W}$ is any \vosa{} of $\VOA{V}$, then the $\VOA{W}$-coset of $\VOA{V}$ is the commutant $\VOA{C}=\Com{\VOA{W}}{\VOA{V}}$. Both the orbifold and coset constructions provide a way to construct new \voas{} from known ones. Unfortunately, few general results concerning the structure of the resulting \vosas{} are known, but it is believed that many nice properties of $\VOA{V}$ are inherited by its orbifolds and cosets.  We remark that while most of the literature is primarily concerned with completely reducible representations of \voas{}, we are also interested in the logarithmic case in which the \voa{} admits indecomposable but reducible representations.

We begin by recalling some important results in the invariant theory of \voas{} that are connected to the questions that we address in this work.

\subsection{From classical to vertex-algebraic invariant theory}

It is valuable to view invariant-theoretic results about \voas{} as generalizations of the classical results, \`a la Howe and Weyl \cite{H, W}, concerning Lie algebras and groups. For example, a well-known result of Dong, Li and Mason \cite{DLM} amounts to a type of Schur-Weyl duality for orbifolds, stating that for a simple \voa{} $\VOA{V}$ and a compact subgroup $\grp{G}$ of $\Aut \VOA{V}$ (acting continuously and faithfully), the following decomposition holds as a $\grp{G} \times \VOA{V}^{\grp{G}}$-module:
\begin{equation}
	\VOA{V}= \bigoplus_{\lambda} \lambda \otimes \Mod{V_\lambda}.
\end{equation}
Here, the sum runs over all the simple $\grp{G}$-modules $\lambda$ and is multiplicity-free in the sense that $\VOA{V}_\lambda \not\cong \VOA{V}_\mu$ if $\lambda\neq \mu$.  They moreover prove that the $\Mod{V}_\lambda$ are simple modules for the orbifold \voa{} $\VOA{V}^{\grp{G}}$. Similar results have also been obtained by Kac and Radul \cite{KR} (see \cref{sec:KR}).

Invariant theory for the classical groups \cite{W} can be used to obtain generators and relations for orbifold \voas{} $\VOA{V}^{\grp{G}}$, provided that $\VOA{V}$ is of free field type (meaning that the only field appearing in the singular terms of the \opes{} of the strong generators is the identity field). 
Interestingly, the relations can be used to show that these \voas{} are strongly finitely generated and, in many cases, explicit minimal strong generating sets can be obtained \cite{L1, L2, L3, L4, L6, CLIII}. Questions concerning cosets are usually more involved than their orbifold counterparts. However, the notion of a deformable family of \voas{} \cite{CLII} can sometimes be used to reduce the identification of a minimal strong generating set for a coset to a known orbifold problem for a free field algebra \cite{CLI}.

It is of course desirable to understand the representation theory of coset \voas{}. An important first question to ask is if there is also a Schur-Weyl type duality, as in the orbifold case. Let $\VOA{V}$ be a simple \voa{} that is self-contragredient and let $\VOA{A}, \VOA{B} \subseteq \VOA{V}$ be \vosas{} satisfying
\begin{equation}
	\VOA{A}=\Com{\VOA{B}}{\VOA{V}} \qquad \text{and} \qquad \VOA{B}=\Com{\VOA{A}}{\VOA{V}}.
\end{equation}
Under the further assumption that $\VOA{A}$ and $\VOA{B}$ are both simple, self-contragredient, regular and of CFT-type,
\begin{equation}
	\VOA{V} =\bigoplus_i \Mod{M}_i \otimes \Mod{N}_i
\end{equation}
as an $\VOA{A}\otimes \VOA{B}$-module, where each $\Mod{M}_i$ is a simple $\VOA{A}$-module and each $\Mod{N}_i$ is a simple $\VOA{B}$-module.  
Under further conditions, Lin finds \cite{Lin} that this decomposition is multiplicity-free and the argument relies on knowing that the representation categories of $\VOA{A}$ and $\VOA{B}$ are both semisimple modular tensor categories.

We are aiming for similar results, but generalised to include decompositions of modules that are not necessarily semisimple. Our setup is that $\VOA{V}$ is a simple \voa{} containing a Heisenberg \vosa{} $\VOA{H}$. We then study the commutant $\VOA{C} = \Com{\VOA{H}}{\VOA{V}}$. For this, we assume that $\VOA{C}$ has a module category $\categ{C}$ that is a vertex tensor category in the sense of Huang, Lepowsky and Zhang \cite{HLZ} and that the $\VOA{C}$-modules obtained upon decomposing $\VOA{V}$ as an $\VOA{H} \otimes \VOA{C}$-module belong to $\categ{C}$. In \cref{sec:HLZ}, we summarize some known statements about vertex tensor categories that are relevant for our study.  These statements make it clear that $C_1$-cofiniteness of modules is a key concept.  In \cref{sec:C1}, we establish the $C_1$-cofiniteness of Heisenberg coset modules in two families of examples.

\subsection{Rational parafermion \voas{}}

Heisenberg cosets of rational affine \voas{} are usually called parafermion \voas{}.
They first appeared in the form of the $Z$-algebras discovered by Lepowsky and Wilson in \cite{LW1, LW2, LW3, LW4}, see also \cite{LP}.  In physics, parafermions first appeared in the work of Fateev and Zamolodchikov \cite{FZ} where they were given their standard appellation.  The relation between parafermion \voas{} and $Z$-algebras was subsequently clarified in \cite{DL}.

Parafermions are surely among the best understood coset \voas{} and there has been substantial recent progress towards establishing a complete picture of their properties. Key results include $C_2$-cofiniteness \cite{ALY}, see also \cite{DLY,DW}, and rationality \cite{DR}, using previous results on strong generators \cite{DLWY}.  In principle, strong generators can now also be determined as in \cite{CLI}, where this was detailed for the parafermions related to $\SLA{sl}{3}$. We remark that $C_2$-cofiniteness also follows from a recent result of Miyamoto on orbifold \voas{} \cite{M1}. These powerful results also allow one, for example, to compute fusion coefficients \cite{DW2}.

It has, for some time, been believed that if a simple, rational, $C_2$-cofinite, self-contragredient \voa{} of CFT-type contains a lattice \vosa{} (corresponding to an even positive-definite lattice), then the corresponding coset \voa{} will also be rational. For example, this was recently shown using indirect methods for the rational Bershadsky-Polyakov \voas{} \cite{ACL}. We prove this statement in general (see \cref{thm:rational}).

\subsection{Results}

This work is, at least in part, a continuation of our previous work on simple current extensions of \voas{} \cite{CKL}. In this vein, we start by proving some properties of simple currents (\cref{prop:simplecurrents}), in particular that fusing with a simple current defines an autoequivalence of any suitable category of modules. As further preparation, we also prove (\cref{thm:MiySimpleCurrent}) that if $\VOA{V}$ is simple, $\grp{G}$ is an abelian group of $\VOA{V}$-automorphisms acting semisimply on $\VOA{V}$, and
\begin{equation}\label{eq:orbifolddecomp}
	\VOA{V} = \bigoplus_{\lambda\in\grp{L}\subset\widehat{\grp{G}}} \VOA{V}_\lambda,
\end{equation}
then $\VOA{V}_\lambda$ is a simple current for every $\lambda$ in $\grp{L}$.
The proof essentially amounts to adding details to a very similar result of Miyamoto \cite[Sec.~6]{M2}, \cite{CarMi}.

\subsubsection{Schur-Weyl duality}

We then prove a Schur-Weyl duality for Heisenberg cosets $\VOA{C} = \Com{\VOA{H}}{\VOA{V}}$. 
The set-up is as follows. Let $\VOA{V}$ be a simple \voa, $\VOA{H} \subseteq \VOA{V}$ be a Heisenberg \vosa{} that acts semisimply on $\VOA{V}$, $\VOA{C} = \Mod{C}_0$ be the commutant of $\VOA{H}$ in $\VOA{V}$ and $\grp{L}$ be the lattice of Heisenberg weights of $\VOA{V}$.  Here $\VOA{V}$ is regarded as an $\VOA{H}$-module. 
Then $\VOA{W}=\Com{\VOA{C}}{\VOA{V}}$  is an extension of $\VOA{H}$ by an abelian intertwining algebra. Of course, it might happen that the extension is trivial, that is, $\VOA{H}=\VOA{W}$. 
\cref{eq:orbifolddecomp} translates into
\begin{equation} \label{eq:DecompVFCintro}
	\VOA{V} = \bigoplus_{\lambda \in \grp{L}} \Mod{F}_\lambda \otimes \Mod{C}_\lambda.
\end{equation}
Let $\grp{N}$ be the sublattice of all $\lambda \in \grp{L}$ for which $\Mod{C}_{\lambda} \cong \Mod{C}$. 
\cref{thm:SW-VOA} now says that the abelian group $\grp{L}/\grp{N}$ controls the decomposition of $\VOA{V}$ as a $\VOA{W} \otimes \VOA{C}$-module: 
\begin{equation} \label{eq:DecompVWCintro}
	\VOA{V}= \bigoplus_{[\lambda] \in \grp{L} / \grp{N}} \Mod{W}_{[\lambda]} \otimes \Mod{C}_{[\lambda]}.
\end{equation}
Moreover, %the decompositions \eqref{eq:DecompVFCintro} and \eqref{eq:DecompVWCintro} hold, where 
the $\Mod{C}_{[\lambda]}$, $\lambda \in \grp{L}/\grp{N}$ are simple currents for $\VOA{C}$ whose fusion products include 
$$\Mod{C}_{[\lambda]} \fus{\VOA{C}} \Mod{C}_{[\mu]} = \Mod{C}_{[\lambda + \mu]}.$$
This decomposition is multiplicity free in the sense that $\Mod{C}_{[\lambda]}\not\cong \Mod{C}_{[\mu]}$ if $[\lambda]\neq[\mu]$.	
The \voa
$$\VOA{W} = \bigoplus\limits_{\lambda \in \grp{N}} \Mod{F}_{\lambda}$$ is a simple current extension of $\VOA{H}$ and  the $\Mod{W}_{[\lambda]}$, $[\lambda] \in \grp{L} / \grp{N}$, are simple currents for $\VOA{W}$ with fusion products 
	$\Mod{W}_{[\lambda]} \fus{\VOA{W}} \Mod{W}_{[\mu]} = \Mod{W}_{[\lambda + \mu]}$.

We note that Li has proven \cite{Li} that $\bigoplus\limits_{\lambda \in\grp{L}/\grp{N}} \Mod{C}_{[\lambda]}$ is a generalized vertex algebra.

%%%%%%%%%%%%%%%%%%%%%%%%%%%%%%%%%%%%%%%%%%%%%%%%%%%%%%%%%%%%%%%%%%%%%%

The main Schur-Weyl duality result is then a similar decomposition for \voa{} modules, see \cref{thm:SW-Mod}.
For this let $\VOA{V}$, $\VOA{H}$, $\VOA{C}$, $\VOA{W}$, $\grp{L}$ and $\grp{N}$ be as above and let $\Mod{M}$ be a 
$\VOA{V}$-module
upon which $\VOA{H}$ acts semisimply.
Then,  $\Mod{M}$ decomposes as
\begin{equation} \label{eq:DecompMintro}
	\Mod{M} = \bigoplus_{\mu \in \grp{M}} \Mod{M}_\mu = \bigoplus_{\mu \in \grp{M}} \Mod{F}_{\mu} \otimes \Mod{D}_{\mu} = \bigoplus_{[\mu] \in \grp{M} / \grp{N}} \Mod{W}_{[\mu]} \otimes \Mod{D}_{[\mu]},
\end{equation}
where $\grp{M}$ is a union of $\grp{L}$-orbits and the $\Mod{D}_{\mu} = \Mod{D}_{[\mu]}$ are $\VOA{C}$-modules satisfying $\Mod{C}_{\lambda} \fus{\VOA{C}} \Mod{D}_{\mu} = \Mod{D}_{\lambda + \mu}$, for all $\lambda \in \grp{L}$ and $\mu \in \grp{M}$.  
%Let $\grp{M}$ be a single $\grp{L}$-orbit, then the message of  
Next, in \cref{thm:SW-Mod} we show that
each of the $\Mod{D}_{\mu}$ have the same decomposition structure as that of $\Mod{M}$. One example of this is if $\ses{\Mod{M}'}{\Mod{M}}{\Mod{M}''}$ is exact, with $\Mod{M}'$ and $\Mod{M}''$ non-zero, then $\Mod{M}'$ and $\Mod{M}''$ decompose as in \eqref{eq:DecompMintro}:
	\begin{equation} \label{eq:DecompM'}
		\Mod{M}' = \bigoplus_{\mu \in \grp{M}} \Mod{M}'_\mu = \bigoplus_{\mu \in \grp{M}} \Mod{F}_{\mu} \otimes \Mod{D}'_{\mu}, \qquad
		\Mod{M}'' = \bigoplus_{\mu \in \grp{M}} \Mod{M}''_\mu = \bigoplus_{\mu \in \grp{M}} \Mod{F}_{\mu} \otimes \Mod{D}''_{\mu}.
	\end{equation}
	Moreover, $\ses{\Mod{D}'_{\mu}}{\Mod{D}_{\mu}}{\Mod{D}''_{\mu}}$ is also exact, for all $\mu \in \grp{M}$. 

However, in general, multiplicity-freeness does not hold, for example, the parafermion coset of $\AffVOA{2}{\SLA{sl}{2}}$ 
yields an example of a coset module that appears twice in the decomposition of a simple $\AffVOA{2}{\SLA{sl}{2}}$-module. 
We give three criteria to guarantee that a given decomposition is multiplicity-free. One based on characters, one based on the signature of the lattice $\grp{L}$, and one based on open Hopf link invariants following \cite{CG, CG2}.

\subsubsection{Extensions of \voas{}}

Let $\grp{E}$ be a sublattice of $\grp{L}$. We would like to know if
\begin{equation}
	\VOA{C}_\grp{E} = \bigoplus_{\lambda\in\grp{E}} \Mod{C}_\lambda
\end{equation}
carries the structure of a \voa{} extending that of $\VOA{C} = \Mod{C}_0$. \cref{thm:extension},
which itself follows immediately from \cite{Li}, implies that this is the case provided that
\begin{equation}
	\VOA{W}_\grp{E} =\bigoplus_{\lambda\in\grp{E}} \Mod{W}_\lambda
\end{equation}
is a \voa{}. If $\grp{E}$ is a rank one subgroup, then this conclusion also follows from \cite{CKL}.

\subsubsection{Lifting Modules}

%%%%%%%%%%%%%%%%%%%%%%%%%%%%%%%%%%%%

Let $\Mod{D}$ be a ${\VOA{C}}$-module. We would like to know if it lifts to a $\VOA{C}_\grp{E}$-module and also if there exists a $\VOA{H}$-module $\Mod{F}_\beta$ such that $\Mod{F}_\beta\otimes\Mod{D}$ lifts to a $\VOA{V}$-module. 

This question is decided by the monodromy (composition of braidings)
\begin{equation}
	M_{\VOA{C}_\lambda, \Mod{D}} : \Mod{C}_\lambda\fus{}\Mod{D} \rightarrow \Mod{C}_\lambda\fus{}\Mod{D}.
\end{equation}
We have (\cref{thm:lift}):
Let $\Mod{D}$ be a generalized $\VOA{C}$-module that appears as a subquotient
of the fusion product of some finite collection of simple $\VOA{C}$-modules.
Let $\grp{L}'$ be the dual lattice of $\grp{L}$ and let $U=\grp{L} \otimes_\ZZ \RR$.
Then, there exists $\alpha \in U$ such that
\begin{equation}
	M_{\VOA{C}_\lambda, \Mod{D}} = e^{-2\pi i \langle \alpha, \lambda \rangle} \Id_{\Mod{C}_\lambda\fus{}\Mod{D}}
\end{equation}
and $\Mod{F}_\beta\otimes\Mod{D}$ lifts to a $\VOA{V}$-module if and only if $\beta\in\alpha+\grp{L}'$. Moreover, the lifted module is $\VOA{V}  \fus{\VOA{H}\otimes\VOA{C}} \left(\Mod{F}_\beta\otimes\Mod{D}\right)$.
Note that the lifting problem when all involved \voas{} are regular was treated in \cite{KrMi}.

Further, $\Mod{D}$ lifts to a $\VOA{C}_\grp{E}$-module if and only if $\alpha$ is in a certain lattice associated to $\grp{E}$ (see \cref{cor:liftext}) and every $\VOA{C}_\grp{E}$-module is a quotient of a lifted module (this follows essentially from \cite{Lam}). The lifted module is then $\VOA{C}_\grp{E}\fus{\VOA{C}} \Mod{D}$.

\subsubsection{Rationality}

Miyamoto \cite{M1} has proven that $\VOA{C}$-is $C_2$-cofinite provided $\VOA{W}$ is a lattice \voa{} of a positive definite even lattice and provided $\VOA{V}$ is $C_2$-cofinite.
Together with our ability of controlling modules that lift to $\VOA{V}$-modules and exactness of fusion with simple currents this implies a rationality theorem (\cref{thm:rational}):

Let $\VOA{V}$ be simple, rational, {$C_2$-cofinite} of CFT-type. Then, every grading-restricted generalized $\VOA{C}$ module is completely reducible. 

Especially, we thus have an alternative proof 
of the rationality of the parafermion cosets \cite{DR}, \cite{CarMi} as well as of the Heisenberg cosets of the rational Bershadsky-Polyakov algebras \cite{ACL}.

\subsubsection{Examples}

We illustrate our results with various examples, both rational and non-rational ones, though 
our main interest are applications to \voas{} of logarithmic conformal field theory, that is especially to indecomposable but reducible modules. 
Schur-Weyl duality is exemplified in the well-known rational example of  $\AffVOA{2}{\SLA{sl}{2}}$ (\cref{ex:Parak=2}) and then in much detail in the case of $\AffVOA{-4/3}{\SLA{sl}{2}}$ (\cref{ex:Parak=-4/3}). We especially explain how Schur-Weyl duality works for the projective covers of simple modules. Extensions of the Heisenberg cosets of $\AffVOA{k}{\alg{g}}$ for rational and non-zero $k$ are discussed in
\cref{ex:sl2decompcont}. 
\cref{ex:singletgl11} then deals with the relation via Heisenberg cosets of various archetypical logarithmic \voas, most notably the $\Sing{2}$-singlet algebra  and  $\UAffVOA{k}{\SLSAp{gl}{1}{1}}$. Especially we give the decomposition of the projective indecomposable modules of 
$\UAffVOA{k}{\SLSAp{gl}{1}{1}}$ in terms of projective $\VOA{H}\otimes \VOA{H}\otimes \Sing{2}$-modules. The triplet algebra $\Trip{2}$ is then an example of an extended \voa{} that is $C_2$-cofinite. 
The lifting of modules is  illustrated in \cref{ex:N2} for the modules of the $N=2$ super Virasoro algebra.  
Finally, we use the opportunity to prove that $\AffVOA{-1}{\SLSAp{sl}{m}{n}}$ appears as a Heisenberg coset of appropriate $\beta\gamma$ and $bc$-\voas. This generalizes the case $n=0$ of \cite{AP}. Also the case $m=2$ and $n=0$ is exceptional and identified with a rectangular W-algebra of $\SLA{sl}{4}$.

\subsubsection{On $C_1$-cofiniteness}

Our results rely on the applicability of the vertex tensor theory of \cite{HLZ}. Our belief is that the key criterion for this applicability is the $C_1$-cofiniteness of the modules with finite composition length, see also \cite[Sec.~6]{CMR}.  In \cref{sec:C1}, we prove a few $C_1$-cofiniteness results for modules of Heisenberg cosets of the affine \voas{} of type $\SLA{sl}{2}$ as well as those of the Bershadsky-Polyakov algebras.

\subsubsection{Outlook on fusion}

The main concern of this work is the relationship between the modules of the Heisenberg coset \voa{} $\VOA{C}$ and those of its parent algebra $\VOA{V}$. A valid question is then if there is also a clear relation between the fusion product of the $\VOA{C}$-modules and the corresponding $\VOA{V}$-modules. This question is work in progress and here we announce that one can prove that the induction functor is a tensor functor under appropriate assumptions on the module category \cite{CKM}.  This rigorously establishes the connection between fusion and extended algebras that has been proposed in the physics literature \cite{RidVer14}.

\subsection{Application: Towards new $C_2$-cofinite logarithmic \voas{}}

Presently, there are very few known examples of $C_2$-cofinite non-rational \voas{}; these include the triplet algebras \cite{AM3, TW, TsuExt13} and their close relatives \cite{Ab-SF}.  In order to gain more experience with such logarithmic $C_2$-cofinite \voas{}, new examples are needed.  The main application we have in mind for the work reported here is the construction of new examples of this type.

The idea is a two-step process illustrated as follows:
\[
\VOA{V} \xrightarrow{\quad \VOA{H}\text{-coset}\quad} \VOA{C} \xrightarrow{\quad \text{extension}\quad} \VOA{C}_\grp{E}.
\]
A series of examples that confirms this idea were explored in \cite{CRW}, see also Example \ref{ex:sl2decompcont}. There, the $\Sing{p}$ singlet algebras of Kausch \cite{Ka} were (conjecturally) obtained as Heisenberg cosets of the Feigin-Semikhatov algebras \cite{FS}, see also \cite{GenScr16}.  The extension in the above process is then an infinite order simple current extension and the results \cite{CM1,RidMod13} are the best understood $C_2$-cofinite logarithmic \voas{}, the $\Trip{p}$ triplet algebras.

New examples may be obtained by taking $\VOA{V}$ to be the simple affine \voa{} associated to the simple Lie algebra $\alg{g}$ at admissible, but negative, level $k$ and $\VOA{H}$ to be the Heisenberg \vosa{} generated by the affine fields associated to the Cartan subalgebra of $\alg{g}$.  Here, $\grp{H}$ is a subgroup of $\grp{G}$ of maximal rank.
The module categories of such admissible level affine \voas{} remain quite mysterious despite strong results concerning category $\categ{O}$ \cite{KacMod88,AraRat15}.  Beyond category $\categ{O}$, detailed results are currently only known for $\alg{g}=\SLA{sl}{2}$ \cite{AM1, GabFus01, R1, R2, R3, CR1, CR2, RidRel15} and $\alg{g}=\SLA{sl}{3}$ \cite{AraWei16}. A first feasible task here would be to compute the characters of coset modules appearing in the decomposition of modules in $\categ{O}$. We expect the appearance of Kostant false theta functions \cite{CM2} as they are the natural generalization of ordinary false theta functions that appear in the case of the admissible level parafermion coset of $\AffVOA{k}{\SLA{sl}{2}}$ \cite{ACR}.

In \cite{ACR}, we will study $\VOA{C}_{\grp{E}}$ when $\alg{g}=\SLA{sl}{2}$ and $k$ is negative and admissible. Under the assumption that the tensor theory of Huang-Lepowsky-Zhang applies to $\VOA{C}$, we can prove that there are only finitely many inequivalent simple $\VOA{C}_\grp{E}$-modules. It is thus natural to conjecture that $\VOA{C}_{\grp{E}}$ is $C_2$-cofinite. A consequence of $C_2$-cofiniteness is modularity of characters (supplemented by pseudotrace functions) \cite{M3}. In \cite{ACR}, we can also demonstrate this modularity of characters (plus pseudotraces) for all modules that are lifts of $\VOA{C}$-modules. We will prove the $C_2$-cofiniteness of $\VOA{C}_\grp{E}$, for various choice of $\grp{E}$ in subsequent works. 

A third family of examples that fit this idea concern simple minimal W-algebras in the sense of Kac and Wakimoto \cite{KW}. These are quantum Hamiltonian reductions that are strongly generated by fields in conformal dimension one and $3/2$ together with the Virasoro field. 
For certain levels, these W-algebras have a one-dimensional associated variety and they 
contain a rational affine \vosa{}. The Heisenberg coset of the coset of the minimal $\VOA{W}$-algebra by the rational affine \voa{} thus seems to be another candidate for new $C_2$-cofinite algebras as infinite order simple current extensions. These cosets are explored in \cite{ACKL}.

\subsection{Organization}

We start with a background section. There we review the vertex tensor theory of Huang, Lepowsky and Zhang and especially discuss it in the case of the Heisenberg \voa. Next, we prove various properties of simple currents and then discuss \voa{} orbifolds following Kac and Radul. 
\cref{sec:SWDuals} is then on Schur-Weyl duality. \cref{sec:properties} is concerned with extended algebras, lifting of modules and as a special application proves our rationality theorem. in \cref{sec:betagamma} we give a short proof that $\AffVOA{-1}{\SLSAp{sl}{m}{n}}$ is a Heisenberg coset of appropriate $\beta\gamma$ and $bc$-\voas. In \cref{sec:C1} we prove $C_1$-cofiniteness of modules appearing in Heisenberg cosets of Bershadsky-Polyakov algebras and $\AffVOA{k}{\SLA{sl}{2}}$.

\subsection{Acknowledgments}
T.C. and S.K. would like to thank Yi-Zhi Huang and Robert McRae for helpful discussions regarding vertex tensor categories,
\cite{HLZ}. T.C. also thanks Antun Milas for discussions on the applicability of the theory of vertex tensor categories.

\section{Background}\label{sec:background}

In this section, we give a brief exposition of the results of Huang, Lepowsky and Zhang
regarding the vertex tensor categories that we shall use. We mention the case of Heisenberg
\voas{} separately in detail. Then, we present our new results regarding properties of
simple currents under fusion. After that, we review a {useful} result of Kac and Radul on the simplicity
of orbifold models.

\subsection{Conditions and assumptions regarding {the theory of Huang-Lepowsky-Zhang}} \label{sec:HLZ}
We begin with a quick glossary of the terminology that we shall use.
\begin{itemize}
\item By a \emph{generalized} module of a \voa{}, we shall mean a module that is graded by
generalized eigenvalues of $L_0$. A generalized module need not satisfy
any of the other restrictions mentioned below regarding grading.
For $n\in\CC$ and a generalized module $\Mod{W}$, we let $\Mod{W}_{[n]}$
denote the generalized $L_0$-eigenspace with generalized eigenvalue $n$.
\item A generalized module $\Mod{W}$ is called \emph{lower truncated}
if $\Mod{W}_{[n]}=0$ whenever the real part of $n$ is sufficiently negative.
\item A generalized module $\Mod{W}$ is called \emph{grading-restricted} if it is
lower truncated and if, moreover, for all $n$, $\mathrm{dim}(\Mod{W}_{[n]})<\infty$.
\item A generalized module $\Mod{W}$ is called \emph{strongly graded} if $\mathrm{dim}(\Mod{W}_{[n]})<\infty$ and, 
for each $n\in\CC$, $\Mod{W}_{[n+k]}=0$ for all sufficiently negative
integers $k$.  This notion is slightly more general than that of being grading-restricted.
\item In the definitions above, we shall replace the qualifier ``generalized'' with ``ordinary''
if the module is graded by eigenvalues of $L_0$ as opposed to generalized eigenvalues.
\item Henceforth, by ``module'', without qualifiers, we shall mean 
a grading-restricted generalized module.
{For convenience in the applications to follow, we shall also assume that every \voa{} module is of at most countable dimension.  This implies, of course, that the dimension of all \voas{} will also be at most countable.}
\item We will sometimes need broader analogues of the concepts above, wherein the restrictions
	pertain to doubly-homogeneous spaces with respect to Heisenberg zero modes and $L_0$. 
	The actual statements in \cite{HLZ} pertain to such broader situations.  However, the theorems
	in \cite{H-projcov}, that guarantee that \cite{HLZ} may be applied in specific scenarios, assume the definitions that we have
	recalled above. We {expect} 
	that the theorems and concepts in \cite{H-projcov} {may be generalized to the broader setting we require}.
\end{itemize}

Recall the notion \cite[Def.~3.10]{HLZ} of a \emph{(logarithmic) intertwining operator} among a triple of modules.
When the formal variable in a logarithmic intertwining operator is carefully specialized to a fixed $z\in\CC^\times$,
one gets the notion of a \emph{$P(z)$-intertwining map}, \cite[Def.~4.2]{HLZ}. These maps form the backbone of the logarithmic tensor category theory developed in \cite{HLZ}. There, tensor products (fusion products) of modules are defined via certain universal $P(z)$-intertwining maps $\fus{P(z)}$ and the monoidal structure on the module category is obtained by fixing $z\in\CC^\times$, typically chosen to be $z=1$ for convenience.\footnote{We mention that the same notation is generally used to denote both the fusion product operation and the universal $P(z)$-intertwining map corresponding to said fusion product.} 
We remark that the products $\fus{P(z)}$, for different values of $z$, together form a structure richer than that of a braided monoidal category, called \emph{vertex tensor category}.  This richer structure is exploited in the proofs of many important theorems, see \cite{HKL} for some examples.

For convenience, and especially with a view towards the proof of \cref{prop:FusingTensorProducts} below, we give a definition of the fusion product of two modules, equivalent to that of \cite{HLZ}, using intertwining operators instead of intertwining maps.
\begin{defin}\label{def:fusion}
	Given modules $\Mod{W}_1$ and $\Mod{W}_2$, the \emph{fusion product} $\Mod{W}_1\fus{} \Mod{W}_2$
	is the pair $(\Mod{W}_1\fus{}\Mod{W}_2, {\intw^{\fus{}}})$, where $\Mod{W}_1\fus{} \Mod{W}_2$ is a module
	and $\intw^{\fus{}}$ is an intertwining operator of type $\fusco{\Mod{W}_1}{\Mod{W}_2}{\Mod{W}_1\fus{} \Mod{W}_2}$,
	that satisfies the following universal property:
	Given any other ``test module'' $\Mod{W}$ and an intertwining operator
	$\intw$ of type $\fusco{\Mod{W}_1}{\Mod{W}_2}{\Mod{W}}$,
	there exists a \emph{unique} morphism $\eta:\Mod{W}_1{\fus{}} \Mod{W}_2\rightarrow \Mod{W}$ such that
	$\intw=\eta\circ{\intw^{\fus{}}}$.
\end{defin}
\noindent Note that the universal intertwining operator $\intw^{\fus{}}$ will often be clear from the context and hence
we shall often refer to the fusion product by its underlying module.

Now, let $\VOA{V}$ be a \voa{} and let $\categ{C}$ be a category of generalized $\VOA{V}$-modules that satisfies the following properties:
\begin{enumerate}
	\item \label{it:HLZFirst} $\categ{C}$ is a full abelian subcategory of the category of all strongly graded generalized $\VOA{V}$-modules.
	\item $\categ{C}$ is closed under taking contragredient duals 
	and the $P(z)$-tensor product $\fus{P(z)}$ (recall \cite[Def.~4.15]{HLZ}).
	\item $\VOA{V}$ is itself an object of $\categ{C}$.
	\item For each object $\Mod{W}$ of $\categ{C}$, the (generalized) $L_0$-eigenvalues are real and the size of the Jordan blocks of $L_0$ is 
	bounded above (the bound may depend on $\Mod{W}$).
	\item Assumption~12.2 of \cite{HLZ} holds. \label{it:HLZAssumption}
\end{enumerate}
Then, $\categ{C}$ is a vertex tensor category in the sense of Huang-Lepowsky \cite[Thm.~12.15]{HLZ}. In particular, it is an additive braided monoidal category.  A precise formulation of \ref{it:HLZAssumption} may be found in \cite{HLZ}. In essence, this assumption guarantees the convergence of products and iterates of intertwining operators in a specific class of multivalued analytic functions.  It, moreover, guarantees that products of intertwining operators can be written as iterates and vice versa.

\begin{thm}[\protect{\cite{H-projcov}}]\label{thm:huangC1}
	Let $\VOA{V}$ be a \voa{} satisfying the following conditions:
	\begin{itemize}
		\item $\VOA{V}$ is $C_1^{\,\textup{alg.}}$-cofinite, meaning that the space spanned by
		\[\left\{\mathrm{Res}_z\:z^{-1}Y(u,z)v\,\middle\vert\, u,v \in \VOA{V}_{[n]} \text{ with } n>0	\right\}\cup L_{-1}\VOA{V}\]
		has finite codimension in $\VOA{V}$.
		\item There exists a positive integer $N$ 
		that bounds the differences between the real parts of the lowest conformal weights of the simple $\VOA{V}$-modules and is such that the $N$-th Zhu algebra $A_N(\VOA{V})$ (see \cite{DLM-nthzhu}) is finite-dimensional.
		\item Every simple $\VOA{V}$-module is $\RR$-graded and $C_1$-cofinite. \label{it:simplereal}
	\end{itemize}
	Then, the category of grading-restricted generalized modules of $\VOA{V}$ satisfies the conditions \ref{it:HLZFirst}--\ref{it:HLZAssumption} given above,  hence is a vertex tensor category.
\end{thm}

\noindent If $\VOA{V}$ is $C_2$-cofinite, has no states of negative conformal weight, and the space of conformal weight 0 states is spanned by vacuum, then these conditions are satisfied \cite{H-projcov} and so the theory of vertex tensor categories may be applied to the grading-restricted generalized $\VOA{V}$-modules.

As is amply clear from \cref{thm:huangC1}, \cite{M-c1} and \cite[Rem.~12.3]{HLZ}, $C_1$-cofiniteness already takes us a long way towards establishing that a given category of $\VOA{V}$-modules is a vertex tensor category. Our hope is that, in the future, $C_1$-cofiniteness will be, along with other minor conditions (such as conditions on the eigenvalues and Jordan blocks of $L_0$), essentially enough to invoke the theory developed by Huang, Lepowsky and Zhang.
With this hope in mind, we shall prove several useful $C_1$-cofiniteness results in \cref{sec:C1}.

We would also like to remark that there are still many examples of \voas{}, some quite fundamental, which do not meet the known conditions that guarantee the applicability of the vertex tensor theory of \cite{HLZ}.  It is an important problem to analyse the module categories of these examples and bring them ``into the fold'', as it were.  Not only will this make the theory more wide-reaching, but we expect that accommodating these new examples will lead to further crucial insights into the true nature of \voa{} module categories.

\subsection{Vertex tensor categories for the Heisenberg algebra}\label{sec:heisenbergtensor}

For Heisenberg \voas{}, there exist simple modules with non-real
conformal weights and, therefore, one can not invoke \cref{thm:huangC1}.
In this section, we shall deal with general Heisenberg \voas{}, bypassing \cref{thm:huangC1} and
instead relying (mostly) on the results in \cite{DL}.
For related discussions, including self-extensions of 
simple modules (which are not relevant for our purposes), see \cite{M, CMR, Ru}.

We shall verify that a certain semi-simple category $\categ{C}_\RR$
of modules with real conformal weights (see \ref{it:HRCat} below)
is closed under fusion and satisfies the associativity requirements for 
intertwining operators, by invoking results in \cite{DL}.
Once this is done, it is straightforward to verify that $\categ{C}_\RR$
satisfies the assumptions for being vertex tensor category as in \cite[Sec.~12]{HLZ}.

Let $\alg{h}$ be a finite-dimensional abelian Lie algebra over $\CC$, equipped with a
symmetric non-degenerate bilinear form $\langle\cdot,\cdot\rangle$.
We shall identify $\alg{h}$ and its dual $\alg{h}^*$ via this form.
As in \cite[Ch.~6]{LL}, let $\widehat{\alg{h}}$ denote the Heisenberg Lie algebra and
$\VOA{H}$ the corresponding Heisenberg \voa{} (of level $1$, for convenience). 
Given $\alpha\in\alg{h}$, we denote the (simple) Fock module of $\VOA{H}$, with highest weight $\lambda \in \alg{h}$, by $\Mod{F}_{\lambda}$. 
It is known (see \cite{LW2}),
as an algebraic analogue of the Stone-von Neumann theorem, that these simple Fock modules exhaust the
isomorphism classes of the simple $\VOA{H}$-modules.
Let $\categ{C}$ be the semisimple abelian category of $\VOA{H}$ modules generated by these simple $\VOA{H}$-modules and let $\categ{C}_\RR$ be the full subcategory generated by the Fock modules with real highest weights.

\begin{thm}
The subcategory $\categ{C}_\RR$ can be given the structure of a vertex tensor category.
\end{thm}
\begin{proof}
The proof splits into the following steps.  Let $\lambda,\mu,\nu\in \alg{h}=\alg{h}^*$.
\begin{enumerate}
	\item Using \cite[Eq.~(12.10)]{DL}, the fusion coefficient $\fusco{\Mod{F}_\mu}{\Mod{F}_\nu}{\Mod{W}}$ is zero if $\Mod{W}$ does not have $\Mod{F}_{\mu+\nu}$ as a direct summand.
	\item Proceeding exactly as in \cite[Lem.~12.6--Prop.~12.8]{DL}, we see that
	the fusion coefficient $\fusco{\Mod{F}_\mu}{\Mod{F}_\nu}{\Mod{F}_{\mu+\nu}}$ is either $0$ or $1$.
	\item Let $\grp{L}$ be the lattice spanned by $\mu$ and $\nu$.  One can check that the (generalised) lattice \voa{} $\VOA{V}_{\grp{L}}$ satisfies the Jacobi identity
	given in \cite[Thm.~5.1]{DL}, 
	even though $\grp{L}$ is not necessarily 
	rational. 
	This implies that the vertex map $Y$ 
	of $\VOA{V}_\grp{L}$ 
	furnishes explicit (non-zero) intertwining operators
	of type $\fusco{\Mod{F}_\mu}{\Mod{F}_\nu}{\Mod{F}_{\mu+\nu}}$, thereby implying that
	the fusion coefficient 
	$\fusco{\Mod{F}_\mu}{\Mod{F}_\nu}{\Mod{F}_{\mu+\nu}}$ is always $1$.
	\item We conclude that $\categ{C}$ is closed under $\fus{P(z)}$ (recall \cite[Def.~4.15]{HLZ}).
	In general, if $\grp{M}$ is a subgroup of $\alg{h}$, 
	regarded as an additive abelian group,
	and if $\categ{C}'$ is the semi-simple category generated
	by the Fock modules with highest weights in $\grp{M}$, then $\categ{C}'$ is closed under $\fus{P(z)}$.  In particular, the subcategory $\categ{C}_{\RR}$ is closed under $\fus{P(z)}$. 
	\label{it:HCat}\label{it:HRCat}
	\item 
	Given $\mu_1,\dots,\mu_j\in\alg{h}_{\RR}$, let $\grp{L}$ be the lattice that they span. 
	Then, 
	$\VOA{V}_{\grp{L}}$ again satisfies 
	the Jacobi identity \cite[Thm.~5.1]{DL} 
	and the duality results of \cite[Ch.~7]{DL} also go through.
	As a consequence, the expected convergence and associativity properties of intertwining operators among Fock modules in $\categ{C}_{\RR}$ hold.
	\item Since the conformal weights of all modules in $\categ{C}_{\RR}$ are real, the associativity of the intertwining operators
	yields a natural associativity isomorphism for $\categ{C}_{\RR}$ \cite{HLZ}.
	\item Finally, one can proceed as in \cite[Sec.~12]{HLZ} to verify the remaining 
	properties satisfied by the braiding and associativity isomorphisms. Thus, $\categ{C}_{\RR}$ forms a vertex tensor category in the sense of Huang-Lepowsky and, in particular, is a braided tensor category. \qedhere
\end{enumerate}
\end{proof}

\subsection{Simple Currents}

An important concept in the theory of \voas{} is the simple current extension, wherein a given algebra $\VOA{V}$ is embedded in a larger one $\VOA{W}$ that is constructed from certain $\VOA{V}$-modules called simple currents.  The utility of this construction is that, unlike general embeddings, the representation theories of $\VOA{V}$ and $\VOA{W}$ are very closely related.
\begin{defin}
A \emph{simple current} $\Mod{J}$ of a \voa{} $\VOA{V}$ is a $\VOA{V}$-module that possesses a fusion inverse:  $\Mod{J}\fus{} \Mod{J}^{-1}\cong\VOA{V}\cong\Mod{J}^{-1}\fus{}\Mod{J}$.
\end{defin}
\noindent Simple currents and simple current extensions were introduced by Schellekens and Yankielowicz in \cite{SY}.  We note that more general definitions of a simple current exist, see \cite{DLM3} for example, but that the one adopted above will suffice for the \voas{} that we treat below.  Pertinent examples of simple currents are the Heisenberg Fock modules $\Mod{F}_{\lambda}$ discussed in \cref{sec:heisenbergtensor}: the fusion inverse of $\Mod{F}_{\lambda}$ is $\Mod{F}_{-\lambda}$.

The great advantage of requiring invertibility is that each simple current $\Mod{J}$ gives rise to a functor $\Mod{J}\fus{}\blank$ which is an autoequivalence of any $\VOA{V}$-module category that is closed under $\fus{}$.  The following \lcnamecref{prop:simplecurrents} gives some consequences of this; we provide proofs in order to prepare for the similar, but more subtle arguments of the next \lcnamecref{sec:SWDuals}.  We remark that the isomorphism classes of the simple currents naturally form a group, sometimes called the Picard group of the category.
\begin{prop}\label{prop:simplecurrents}
Let $\Mod{J}$ be a simple current of a 
\voa{} $\VOA{V}$.
\begin{enumerate}
\item If $\Mod{M}$ is a non-zero $\VOA{V}$-module, then $\Mod{J}\fus{}\Mod{M}$ is non-zero. \label{it:SCNonZero}
\item If $\Mod{M}$ is an indecomposable $\VOA{V}$-module, then $\Mod{J}\fus{}\Mod{M}$ is indecomposable. \label{it:SCIndec}
\item If $\Mod{M}$ is a simple $\VOA{V}$-module, then $\Mod{J}\fus{}\Mod{M}$ is simple.  In particular, $\Mod{J}$ is simple if $\VOA{V}$ is. \label{it:SCSimple}
\item The covariant functor $\Mod{J}\fus{}\blank$ is exact (hence, so is $\blank\fus{}\Mod{J}$). \label{it:SCExact} 
\item If $\Mod{M}$ has a composition series with composition factors $\Mod{S}_i$, $1 \le i \le n$, then $\Mod{J}\fus{}\Mod{M}$ has a composition series with composition factors $\Mod{J}\fus{}\Mod{S}_i$, $1 \le i \le n$. \label{it:SCCompSeries}
\item If $\Mod{M}$ has a radical or socle, then so does $\Mod{J}\fus{}\Mod{M}$.  Moreover, the latter radical or socle is then given by $\Mod{J}\fus{}\rad\Mod{M} \cong \rad(\Mod{J}\fus{}\Mod{M})$ or $\Mod{J}\fus{}\soc\Mod{M} \cong \soc(\Mod{J}\fus{}\Mod{M})$.\label{it:SCRadSoc}
\item If $\Mod{M}$ has a 
radical or socle series, then so does $\Mod{J}\fus{}\Mod{M}$.  In particular, the corresponding Loewy diagrams of $\Mod{J} \fus{} \Mod{M}$ are obtained by replacing each composition factor $\Mod{S}_i$ of $\Mod{M}$ by $\Mod{J}\fus{}\Mod{S}_i$. \label{it:SCLoewy}
\end{enumerate}
\end{prop}
\begin{proof}
If $\Mod{J} \fus{} \Mod{M} = 0$, then $0 = \Mod{J}^{-1} \fus{} \Mod{J} \fus{} \Mod{M} \cong \VOA{V} \fus{} \Mod{M} \cong \Mod{M}$.  Thus, \ref{it:SCNonZero} follows:
\begin{equation} \label{eq:SCInj}
\Mod{M} \neq 0 \qquad \Ra \qquad \Mod{J} \fus{} \Mod{M} \neq 0.
\end{equation}
Similarly, if $\Mod{J} \fus{} \Mod{M} \cong \Mod{M}' \oplus \Mod{M}''$, then $\Mod{M} \cong \Mod{J}^{-1} \fus{} \Mod{J} \fus{} \Mod{M} \cong (\Mod{J}^{-1} \fus{} \Mod{M}') \oplus (\Mod{J}^{-1} \fus{} \Mod{M}'')$.  In other words, $\Mod{M}$ indecomposable implies that $\Mod{J} \fus{} \Mod{M}$ is indecomposable, {which is \ref{it:SCIndec}}.

Suppose now that $\Mod{M}$ is simple, but that $\Mod{J} \fus{} \Mod{M}$ has a proper submodule $\Mod{M}'$.  Then,
\begin{equation}
\dses{\Mod{M}'}{\Mod{J} \fus{} \Mod{M}}{\Mod{M}''}
\end{equation}
is exact, for $\Mod{M}'' \cong (\Mod{J} \fus{} \Mod{M}) / \Mod{M}' \neq 0$.  But, fusion is right-exact \cite[Prop.~4.26]{HLZ}, so
\begin{equation}
\drses{\Mod{J}^{-1} \fus{} \Mod{M}'}{\Mod{M}}{\Mod{J}^{-1} \fus{} \Mod{M}''}
\end{equation}
is exact.  However, $\Mod{M}'' \neq 0$ implies that $\Mod{J}^{-1} \fus{} \Mod{M}''$ is a non-zero quotient of $\Mod{M}$, by \ref{it:SCNonZero}, so we must have $\Mod{J}^{-1} \fus{} \Mod{M}'' \cong \Mod{M}$, as $\Mod{M}$ is simple.  Fusing with $\Mod{J}$ now gives $\Mod{J} \fus{} \Mod{M} \cong \Mod{M}''$, so we conclude that $\Mod{M}' = 0$ and that $\Mod{J} \fus{} \Mod{M}$ is simple.  The simplicity of $\Mod{J} \cong \Mod{J}\fus{} \VOA{V}$ now follows from that of $\VOA{V}$, {completing the proof of \ref{it:SCSimple}}.

To prove \ref{it:SCExact}, 
note that applying 
right-exactness to the short exact sequence $\ses{\Mod{M}'}{\Mod{M}}{\Mod{M}''}$ results in
\begin{equation} \label{eq:ExactKey}
\Mod{J} \fus{} \frac{\Mod{M}}{\Mod{M'}} \cong \frac{\Mod{J} \fus{} \Mod{M}}{(\Mod{J} \fus{} \Mod{M}') / \ker f},
\end{equation}
where $f$ is the induced map from $\Mod{J} \fus{} \Mod{M'}$ to $\Mod{J} \fus{} \Mod{M}$ that might not be an inclusion.  Fusing with $\Mod{J}^{-1}$ and applying \eqref{eq:ExactKey}, we arrive at
\begin{equation}
\frac{\Mod{M}}{\Mod{M}'} \cong \Mod{J}^{-1} \fus{} \frac{\Mod{J} \fus{} \Mod{M}}{(\Mod{J} \fus{} \Mod{M}') / \ker f} \cong \frac{\Mod{M}}{\Bigl( \Mod{J}^{-1} \fus{} \dfrac{\Mod{J} \fus{} \Mod{M}'}{\ker f} \Bigr) / \ker g},
\end{equation}
where $g \colon \Mod{J}^{-1} \fus{} \bigl( (\Mod{J} \fus{} \Mod{M}') / \ker f \bigr) \to \Mod{M}$ might not be an inclusion.  Thus,
\begin{equation}
\Mod{M}' \cong \frac{\Mod{J}^{-1} \fus{} \dfrac{\Mod{J} \fus{} \Mod{M}'}{\ker f}}{\ker g} \cong \frac{\dfrac{\Mod{M}'}{(\Mod{J}^{-1} \fus{} \ker{f}) / \ker h}}{\ker g},
\end{equation}
where $h \colon \Mod{J}^{-1} \fus{} \ker f \to \Mod{M}'$ might not be an inclusion.  We conclude that $\ker g = 0$ and $\ker h = \Mod{J}^{-1} \fus{} \ker{f}$.  But, both require that
\begin{equation}
\Mod{M}' \cong \Mod{J}^{-1} \fus{} \frac{\Mod{J} \fus{} \Mod{M}'}{\ker f} \qquad \Ra \qquad \Mod{J} \fus{} \Mod{M}' \cong \frac{\Mod{J} \fus{} \Mod{M}'}{\ker f} \qquad \Ra \qquad \ker f = 0.
\end{equation}
$f \colon \Mod{J} \fus{} \Mod{M'} \to \Mod{J} \fus{} \Mod{M}$ is therefore an inclusion, hence $\Mod{J} \fus{} \blank$ is exact.

Suppose now that $0 = \Mod{M}_0 \subset \Mod{M}_1 \subset \cdots \subset \Mod{M}_{n-1} \subset \Mod{M}_n = \Mod{M}$ is a composition series for $\Mod{M}$, so that each \mbox{$\Mod{S}_i = \Mod{M}_i/\Mod{M}_{i-1}$} is simple.  By \ref{it:SCExact}, applying $\Mod{J}\fus{}\blank$ to each exact sequence $\ses{\Mod{M}_{i-1}}{\Mod{M}_i}{\Mod{S}_i}$ gives another exact sequence $\ses{\Mod{J}\fus{}\Mod{M}_{i-1}}{\Mod{J}\fus{}\Mod{M}_i}{\Mod{J}\fus{}\Mod{S}_i}$.  Moreover, $\Mod{J}\fus{}\Mod{S}_i$ is simple, by \ref{it:SCSimple}.  Assembling all of these exact sequences gives \ref{it:SCCompSeries}.

For \ref{it:SCRadSoc}, first recall that $\rad \Mod{M}$ is the intersection of the maximal proper submodules of $\Mod{M}$ and that $\Mod{M}_i\subset\Mod{M}$ is maximal proper if and only if $\Mod{M}/\Mod{M}_i$ is simple.  In this case, \ref{it:SCSimple} and \ref{it:SCExact} now imply that $\Mod{J}\fus{}(\Mod{M}/\Mod{M}_i)$ is simple and isomorphic to $(\Mod{J}\fus{}\Mod{M})/(\Mod{J}\fus{}\Mod{M}_i)$, whence $\Mod{J}\fus{}\Mod{M}_i$ is maximal proper in $\Mod{J}\fus{}\Mod{M}$.  Applying $\Mod{J}^{-1}\fus{}\blank$ gives the converse.  Second, given a collection $\Mod{M}_i\subseteq\Mod{M}$, \ref{it:SCExact} also implies that $\Mod{J}\fus{}(\cap_i\Mod{M}_i)$ is a submodule of each $\Mod{J}\fus{}\Mod{M}_i$, hence of $\cap_i(\Mod{J}\fus{}\Mod{M}_i)$.  But now, $\cap_i(\Mod{J}\fus{}\Mod{M}_i)\cong \Mod{J}\fus{}\Mod{J}^{-1}\fus{}\bigl(\cap_i(\Mod{J}\fus{}\Mod{M}_i)\bigr)\subseteq\Mod{J}\fus{}(\cap_i \Mod{M}_i)$, hence we have $\Mod{J}\fus{}(\cap_i\Mod{M}_i) \cong \cap_i(\Mod{J}\fus{}\Mod{M}_i)$.  These two conclusions together give $\Mod{J}\fus{}\rad\Mod{M} \cong \rad(\Mod{J}\fus{}\Mod{M})$.  A similar, but easier, argument establishes $\Mod{J}\fus{}\soc\Mod{M} \cong \soc(\Mod{J}\fus{}\Mod{M})$.

Finally, \ref{it:SCLoewy} follows by combining \ref{it:SCRadSoc} with slight generalisations of the arguments used to prove \ref{it:SCCompSeries}.
\end{proof}
\noindent This proposition has a simple summary:  fusing with a simple current preserves module structure.  We remark, obviously, that a simple current $\Mod{J}$ need not be simple if the \voa{} $\VOA{V}$ is not simple.

\subsection{Orbifold modules} \label{sec:KR}

Here, we review a result of Kac and Radul \cite{KR} on the simplicity of orbifold modules. For a very similar result see \cite{DLM}.

Let $\VOA{A}$ be an associative algebra, for example the mode algebra of a \voa{}, and let $\grp{G}$ be a subgroup of $\Aut \VOA{A}$ acting semisimply on $\VOA{A}$.  We consider $\VOA{A}$-modules $\Mod{M}$ which admit a semisimple $\grp{G}$-action that is compatible with the $\grp{G}$-action on $\VOA{A}$ and which decompose as a countable direct sum of finite-dimensional simple $\grp{G}$-modules.  This compatibility means that
\begin{equation} \label{eq:CompatGrpAct}
g(am) = (ga)(gm), \qquad \text{for all \(g \in \grp{G}\), \(a \in \VOA{V}\) and \(m \in \Mod{M}\).}
\end{equation}
If we now define $\VOA{A}_0$ to be the space of $\grp{G}$-invariants $a \in \VOA{A}$, so $ga=a$ for all $g \in \grp{G}$, then the actions of each $g \in \grp{G}$ and each $a \in \VOA{A}_0$ commute on every such module $\Mod{M}$.

Choose an $\Mod{M}$ satisfying \eqref{eq:CompatGrpAct} and let $\Mod{N}$ be a simple $\grp{G}$-module.  Then, we may define the $\grp{G}$-module
\begin{equation}
\Mod{M}_{\Mod{N}} = \sum \set{\Mod{N}_i \subseteq \Mod{M} \st \Mod{N}_i \cong \Mod{N}}.
\end{equation}
As the action of $\VOA{A}_0$ commutes with that of $\grp{G}$, every $a \in \VOA{A}_0$ maps a given $\Mod{N}_i$ to some $\Mod{N}_j$ or $0$, by Schur's lemma.  Thus, $\Mod{M}_{\Mod{N}}$ is an $\VOA{A}_0$-module.

If we choose a one-dimensional subspace $\CC \subseteq \Mod{N}$, then Schur's lemma picks out a one-dimensional subspace $\CC_i \subseteq \Mod{N}_i$, for each i.  Then, each $a \in \VOA{A}_0$ maps each $\CC_i$ to some $\CC_j$ or to $0$, hence
\begin{equation}
\Mod{M}^{\Mod{N}} = \sum_{\Mod{N}_i \cong \Mod{N}} \CC_i
\end{equation}
is an $\VOA{A}_0$-module.  But, because $\Mod{N}_i \cong \Mod{N} \cong \Mod{N} \otimes \CC_i$, we may write
\begin{equation}
\Mod{M}_{\Mod{N}} \cong \sum_{\Mod{N}_i \cong \Mod{N}} \Mod{N} \otimes \CC_i = \Mod{N} \otimes \Mod{M}^{\Mod{N}}
\end{equation}
as a $\CC \grp{G} \otimes \VOA{A}_0$-module.  The semisimplicity of $\Mod{M}$, as a $\grp{G}$-module, now gives us the decomposition
\begin{equation} \label{eq:KR}
\Mod{M} \cong \bigoplus_{[\Mod{N}]} \Mod{M}_{\Mod{N}} \cong \bigoplus_{[\Mod{N}]} \Mod{N} \otimes \Mod{M}^{\Mod{N}},
\end{equation}
again as a $\CC \grp{G} \otimes \VOA{A}_0$-module.  Here, $[\Mod{N}]$ denotes the isomorphism class of the simple $\grp{G}$-module $\Mod{N}$.

The result of Kac and Radul gives conditions under which the $\VOA{A}_0$-modules $\Mod{M}^{\Mod{N}}$, appearing in \eqref{eq:KR}, are guaranteed to be simple.
\begin{thm}[\protect{\cite[Thm.~1.1 and Rem.~1.1]{KR}}] \label{thm:KR}
With the above setup, the (non-zero) $\Mod{M}^{\Mod{N}}$ appearing in \eqref{eq:KR} will be simple $\VOA{A}_0$-modules provided that $\Mod{M}$ is a simple $\VOA{A}$-module.
\end{thm}

\section{Schur-Weyl duality} \label{sec:SWDuals}

In this section, we state and prove results concerning the decomposition of a \voa{} and its modules into modules over a Heisenberg \vosa{} and its commutant.  We regard this decomposition as a vertex-algebraic analogue of the well known Schur-Weyl duality familiar for symmetric groups and general linear Lie algebras.  These results are enhanced by deducing sufficient conditions for the decompositions, and their close relations, to be multiplicity-free.  Finally, we illustrate our results with several carefully chosen examples.

\subsection{Heisenberg cosets} \label{sec:HCosets}

Let $\grp{G}$ be a {finitely generated} abelian subgroup of the automorphism group of a simple \voa{} $\VOA{V}$.  We assume that $\grp{G}$ grades $\VOA{V}$, meaning that the actions of these automorphisms may be simultaneously diagonalised, hence that $\VOA{V}$ decomposes into a direct sum of $\grp{G}$-modules:
\begin{equation} \label{eq:Vdecomp}
\VOA{V}=\bigoplus_{\lambda \in \grp{L}} \Mod{V}_{\lambda}.
\end{equation}
Here, the $\lambda$ are elements of the (abelian) dual group $\hat{\grp{G}}$ of inequivalent {(complex, not necessarily unitary)} one-dimensional representations of $\grp{G}$ (recall that addition is tensor product and negation is contragredient dual), $\Mod{V}_{\lambda}$ denotes the simultaneous eigenspace upon which each $g \in \grp{G}$ acts as multiplication by $\lambda(g) \in \CC$, and $\grp{L}$ is the subset of $\lambda \in \hat{\grp{G}}$ for which $\Mod{V}_{\lambda} \neq 0$.  Note that the cardinality of $\grp{L}$ is at most countable.

The action of $\VOA{V}$ on itself restricts to an action of each $\Mod{V}_{\lambda}$ on each $\Mod{V}_{\mu}$.  For $\lambda = \mu = 0$, where $0$ denotes the trivial $\grp{G}$-module, this implies that $\VOA{V}_0$ is a \vosa{} of $\VOA{V}$; for $\lambda = 0$, this implies that each $\Mod{V}_{\mu}$ is a $\VOA{V}_0$-module.  From the simplicity of $\VOA{V}$, it now easily follows that $\grp{L}$ is a subgroup of $\hat{\grp{G}}$:  closure under addition follows from annihilating ideals being trivial \cite[Cor.~4.5.15]{LL} and closure under negation follows similarly, see \cite[Prop.~3.6]{LX}.

Applying \cref{thm:KR}, with $\Mod{M} = \VOA{V}$ and $\VOA{A}$ being the mode algebra of $\VOA{V}$, we can now improve upon \eqref{eq:Vdecomp}.  Indeed, in this setting, \eqref{eq:KR} becomes
\begin{equation}
\VOA{V}=\bigoplus_{\lambda \in \grp{L}} \CC_{\lambda} \otimes \Mod{V}_{\lambda},
\end{equation}
where $\CC_{\lambda}$ denotes the one-dimensional module upon which $g \in \grp{G}$ acts as multiplication by $\lambda(g)$, and we learn that the $\Mod{V}_{\lambda}$ are simple as $\VOA{V}_0$-modules.  In particular, $\VOA{V}_0$ is a simple \voa{}.

If we assume that $\VOA{V}_0$ satisfies the conditions required to invoke the tensor category theory of Huang, Lepowsky and Zhang (\cref{sec:HLZ}),
then more is true.  As 
Miyamoto has shown, 
the $\Mod{V}_{\lambda}$ are then simple currents for $\VOA{V}_0$
see \cite{M2, CarMi}.  
It should be noted that the proofs in \cite{M2, CarMi} assume that the group of automorphisms under consideration
	is finite, however, the proof works more generally under the assumption that tensor category
	theory for the fixed-point algebra can be invoked. For completeness, we include an exposition of their proof in our slightly more general setting in \cref{app:Miyamoto}.
	
\begin{thm}[\protect{\cite[Sec.~6]{M2}}] \label{thm:MiySimpleCurrent}
{Assume the above setup and that} 
$\VOA{V}_0 = \VOA{V}^\grp{G}$ 
satisfies conditions sufficient to invoke Huang, Lepowsky and Zhang's tensor category theory, for example those of \cref{thm:huangC1}. 
Then, the
$\Mod{V}_{\lambda}$ are simple currents for $\Mod{V}_0$ {with} 
$\Mod{V}_{\lambda} \fus{\Mod{V}_0} \Mod{V}_{\mu} \cong \Mod{V}_{\lambda + \mu}$, {for all $\lambda, \mu \in \grp{L}$}.
\end{thm}

Let us now restrict to \voas{} $\VOA{V}$ that contain a Heisenberg \vosa{} $\VOA{H}$, generated by $r$ fields $h^i(z)$, $i=1, \ldots, r$, of conformal weight $1$.  We will assume throughout that the action of $\VOA{H}$ on $\VOA{V}$ is semisimple\footnote{Examples on which a Heisenberg \vosa{} does not act semisimply are provided by the Takiff \voas{} of \cite{BR, BC}.} {and that the eigenvalues of the zero modes $h^i_0$, $i=1, \ldots, r$, are all real}.  Let $\VOA{C}$ denote the commutant \voa{} of $\VOA{H}$ in $\VOA{V}$ and let $\grp{G} \cong \ZZ^r$ be the {lattice} 
generated by the 
$h^i_0$.  
Each $\Mod{V}_{\lambda}$ of the $\grp{G}$-decomposition \eqref{eq:Vdecomp} is a module for $\VOA{H}$ since the fields of $\VOA{H}$ commute with the zero modes of $\grp{G}$. As $\grp{G}$ acts semisimply on $\Mod{V}_{\lambda}$ and the only simple $\VOA{H}$-module with $h^i_0$-eigenvalues $\lambda = (\lambda^1, \ldots, \lambda^r)$ is the Fock module $\Mod{F}_\lambda$, we must have the following $\VOA{H}\otimes \VOA{C}$-module decomposition:
\begin{equation}
\Mod{V}_\lambda = \Mod{F}_\lambda \otimes \Mod{C}_\lambda, \qquad \text{for all \(\lambda \in \grp{L}\).}
\end{equation}
{In this setting, we may take $\grp{L}$ to be the lattice of all $\lambda \in \RR^r$ for which $\Mod{V}_{\lambda} \neq 0$.}  Moreover, the $\VOA{C}$-module $\Mod{C}_{\lambda}$ is simple because $\Mod{V}_\lambda$ and $\Mod{F}_\lambda$ are.  In particular, the commutant $\VOA{C} = \Mod{C}_0$ is a simple \voa{}.
We summarise this as follows.
\begin{prop}
Let $\VOA{V}$ be a simple \voa{} 
with a Heisenberg \vosa{} $\VOA{H}$ that acts semisimply on $\VOA{V}$.  Then, the coset \voa{} $\VOA{C} = \Com{\VOA{H}}{\VOA{V}}$ is likewise simple.
\end{prop}

{From here on, we make the following natural assumption:}
\begin{quote}
We assume that we are working with categories of (generalized) $\VOA{V}_0$- 
and $\VOA{C}$-modules {for which the tensor category theory of Huang, Lepowsky and Zhang \cite{HLZ} may be invoked}.
\end{quote}
{Of course, we have confirmed in \cref{sec:heisenbergtensor} that this theory may be invoked for semisimple $\VOA{H}$-modules with real weights.  In general, we would like to apply our results to \voas{} for which we are not currently able to verify this assumption.  Such illustrations should therefore be regarded as conjectural.  However, we view the results in these cases as strong evidence that the conditions required to invoke Huang-Lepowsky-Zhang are, in fact, significantly weaker than those that were given in \cref{sec:HLZ}.}

Given now the fusion rules $\Mod{F}_{\lambda} \fus{\VOA{H}} \Mod{F}_{\mu} \cong \Mod{F}_{\lambda + \mu}$ and $\Mod{V}_{\lambda} \fus{\Mod{V}_0} \Mod{V}_{\mu} \cong \Mod{V}_{\lambda + \mu}$, which imply that
\begin{equation} \label{eq:FCxFC}
\bigl( \Mod{F}_\lambda \otimes \Mod{C}_\lambda \bigr) \fus{\VOA{V}_0} \bigl( \Mod{F}_\mu \otimes \Mod{C}_\mu \bigr)
\cong \Mod{F}_{\lambda + \mu} \otimes \Mod{C}_{\lambda + \mu},
\end{equation}
one is naturally led to suppose that $\Mod{C}_\lambda \fus{\VOA{C}} \Mod{C}_{\mu} \cong \Mod{C}_{\lambda + \mu}$.  Proving this, however, is a little subtle because we are not assuming that the corresponding representation categories are semisimple.
We therefore 
present a 
technical result that we shall use to confirm this supposition and other similar assertions. 
We remark that this result can be greatly strengthened when one of the \voas{} involved is of 
Heisenberg or 
lattice type, 
or when the \voas{} involved are rational (see \cite{Lin}).
\begin{prop} \label{prop:FusingTensorProducts}
	Let $\VOA{A}$ and $\VOA{B}$ be \voas{} and let $\Mod{A}_i$ and $\Mod{B}_i$, for $i = 1,2,3$, be 
	$\VOA{A}$-modules and $\VOA{B}$-modules, respectively. 
	Suppose that 
  \begin{align}
		\bigl( (\Mod{A}_1\otimes \Mod{B}_1) \fus{\VOA{A} \otimes \VOA{B}} (\Mod{A}_2\otimes \Mod{B}_2), \intw^{\fus{}}_{\VOA{A} \otimes \VOA{B}} \bigr) = (\Mod{A}_3\otimes \Mod{B}_3,\intw^{\fus{}}_{\VOA{A} \otimes \VOA{B}}).
	\end{align}
	Also assume that either of the fusion coefficients $\binom{\Mod{A}_3}{\Mod{A}_1\,\,\Mod{A}_2}$ or $\binom{\Mod{B}_3}{\Mod{B}_1\,\,\Mod{B}_2}$ is finite.
	Then, $(\Mod{A}_3\otimes \Mod{B}_3,\intw^{\fus{}}_{\VOA{A} \otimes \VOA{B}})$ may be taken to be
	$\bigl( (\Mod{A}_1 \fus{\VOA{A}} \Mod{A}_2) \otimes (\Mod{B}_1 \fus{\VOA{B}} \Mod{B}_2) ,\intw^{\fus{}}_{\VOA{A}} \otimes \intw^{\fus{}}_{\VOA{B}} \bigr)$.  	In particular,
	\begin{equation}
		\Mod{A}_1 \fus{\VOA{A}} \Mod{A}_2 \cong \Mod{A}_3 \qquad \text{and} \qquad \Mod{B}_1 \fus{\VOA{B}} \Mod{B}_2 \cong \Mod{B}_3.
	\end{equation}
\end{prop}
\begin{proof}
	The key here is \cite[Thm.~2.10]{ADL} which, as stated, applies to non-logarithmic intertwining operators but in fact also holds in when logarithmic intertwiners are present.  Using this, we {may write} 
	\begin{align}
		\intw^{\fus{}}_{\VOA{A} \otimes \VOA{B}}=\sum_{j=1}^N \tilde{\intw}_{\VOA{A}}^{(j)}\otimes\tilde{\intw}_{\VOA{B}}^{(j)},
	\end{align}
	for some $N$, where each $\tilde{\intw}_{\VOA{A}}^{(j)}$ is an intertwiner for $\VOA{A}$ of type $\fusco{\Mod{A}_1}{\Mod{A}_2}{\Mod{A}_3}$ and each $\tilde{\intw}_{\VOA{B}}^{(j)}$ is of type $\fusco{\Mod{B}_1}{\Mod{B}_2}{\Mod{B}_3}$ for $\VOA{B}$. The universality of the fusion product now guarantees  $\VOA{A}$-modules,
	the existence of (unique) $\VOA{A}$-module morphisms	$\mu_\VOA{A}^{(j)} \colon \Mod{A}_1 \fus{\VOA{A}} \Mod{A}_2 \rightarrow \Mod{A}_3$,
	such that $\mu_\VOA{A}^{(j)} \circ \intw^{\fus{}}_{\VOA{A}} = \tilde{\intw}_{\VOA{A}}^{(j)}$, and
	$\VOA{B}$-module morphisms	$\mu_\VOA{B}^{(j)} \colon \Mod{B}_1 \fus{\VOA{B}} \Mod{B}_2 \rightarrow \Mod{B}_3$,
	such that $\mu_\VOA{B}^{(j)} \circ \intw^{\fus{}}_{\VOA{B}} = \tilde{\intw}_\VOA{B}^{(j)}$.
	Setting $\mu = \sum_{j=1}^N \mu_\VOA{A}^{(j)} \otimes \mu_\VOA{B}^{(j)}$, we obtain
	\begin{align}
	\mu \circ \left( \intw^{\fus{}}_\VOA{A} \otimes \intw^{\fus{}}_{\VOA{B}} \right)
	= \sum_{j=1}^N \left( \mu_\VOA{A}^{(j)} \otimes \mu_\VOA{B}^{(j)} \right) \circ \left( \intw^{\fus{}}_\VOA{A} \otimes \intw^{\fus{}}_{\VOA{B}} \right)
	= \sum_{j=1}^N \tilde{\intw}_{\VOA{A}}^{(j)} \otimes \tilde{\intw}_{\VOA{B}}^{(j)} = \intw^{\fus{}}_{\VOA{A} \otimes \VOA{B}}.
	\end{align}

	Now, let $\Mod{X}$ be a ``test'' $\VOA{A}\otimes \VOA{B}$-module and let $\intw$ be 
	an intertwining operator of type $\fusco{\Mod{A}_1\otimes \Mod{B}_1}{\Mod{A}_2\otimes \Mod{B}_2}{\Mod{X}}$.
	By the universal property satisfied by $(\Mod{A}_3\otimes \Mod{B}_3, \intw^{\fus{}}_{\VOA{A} \otimes \VOA{B}})$, there exists a (unique) $\eta \colon \Mod{A}_3\otimes \Mod{B}_3\rightarrow \Mod{X}$ such that $\eta\circ \intw^{\fus{}}_{\VOA{A} \otimes \VOA{B}}=\intw$.  It follows 
	that
	\begin{align} \label{eqn:univ}
	(\eta \circ \mu) \circ \left( \intw^{\fus{}}_{\VOA{A}} \otimes \intw^{\fus{}}_{\VOA{B}} \right)
	= \eta \circ \intw^{\fus{}}_{\VOA{A} \otimes \VOA{B}} = \intw.
	\end{align}

	It remains 
	to prove that $\eta\circ\mu \colon (\Mod{A}_1 \fus{\VOA{A}} \Mod{A}_2) \otimes (\Mod{B}_1 \fus{\VOA{B}} \Mod{B}_2) \rightarrow \Mod{X}$
	is the 
	unique $\VOA{A} \otimes \VOA{B}$-module morphism satisfying \eqref{eqn:univ}. 
	However, 
	$\intw^{\fus{}}_{\VOA{A}}$ and $\intw^{\fus{}}_{\VOA{B}}$ are surjective intertwining operators --- this surjectivity goes hand-in-hand with the ``uniqueness'' requirement in the universal property, see \cite[Prop.~4.23]{HLZ} --- and so, therefore,	
	is $\intw^{\fus{}}_{\VOA{A}} \otimes \intw^{\fus{}}_{\VOA{B}}$.
	This means that equation \eqref{eqn:univ} uniquely specifies the morphism $\eta\circ\mu$, completing the proof.
\end{proof}

From this \lcnamecref{prop:FusingTensorProducts}, we immediately obtain the following \lcnamecref{cor:Intriguing}.
\begin{cor} \label{cor:Intriguing}
	If $\VOA{A}$ and $\VOA{B}$ are \emph{simple} \voas{} and 
	$\Mod{M}\otimes \Mod{N}$ is a simple current for $\VOA{A}\otimes \VOA{B}$,
	then 
	$\Mod{M}$ and $\Mod{N}$ are simple currents for $\VOA{A}$ and $\VOA{B}$, respectively.
	Moreover, the inverse of $\Mod{M}\otimes \Mod{N}$ is $\Mod{M}^{-1}\otimes \Mod{N}^{-1}$.
\end{cor}
\begin{proof}
	Because $\VOA{A} \otimes \VOA{B}$ is assumed to be simple, $\Mod{M}\otimes \Mod{N}$ and its inverse are simple $\VOA{A}\otimes \VOA{B}$-modules, by \cref{prop:simplecurrents}\ref{it:SCSimple}.  Moreover, this simplicity hypothesis also guarantees that 
	the inverse has the form 
	$\tilde{\Mod{M}}\otimes\tilde{\Mod{N}}$ 
	\cite[Thm.~4.7.4]{FHL}. Applying \cref{prop:FusingTensorProducts} to 
	$(\tilde{\Mod{M}}\otimes\tilde{\Mod{N}})\fus{\VOA{A} \otimes \VOA{B}}(\Mod{M}\otimes \Mod{N})\cong \Mod{A}\otimes \Mod{B}$, 
	we obtain
	$\tilde{\Mod{M}}\fus{\VOA{A}}\Mod{M}\cong\VOA{A}$ and $\tilde{\Mod{N}}\fus{\VOA{B}}\Mod{N}\cong\VOA{B}$, hence $\tilde{\Mod{M}} \cong \Mod{M}^{-1}$ and $\tilde{\Mod{N}} \cong \Mod{N}^{-1}$.
\end{proof}

In any case, \eqref{eq:FCxFC} and \cref{prop:FusingTensorProducts} give the desired conclusion:
\begin{equation} \label{eq:CxC}
	\Mod{C}_\lambda \fus{\VOA{C}} \Mod{C}_{\mu} \cong \Mod{C}_{\lambda + \mu}.
\end{equation}
In particular, the $\Mod{C}_{\lambda}$ are simple currents for all $\lambda \in \grp{L}$.
We have therefore arrived at the following decomposition of $\VOA{V}$ into simple currents of $\VOA{H}$ and $\VOA{C}$:
\begin{equation} \label{eq:DecompVFC}
\VOA{V} = \bigoplus_{\lambda \in \grp{L}} \Mod{F}_\lambda \otimes \Mod{C}_\lambda.
\end{equation}
However, this may be further refined if $\lambda \neq \mu$ in $\grp{L}$ does not imply that $\Mod{C}_{\lambda} \neq \Mod{C}_{\mu}$ (this implication is obviously true for Fock modules).  Suppose that $\Mod{C}_\lambda = \Mod{C}_{\lambda +\mu}$ for some $\lambda, \mu \in \grp{L}$.  Then, we must have $\Mod{C}_\mu=\VOA{C}$ and hence $\Mod{C}_{n \mu}=\VOA{C}$ for all $n \in \ZZ$.  More generally, let $\grp{N}$ denote the sublattice of $\mu \in \grp{L}$ for which $\Mod{C}_{\mu} = \VOA{C}$.  Then, we may define
\begin{equation}
\Mod{W}_{[\lambda]} = \bigoplus_{\mu \in \grp{N}} \Mod{F}_{\lambda + \mu}
\end{equation}
and note that $\VOA{W} = \Mod{W}_{[0]}$ will be a lattice \voa{} if the conformal weights of the fields of $\Mod{F}_{\mu}$, with $\mu \in \grp{N}$, are all integers.\footnote{If the conformal weights are not all integers, then $\VOA{W}$ is a vertex operator superalgebra, or another type of generalised \voa{}.  This does not significantly affect the following analysis.}  The decomposition \eqref{eq:DecompVFC} then becomes a decomposition as a $\VOA{W} \otimes \VOA{C}$-module:
\begin{equation} \label{eq:DecompVWC}
\VOA{V}= \bigoplus_{[\lambda] \in \grp{L} / \grp{N}} \Mod{W}_{[\lambda]} \otimes \Mod{C}_{[\lambda]}.
\end{equation}
Now the $\Mod{C}_{[\lambda]} \equiv \Mod{C}_{\lambda}$, with $[\lambda] \in \grp{L} / \grp{N}$, are mutually inequivalent:  $[\lambda] \neq [\mu]$ implies that $\Mod{C}_{[\lambda]} \ncong \Mod{C}_{[\mu]}$.  We remark that $\grp{L} / \grp{N}$ may still be infinite because the rank of $\grp{N}$ may be smaller than that of $\grp{L}$.

We summarise these results as follows.
\begin{thm} \label{thm:SW-VOA}
Let:
\begin{itemize}
\item $\VOA{V}$ be a simple \voa{}. 
\item $\VOA{H} \subseteq \VOA{V}$ be a Heisenberg \vosa{} that acts semisimply on $\VOA{V}$.
\item $\VOA{C} = \Mod{C}_0$ be the commutant of $\VOA{H}$ in $\VOA{V}$.
\item $\grp{L}$ be the {lattice} of Heisenberg weights of $\VOA{V}$ ($\VOA{V}$ being regarded as an $\VOA{H}$-module). 
\end{itemize}
Then the decompositions \eqref{eq:DecompVFC} and \eqref{eq:DecompVWC} hold, where:
\begin{itemize}
\item The $\Mod{C}_{\lambda}$, $\lambda \in \grp{L}$, are simple currents for $\VOA{C}$ whose fusion products
include
$\Mod{C}_{\lambda} \fus{\VOA{C}} \Mod{C}_{\mu} = \Mod{C}_{\lambda + \mu}$.
\item $\VOA{W} = \bigoplus_{\lambda \in \grp{N}} \Mod{F}_{\lambda}$ is a simple current extension of $\VOA{H}$ ($\grp{N}$ is the sublattice of $\lambda \in \grp{L}$ for which $\Mod{C}_{\lambda} \cong \Mod{C}$).
\item The $\Mod{W}_{[\lambda]}$, $[\lambda] \in \grp{L} / \grp{N}$, are simple currents for $\VOA{W}$ with fusion products
$\Mod{W}_{[\lambda]} \fus{\VOA{W}} \Mod{W}_{[\mu]} = \Mod{W}_{[\lambda + \mu]}$.
\end{itemize}
In particular, the $\Mod{C}_{[\lambda]}$, $[\lambda] \in \grp{L} / \grp{N}$, of \eqref{eq:DecompVWC} are mutually non-isomorphic.
\end{thm}
\begin{remark}
Note that we may instead choose $\grp{N}$ to be any subgroup of $\grp{L}$ in which every $\lambda \in \grp{N}$ satisfies $\Mod{C}_{\lambda} \cong \Mod{C}$.  In particular, we may take $\grp{N} = 0$, in which case the decomposition \eqref{eq:DecompVWC} reduces to that of \eqref{eq:DecompVFC}.  {Obviously, the conclusion that the $\Mod{C}_{[\lambda]}$ are mutually non-isomorphic will only hold if $\grp{N}$ is taken to be maximal.}
\end{remark}

The corresponding decomposition for $\VOA{V}$-modules proceeds similarly.
Let $\Mod{M}$ be a non-zero $\VOA{V}$-module 
upon which $\VOA{H}$ acts semisimply.
The $\VOA{H}$-weight space decomposition of $\Mod{M}$ then gives $\Mod{M}=\bigoplus_{\mu \in \grp{M}} \Mod{M}_\mu$,
where $\grp{M} =\set{\mu \in \RR^r \st \Mod{M}_{\mu} \neq 0}$ is countable.
Using the triviality of annihilating ideals \cite[Cor.~4.5.15]{LL} as before, we see that $\grp{M}$
is closed under the additive action of $\grp{L}$, meaning that $\lambda\in\grp{L}$ and $\mu\in\grp{M}$
imply that $\lambda+\mu\in\grp{M}$. It follows that each $\Mod{M}_{\mu}$ is a $\VOA{V}_0$-module.
Decomposing as an $\VOA{H}\otimes\VOA{C}$-module,
we get $\Mod{M}_\mu=\Mod{F}_\mu\otimes\Mod{D}_\mu$, for some $\VOA{C}$-module $\Mod{D}_\mu$.
The key step towards proving a decomposition theorem for modules is now to establish certain fusion products involving the $\Mod{M}_{\mu}$ and $\Mod{D}_{\mu}$.

\begin{prop}\label{thm:SW-Mod-basic-fusion}
Let $\VOA{V}$, $\VOA{H}$, $\VOA{C}$, $\VOA{W}$ and $\grp{L}$ 
be as in \cref{thm:SW-VOA} and let $\Mod{M}$, 
$\grp{M}$ and $\Mod{M}_\mu=\Mod{F}_\mu\otimes\Mod{D}_\mu$ be as in the previous paragraph.
Then, the following fusion rules hold for all $\lambda \in \grp{L}$ and $\mu \in \grp{M}$: 
\begin{subequations}
\begin{align}
\Mod{V}_\lambda\fus{\VOA{V}_0}\Mod{M}_\mu & \cong \Mod{M}_{\lambda+\mu} \label{eqn:VlxMm},\\
\Mod{C}_\lambda\fus{\VOA{C}}\Mod{D}_\mu & \cong \Mod{D}_{\lambda+\mu} \label{eqn:ClxDm}.
\end{align}
\end{subequations}
\end{prop}
\noindent We mention that when $\Mod{M}=\Mod{V}$, the fusion rule \eqref{eqn:VlxMm} is precisely the result of Miyamoto reported in
\cref{thm:MiySimpleCurrent}. However, we cannot use Miyamoto's proof in this more general setting because it would amount to assuming the simplicity of the $\Mod{M}_\mu$ as $\VOA{V}_0$-modules.
\begin{proof}
We will detail the proof of the fusion rule \eqref{eqn:VlxMm}, noting that \eqref{eqn:ClxDm} will then follow immediately by applying \cref{prop:FusingTensorProducts}.

To prove \eqref{eqn:VlxMm}, let $\widetilde{\Mod{M}}$ denote the $\VOA{V}$-submodule of $\Mod{M}$ generated by $\Mod{M}_{\mu}$.  Then, $(\Mod{M} / \widetilde{\Mod{M}})_{\mu} = 0$.  If $v \in \Mod{V}_{-\lambda}$ is non-zero, for some $\lambda \in \grp{L}$, and $w \in (\Mod{M} / \widetilde{\Mod{M}})_{\lambda + \mu}$, then it follows that $v$ must annihilate $w$, hence that $w=0$ by the triviality of annihilating ideals \cite[Cor.~4.5.15]{LL}.  We conclude that $(\Mod{M} / \widetilde{\Mod{M}})_{\lambda + \mu} = 0$, that is $\widetilde{\Mod{M}}_{\lambda + \mu} = \Mod{M}_{\lambda + \mu}$, for all $\lambda \in \grp{L}$.

The action of $\VOA{V}$ on $\Mod{M}$ now restricts to an action of $\Mod{V}_{\lambda}$ on $\Mod{M}_{\mu}$.  The space generated by the latter action is therefore precisely $\Mod{M}_{\lambda + \mu}$ \cite[Prop.~4.5.6]{LL}.  It now follows from the universal property of fusion products that there exists a surjection
\begin{equation} \label{eq:surj}
	\Mod{V}_\lambda\fus{\VOA{V}_0}\Mod{M}_\mu \lsra \Mod{M}_{\lambda+\mu},
\end{equation}
for each $\lambda \in \grp{L}$ and $\mu \in \grp{M}$.  Fusing with the simple current $\Mod{V}_{-\lambda}$ therefore gives
\begin{equation}
	\Mod{M}_{\mu} \cong \Mod{V}_{-\lambda} \fus{\VOA{V}_0} (\Mod{V}_{\lambda} \fus{\VOA{V}_0} \Mod{M}_{
	\mu}) \lsra \Mod{V}_{-\lambda} \fus{\VOA{V}_0} \Mod{M}_{\lambda+\mu} \lsra \Mod{M}_{\mu},
\end{equation}
the second surjection just being \eqref{eq:surj} with $(\lambda, \mu)$ replaced by $(-\lambda, \lambda + \mu)$.  Since these surjections preserve conformal weights and the dimensions of the generalised eigenspaces of $L_0$ are finite, by hypothesis, it follows that $\Mod{V}_{-\lambda} \fus{\VOA{V}_0} \Mod{M}_{\lambda+\mu} = \Mod{M}_{\mu}$, for all $\lambda \in \grp{L}$, proving \eqref{eqn:VlxMm}.
\end{proof}

If $\lambda \in \grp{N}$, then the fusion rules \eqref{eqn:ClxDm} imply that $\Mod{D}_{\lambda + \mu} = \Mod{D}_{\mu}$, hence that the $\Mod{D}_{[\mu]} \equiv \Mod{D}_{\mu}$ are well defined.  The decomposition of $\Mod{M}$ as a $\VOA{W} \otimes \VOA{C}$-module now follows as before.  Before stating this formally, it is convenient to observe that if $\grp{M} = \grp{M}^1 \cup \cdots \cup \grp{M}^n$ is a disjoint union of orbits under the action of $\grp{L}$, then $\Mod{M} = \Mod{M}^1 \oplus \cdots \oplus \Mod{M}^n$ as a $\VOA{V}$-module, where $\Mod{M}^i = \bigoplus_{\mu \in \grp{M}^i} \Mod{M}^i_{\mu}$.  While the $\Mod{M}_i$ need not be indecomposable as $\VOA{V}$-modules, several of the arguments to come will be simplified if we assume that $\grp{M}$ consists of a single $\grp{L}$-orbit.  Conclusions about more general $\Mod{M}$ then follow immediately from the properties of direct sums.
\begin{thm} \label{thm:SW-Mod}
Let $\VOA{V}$, $\VOA{H}$, $\VOA{C}$, $\VOA{W}$, $\grp{L}$ and $\grp{N}$ be as in \cref{thm:SW-VOA} and let $\Mod{M}$ be a 
$\VOA{V}$-module 
upon which $\VOA{H}$ acts semisimply.
Then,  $\Mod{M}$ decomposes as
\begin{equation} \label{eq:DecompM}
\Mod{M} = \bigoplus_{\mu \in \grp{M}} \Mod{M}_\mu = \bigoplus_{\mu \in \grp{M}} \Mod{F}_{\mu} \otimes \Mod{D}_{\mu} = \bigoplus_{[\mu] \in \grp{M} / \grp{N}} \Mod{W}_{[\mu]} \otimes \Mod{D}_{[\mu]},
\end{equation}
where $\grp{M}$ is a union of $\grp{L}$-orbits and the $\Mod{D}_{\mu} = \Mod{D}_{[\mu]}$ are $\VOA{C}$-modules satisfying $\Mod{C}_{\lambda} \fus{\VOA{C}} \Mod{D}_{\mu} = \Mod{D}_{\lambda + \mu}$, for all $\lambda \in \grp{L}$ and $\mu \in \grp{M}$.  Moreover, {if we assume (for convenience) that $\grp{M}$ is a single $\grp{L}$-orbit, then}:
\begin{enumerate}
\item If $\Mod{M}$ is a non-zero $\VOA{V}$-module, then all of the $\Mod{D}_{\mu}$ are non-zero. \label{it:SWNonZero}
\item If $\Mod{M}$ is a simple $\VOA{V}$-module, then all of the $\Mod{D}_{\mu}$ are simple. \label{it:SWSimple}
\item If $\Mod{M}$ is an indecomposable $\VOA{V}$-module, then all of the $\Mod{D}_{\mu}$ are indecomposable. \label{it:SWIndec}
\item If $\ses{\Mod{M}'}{\Mod{M}}{\Mod{M}''}$ is exact, with $\Mod{M}'$ and $\Mod{M}''$ non-zero, then $\Mod{M}'$ and $\Mod{M}''$ decompose as in \eqref{eq:DecompM}:
\begin{equation} \label{eq:DecompM'}
\Mod{M}' = \bigoplus_{\mu \in \grp{M}} \Mod{M}'_\mu = \bigoplus_{\mu \in \grp{M}} \Mod{F}_{\mu} \otimes \Mod{D}'_{\mu}, \qquad
\Mod{M}'' = \bigoplus_{\mu \in \grp{M}} \Mod{M}''_\mu = \bigoplus_{\mu \in \grp{M}} \Mod{F}_{\mu} \otimes \Mod{D}''_{\mu}.
\end{equation}
Moreover, $\ses{\Mod{D}'_{\mu}}{\Mod{D}_{\mu}}{\Mod{D}''_{\mu}}$ is also exact, for all $\mu \in \grp{M}$. \label{it:SWExact}
\item If $\Mod{M}$ has a composition series with composition factors $\Mod{S}^i$, $1 \le i \le n$, then each $\Mod{S}^i$ decomposes into an $\VOA{H} \otimes \VOA{C}$-module as $\Mod{S}^i = \bigoplus_{\mu \in \grp{M}} \Mod{F}_{\mu} \otimes \Mod{T}^i_{\mu}$, where the $\Mod{T}^i_{\mu}$, $1 \le i \le n$, are the composition factors of $\Mod{D}_{\mu}$, for each $\mu \in \grp{M}$.  In particular, each $\Mod{D}_{\mu}$ has the same composition length as $\Mod{M}$. \label{it:SWCompSeries}
\item If $\Mod{M}$ has a socle, then so do the $\Mod{D}_{\mu}$ and $\soc\Mod{M} = \bigoplus_{\mu \in \grp{M}} \Mod{F}_{\mu} \otimes \soc\Mod{D}_{\mu}$. \\
If $\Mod{M}$ has a radical, then so do the $\Mod{D}_{\mu}$.  If, in addition, $\Mod{M}$ has no subquotient isomorphic to the direct sum of two isomorphic simple $\VOA{V}$-modules, then $\rad\Mod{M} = \bigoplus_{\mu \in \grp{M}} \Mod{F}_{\mu} \otimes \rad\Mod{D}_{\mu}$. \label{it:SWRadSoc}
\item If $\Mod{M}$ has a socle series, then so do the $\Mod{D}_{\mu}$ and the corresponding Loewy diagram is obtained by replacing each composition factor $\Mod{S}^i$ by $\Mod{T}^i_{\mu}$, where $\Mod{S}^i = \bigoplus_{\mu \in \grp{M}} \Mod{F}_{\mu} \otimes \Mod{T}^i_{\mu}$. \\
If $\Mod{M}$ has a radical series, then so do the $\Mod{D}_{\mu}$. If, in addition, $\Mod{M}$ has no subquotient isomorphic to the direct sum of two isomorphic simple $\VOA{V}$-modules, then the corresponding Loewy diagram is obtained by replacing each composition factor $\Mod{S}^i$ by $\Mod{T}^i_{\mu}$, where $\Mod{S}^i = \bigoplus_{\mu \in \grp{M}} \Mod{F}_{\mu} \otimes \Mod{T}^i_{\mu}$. \label{it:SWLoewy}
\end{enumerate}
\end{thm}
\begin{proof}
We have already proven the non-numbered statements.  For \ref{it:SWNonZero}, suppose that $\Mod{D}_{\mu} = 0$, for some $\mu \in \grp{M}$.  Then, $\Mod{M}_{\mu} = \Mod{F}_{\mu} \otimes \Mod{D}_{\mu}$ would be $0$, contradicting the definition of $\grp{M}$.  The argument for \ref{it:SWSimple} is likewise short:  $\Mod{M}$ simple implies that each $\Mod{M}_{\mu}$, with $\mu \in \grp{M}$, is simple, by \cref{thm:KR}, which forces each of the $\Mod{D}_{\mu}$ to be simple.  To prove \ref{it:SWIndec}, note that if some $\Mod{D}_{\nu}$, $\nu \in \grp{M}$, were decomposable, then every $\Mod{D}_{\mu}$, $\mu \in \grp{M}$, would be decomposable because $\mu - \nu \in \grp{L}$, hence $\Mod{D}_{\mu} \cong \Mod{C}_{\mu - \nu} \fus{\Mod{C}} \Mod{D}_{\nu}$.  But then, every $\Mod{M}_{\mu}$ would be decomposable, hence so would $\Mod{M}$, a contradiction.

Given the exact sequence in \ref{it:SWExact}, it is clear that $\VOA{H}$ acts semisimply on both $\Mod{M}'$ and $\Mod{M}''$, hence that we have the decompositions \eqref{eq:DecompM'} except that some of the $\Mod{M}'_{\mu}$ or $\Mod{M}''_{\mu}$ might be zero, for some $\mu \in \grp{M}$.  However, $\grp{M}$ is assumed to consist of a single $\grp{L}$-orbit, so either all the $\Mod{M}'_{\mu}$ are zero or none of them are (and the same for the $\Mod{M}''_{\mu}$).  But, either being zero would imply that the corresponding module is zero, which is ruled out by hypothesis.  Thus, the $\Mod{M}'_{\mu}$ and $\Mod{M}''_{\mu}$ are non-zero, for all $\mu \in \grp{M}$.

Since restricting to a $\VOA{V}_0$-module and projecting onto the (simultaneous) eigenspaces of the $h_0^i$ (which commute with $\VOA{V}_0 = \VOA{H} \otimes \VOA{C}$) are exact functors, the sequence $\ses{\Mod{F}_{\mu} \otimes \Mod{D}'_{\mu}}{\Mod{F}_{\mu} \otimes \Mod{D}_{\mu}}{\Mod{F}_{\mu} \otimes \Mod{D}''_{\mu}}$ is exact, for all $\mu \in \grp{M}$.  However, $\End_{\VOA{H}} \Mod{F}_{\mu} \cong \CC$ implies that each non-trivial map in this exact sequence has the form $\id_{\Mod{F}_{\mu}} \otimes\, d_{\mu}$, where $d_{\mu}$ is a $\VOA{C}$-module homomorphism.  The required exactness of the sequence of $\VOA{C}$-modules thus follows, proving \ref{it:SWExact}.

For \ref{it:SWCompSeries}, let $0 = \Mod{M}^0 \subset \Mod{M}^1 \subset \cdots \subset \Mod{M}^{n-1} \subset \Mod{M}^n = \Mod{M}$ be a composition series, so that $\Mod{S}^i = \Mod{M}^i / \Mod{M}^{i-1}$ is simple, for all $1 \le i \le n$.  Then, $\ses{\Mod{M}^{i-1}}{\Mod{M}^i}{\Mod{S}^i}$ is exact, hence so is $\ses{\Mod{D}^{i-1}_{\mu}}{\Mod{D}^i_{\mu}}{\Mod{T}^i_{\mu}}$, for all $1 \le i \le n$ and $\mu \in \grp{M}$, by \ref{it:SWExact}.  Here, we have decomposed each $\Mod{M}^i$ as $\Mod{M}^i = \bigoplus_{\mu \in \grp{M}} \Mod{F}_{\mu} \otimes \Mod{D}^i_{\mu}$, so that $\Mod{D}^0_{\mu} = 0$ and $\Mod{D}^n_{\mu} = \Mod{D}_{\mu}$, and each $\Mod{S}^i$ as $\Mod{S}^i = \bigoplus_{\mu \in \grp{M}} \Mod{F}_{\mu} \otimes \Mod{T}^i_{\mu}$.  Since the $\Mod{T}^i_{\mu}$ are non-zero and simple, by \ref{it:SWNonZero} and \ref{it:SWSimple}, they are the composition factors of $\Mod{D}_{\mu}$.

We turn to \ref{it:SWRadSoc}.  Let $\set{\Mod{M}^i}_{i \in I}$ be the set of all simple submodules of $\Mod{M}$ so that $\soc\Mod{M} = \sum_{i \in I} \Mod{M}^i$.  Then, each $\Mod{M}^i$ decomposes as $\Mod{M}^i = \bigoplus_{\mu \in \grp{M}} \Mod{F}_{\mu} \otimes \Mod{D}^i_{\mu}$, where $\Mod{D}^i_{\mu}$ is a simple submodule of $\Mod{D}_{\mu}$, for each $i \in I$ and $\mu \in \grp{M}$, by \ref{it:SWSimple} and \ref{it:SWExact}.  As sums distribute over tensor products, we have
\begin{equation}
	\soc\Mod{M} = \sum_{i \in I} \biggl[ \bigoplus_{\mu \in \grp{M}} \Mod{F}_{\mu} \otimes \Mod{D}^i_{\mu} \biggr] = \bigoplus_{\mu \in \grp{M}} \Mod{F}_{\mu} \otimes \biggl( \sum_{i \in I} \Mod{D}^i_{\mu} \biggr).
\end{equation}
It remains to show that for each $\mu \in \grp{M}$, every simple submodule of $\Mod{D}_{\mu}$ is one of the $\Mod{D}^i_{\mu}$.

Consider therefore a simple submodule $\Mod{E}_{\mu} \subseteq \Mod{D}_{\mu}$, for some given $\mu \in \grp{M}$.  Form $\Mod{E}_{\nu} = \Mod{C}_{\nu-\mu} \fus{\VOA{C}} \Mod{E}_{\mu}$, for all $\nu \in \Mod{M}$ (so that $\nu-\mu \in \grp{L}$), and note that each $\Mod{E}_{\nu}$ is a simple submodule of $\Mod{D}_{\nu}$, by parts \ref{it:SCSimple} and \ref{it:SCExact} of \cref{prop:simplecurrents}.  Tensoring over $\CC$ is exact, so $\bigoplus_{\nu \in \grp{M}} \Mod{F}_{\nu} \otimes \Mod{E}_{\nu}$ is a submodule of $\bigoplus_{\nu \in \grp{M}} \Mod{F}_{\nu} \otimes \Mod{D}_{\nu} = \Mod{M}$.  Moreover, it is a simple submodule because it has the same number of composition factors as $\Mod{E}_{\mu}$, by \ref{it:SWCompSeries}.  It is therefore one of the $\Mod{M}^i$, hence $\Mod{E}_{\mu}$ is one of the $\Mod{D}^i_{\mu}$.  It follows that $\sum_{i \in I} \Mod{D}^i_{\mu} = \soc\Mod{D}_{\mu}$, as required.

The same argument works for the radical, which we recall is the intersection of the maximal proper submodules, except that intersections need not distribute over sums.  The additional condition on $\Mod{M}$ guarantees this \cite{Ben}.  The proof of \ref{it:SWRadSoc} is thus complete and the proof of \ref{it:SWLoewy} now follows similarly to that of \ref{it:SWCompSeries}.
\end{proof}
\begin{remark}
It is not clear if the condition imposed on $\Mod{M}$ in the radical parts \ref{it:SWRadSoc} and \ref{it:SWLoewy} is required.  However, if $\rad\Mod{M}$ decomposes as $\rad\Mod{M} = \bigoplus_{\mu \in \grp{M}} \Mod{F}_{\mu} \otimes \Mod{R}_{\mu}$, then without this condition, the argument used in the proof only establishes that $\Mod{R}_{\mu} \subseteq \rad\Mod{D}_{\mu}$, for each $\mu \in \grp{M}$.
\end{remark}

{Unlike the $\Mod{C}_{[\mu]}$ in \eqref{eq:DecompVWC}, the coset modules $\Mod{D}_{[\mu]}$, $[\mu] \in \grp{M} / \grp{N}$, appearing in \eqref{eq:DecompM} need not be mutually non-isomorphic.  We shall illustrate this with a simple example in \cref{sec:EasyExamples}.  In the following \lcnamecref{sec:multfree}, we first give three useful criteria which guarantee that the $\Mod{D}_{[\mu]}$ are all non-isomorphic.

\subsection{Criteria for being multiplicity-free} \label{sec:multfree}

In this section, we discuss whether the decomposition \eqref{eq:DecompM} is multiplicity-free or not.  In other words, we investigate when one can assert that the {$\Mod{D}_{\mu}$ or the} $\Mod{D}_{[\mu]}$ are mutually non-isomorphic, in the notation of \cref{thm:SW-Mod}.

\subsubsection{Criterion based on conformal weights} \label{sec:CritConfWt}

It may so happen that the conformal weights of the \hwvs{} of the Heisenberg subalgebra $\VOA{H}$ immediately rule out multiplicities.  For example, 
consider the case of an affine \voa{} $\VOA{V}$ of \emph{negative} level $k$ 
and a 
$\VOA{V}$-module $\Mod{M}$ whose conformal weights are bounded below.  We shall assume, as in \cref{thm:SW-Mod}, that the corresponding set $\grp{M}$ is a single orbit of $\grp{L}$. 
Suppose that the decomposition of $\Mod{M}$ is not multiplicity-free, so that $\Mod{D}_{\mu + \lambda} = \Mod{D}_{\mu}$, for some $\lambda \in \grp{L}$.  Then, $\Mod{C}_{\lambda} \fus{\VOA{C}} \Mod{D}_{\mu} = \Mod{D}_{\mu}$ and so $\Mod{D}_{\mu + n \lambda} = \Mod{D}_{\mu}$, for all $n \in \ZZ$.
However, the conformal weight of the \hwv{} of $\Mod{F}_{\mu+n\lambda}$ is $\frac{1}{2k} \norm{\mu+n\lambda}^2$, which becomes arbitrarily negative for $\abs{n}$ large, because $k<0$.  It follows that the conformal weights of $\Mod{F}_{\mu+n\lambda} \otimes \Mod{D}_{\mu+n\lambda} = \Mod{F}_{\mu+n\lambda} \otimes \Mod{D}_{\mu}$ would become arbitrarily negative, for all $\mu \in \grp{M}$.  This contradicts the hypothesis that the conformal weights of $\Mod{M} = \bigoplus_{\mu \in \grp{M}} \Mod{F}_{\mu} \otimes \Mod{D}_{\mu}$ are bounded below, hence 
the $\Mod{D}_{\mu}$, with $\mu \in \grp{M}$, must all be mutually non-isomorphic.

\subsubsection{Criterion based on symmetries of characters} \label{sec:CritChar}

We can also derive a simple test to rule out multiplicities using the characters
\begin{equation}
	\fch{\Mod{F}_{\mu}}{z;q} = \traceover{\Mod{F}_{\mu}} z^{h_0} q^{L^{\VOA{H}}_0 - c/24} = \frac{z^{\mu} q^{\norm{\mu}^2 / 2}}{\eta(q)}
\end{equation}
of the Fock modules.  This relies on the fact 
that the characters of the $\Mod{D}_{\mu}$ appearing in \eqref{eq:DecompM} will not depend on $z$.  We remark that the factors $z^{h_0}$ and $z^{\mu}$ should be interpreted here as $z_1^{h^1_0} \cdots z_r^{h^r_0}$ and $z_1^{\mu_1} \cdots z_r^{\mu_r}$, respectively, where $r$ is the rank of the Heisenberg \voa{} $\VOA{H}$.

Suppose, for simplicity, that $\grp{M}$ consists of a single $\grp{L}$-orbit, as in \cref{thm:SW-Mod}.  Define $\grp{N}'$ to be the sublattice of Heisenberg weights $\lambda$ such that $\Mod{D}_{\mu} = \Mod{D}_{\lambda + \mu}$, for every $\mu \in \grp{M}$, so that $\grp{N} \le \grp{N}' \le \grp{L}$.  It follows that for every $\lambda \in \grp{N}'$, the character of the decomposition \eqref{eq:DecompM} must satisfy
\begin{align}
	\fch{\Mod{M}}{z;q;\ldots} &= \sum_{\mu \in \grp{M}} \frac{z^{\mu} q^{\norm{\mu}^2 / 2}}{\eta(q)} \fch{\Mod{D}_{\mu}}{q;\ldots}
	= \sum_{\mu \in \grp{M}} \frac{z^{\lambda + \mu} q^{\norm{\lambda + \mu}^2 / 2}}{\eta(q)} \fch{\Mod{D}_{\mu}}{q;\ldots} \notag \\
	&= z^{\lambda} q^{\norm{\lambda}^2/2} \sum_{\mu \in \grp{M}} \frac{z^{\mu} q^{\bilin{\lambda}{\mu}} q^{\norm{\mu}^2 / 2}}{\eta(q)} \fch{\Mod{D}_{\mu}}{q;\ldots}  \\ \notag
	&= z^{\lambda} q^{\norm{\lambda}^2/2} \fch{\Mod{M}}{zq^{\lambda};q;\ldots},
\end{align}
where $q^{\lambda}$ acts on a Heisenberg weight $\mu$ to give $q^{\bilin{\lambda}{\mu}}$.  If the character of $\Mod{M}$ only satisfies this equation when $\lambda \in \grp{N}$, then we may conclude that the $\Mod{D}_{[\mu]}$, with $[\mu] \in \grp{M} / \grp{N}$, are mutually non-isomorphic.  In the case that $\grp{N} = 0$, this conclusion gives the mutual inequivalence of the $\Mod{D}_{\mu}$, for all $\mu \in \grp{M}$.

\subsubsection{Criterion based on open Hopf links} \label{sec:CfitHopf}

In the case of rational \voas{}, the closed Hopf links are, up to normalization, the same as the entries of the modular S-matrix \cite{H-rigidity}. There is also a close connection between Hopf links and properties of characters
for non-rational \voas{} \cite{CG, CG2, CMR}. We will now explain how Hopf links give a criterion 
{for the existence of} 
fixed points {under the action} of {fusing with a} simple current. For this subsection,
we assume that we are working in a ribbon category {$\categ{C}$ of \voa{} modules} \cite{EGNO}; such categories allow us to take
(partial) traces of morphisms.

Let $\Mod{J} \in \categ{C}$ be a simple current and fix a module $\Mod{X} \in \categ{C}$. 
Assume that there exists a positive integer
$s$ such that $\Mod{J}^s\fus{} \Mod{X}\cong \Mod{X}$, so that $\Mod{X}$ is a fixed point of $\Mod{J}^s$.  Recall that the monodromy {of two modules $\Mod{A}$ and $\Mod{B}$} 
is defined by $M_{\Mod{A},\Mod{B}}=R_{\Mod{B},\Mod{A}}\circ R_{\Mod{A},\Mod{B}}$, where
$R$ denotes their braiding.
Recall the notion \cite[Def. 8.10.1]{EGNO}
of categorical twist, $\theta$,
which is a system of natural isomorphisms.
The monodromy satisfies the following balancing for any two modules $\Mod{A},\Mod{B}$:
\[\theta_{\Mod{A}\fus{}\Mod{B}}=M_{\Mod{A},\Mod{B}}\circ(\theta_{\Mod{A}}\fus{} \theta_{\Mod{B}}).\]
In vertex-tensor-categorical setup, $\theta$ is given by $e^{2 i \pi L_0}$.
We will also need the open Hopf link operators from \cite{CG, CG2}.
These are defined as the partial traces
$\Phi_{\Mod{A},\Mod{B}}=\ptr^{\text{Left}}(M_{\Mod{A},\Mod{B}}) \in \End(\Mod{B})$ 
and have the important property that they define a representation of the fusion ring on $\End(\Mod{B})$.
In particular, it follows that
$\Phi_{\Mod{J}\fus{}\Mod{X},\Mod{P}}=\Phi_{\Mod{J},\Mod{P}}\circ\Phi_{\Mod{X},\Mod{P}}$, for any module $\Mod{P} \in \categ{C}$,
and hence 
that
\begin{align}
\Phi_{\Mod{X},\Mod{P}} = \Phi_{\Mod{J}^s\fus{}\Mod{X},\Mod{P}} = \Phi_{\Mod{J}^s,\Mod{P}}\circ\Phi_{\Mod{X},\Mod{P}} = \Phi^s_{\Mod{J},\Mod{P}}\circ\Phi_{\Mod{X},\Mod{P}}
\label{eqn:basicfixedpoint1}.
\end{align}

We shall assume now that $\Mod{P}$ is indecomposable with a finite number of composition factors, so that every endomorphism of $\Mod{P}$ {has a single eigenvalue}, 
and that $M_{\Mod{J},\Mod{P}}$, $\Phi_{\Mod{J},\Mod{P}}$ are a semi-simple endomorphisms of $\Mod{J}\fus{}\Mod{P}$ and $\Mod{P}$, respectively.
The latter assumption will be automatically satisfied if $\Mod{J}$ is a simple current of finite order and both
$\End(\Mod{P})$ and $\End(\Mod{J} \fus{} \Mod{P})$ are finite-dimensional 
\cite[Lem.~2.13]{CKL}. It will also be satisfied if $\Mod{P}$ may be identified with a subquotient of an iterated fusion product of simple modules \cite[Lem.~3.19]{CKL}.  With these assumptions on $\Mod{P}$, \cref{eqn:basicfixedpoint1} shows that the image of $\Phi_{\Mod{X},\Mod{P}}$ is contained in the eigenspace of $\Phi_{\Mod{J},\Mod{P}}^s$ with eigenvalue 1 {and that this eigenspace is either $0$ or $\Mod{P}$ itself}. 
{We therefore have two possible conclusions:}  $\Phi_{\Mod{X},\Mod{P}}=0$ or $\Phi_{\Mod{J},\Mod{P}}^s=\Id_{\Mod{P}}$.

Following \cite{CG}, 
we say that a full subcategory $\categ{P}$ of $\categ{C}$ is a left ideal if for all $\Mod{Q}\in\categ{P}$, we have both $\Mod{D}\fus{}\Mod{Q}\in\categ{P}$, for all $\Mod{D}\in\categ{C}$, and that $\Mod{D}\in\categ{P}$ whenever the composition $\Mod{D}\to\Mod{Q}\to\Mod{D}$ is the identity. We shall assume that $\categ{P}$ is equipped with a modified trace {$t_\bullet$} \cite{CG, GKP}
(for $\categ{P}=\categ{C}$, the modified trace is just the ordinary trace {$t=\text{tr}$}) and a modified dimension $d(\Mod{\bullet})=t_{\Mod{\bullet}}(\Id_{\Mod{\bullet}})$.
We also let $\dim(\bullet)=\tr(\Id_\bullet)$ denote the ordinary trace of the identity morphism.

We now assume that $\Mod{P}$, as introduced above, belongs to 
a left ideal $\categ{P}$ of $\categ{C}$.
For any object $\Mod{D}$ of $\categ{C}$, the properties of the modified trace imply that
\begin{align}
t_{\Mod{D}\fus{}\Mod{P}}(\Id_{\Mod{D}\fus{}\Mod{P}})
&=t_{\Mod{D}\fus{}\Mod{P}}(\Id_{\Mod{D}}\fus{}\Id_{\Mod{P}})
=t_{\Mod{P}}(\ptr^{\text{Left}}(\Id_{\Mod{D}}\fus{}\Id_{\Mod{P}}))
=t_{\Mod{P}}(\text{tr}(\Id_{\Mod{D}})\fus{}\Id_{\Mod{P}} )\notag\\
&=\dim(\Mod{D})t_{\Mod{P}}(\Id_{\Mod{P}})
=\dim(\Mod{D})d(\Mod{P})
\end{align}
and hence that
\begin{align}
t_{\Mod{P}}(\Phi_{\Mod{\Mod{J}}^s,\Mod{P}})
&=t_{\Mod{P}}(\ptr^{\text{Left}}(M_{\Mod{J}^s,\Mod{P}}))
=t_{\Mod{J}^s\fus{}\Mod{P}}(M_{\Mod{J}^s,\Mod{P}}) 
=t_{\Mod{J}^s\fus{}\Mod{P}}
(\theta_{\Mod{J}^s\fus{}\Mod{P}}\circ(\theta^{-1}_{\Mod{J}^s}\fus{}\theta^{-1}_{\Mod{P}}))\notag\\
&=\dim(\Mod{J}^s)d(\Mod{P})(\theta_{\Mod{J}^s\fus{}\Mod{P}}\circ(\theta^{-1}_{\Mod{J}^s}\fus{}\theta^{-1}_{\Mod{P}})).
\end{align}
Here, 
we have used the balancing property of monodromy and 
have identified $\theta_{\Mod{J}^s\fus{}\Mod{P}}\circ(\theta^{-1}_{\Mod{J}^s}\fus{}\theta^{-1}_{\Mod{P}})$ with the scalar by which
it acts.
In the case that $\Phi_{\Mod{J}^s,\Mod{P}}=\Id_{\Mod{P}}$, so \(t_{\Mod{P}}(\Phi_{\Mod{\Mod{J}}^s,\Mod{P}}) = t_{\Mod{P}}(\Id_{\Mod{P}})=d(\Mod{P})\), it follows that 
\(
\dim(\Mod{J})^s(\theta_{\Mod{J}^s\fus{}\Mod{P}}\circ(\theta^{-1}_{\Mod{J}^s}\fus{}\theta^{-1}_{\Mod{P}}))=1,
\)
whenever $d(\Mod{P})\neq 0$.
We summarize this as follows.
\begin{prop}\label{prop:hopfmulti}
Let $\categ{C}$ be a ribbon category, $\Mod{J} \in \categ{C}$ be a simple current
and $\Mod{X} \in \categ{C}$ be a fixed point of $\Mod{J}^s$ so that $\Mod{J}^s\fus{}\Mod{X}\cong \Mod{X}$,
for some $s\in\ZZ_{>0}$.  Let $\categ{P}$ be a left ideal of $\categ{C}$,
equipped with a modified trace $t_\bullet$ and modified dimension $d(\bullet)$.
Let $\Mod{P} \in \categ{P}$ be indecomposable such that $d(\Mod{P}) \neq 0$ and 
$M_{\Mod{J},\Mod{P}}$, 
$\Phi_{\Mod{J},\Mod{P}} \in \End(\Mod{P})$ are semisimple endomorphisms. Then,
one of the following must hold:
\begin{enumerate}
	\item $\Phi_{\Mod{X},\Mod{P}}=0$, which in turn implies that $t_{\Mod{P}}(\Phi_{\Mod{X},\Mod{P}}) = 0$. 
	If $\categ{C}$ is a modular tensor 
	category, then this 
	implies	that the corresponding modular S-matrix entry is zero. \label{it:HopfS}
	\item $\dim(\Mod{J})^s(\theta_{\Mod{J}^s\fus{}\Mod{P}}\circ(\theta^{-1}_{\Mod{J}^s}\fus{}\theta^{-1}_{\Mod{P}}))=1,$
	where we have identified $\theta_{\Mod{J}^s\fus{}\Mod{P}}\circ(\theta^{-1}_{\Mod{J}^s}\fus{}\theta^{-1}_{\Mod{P}})$ with the scalar by which it acts. \label{it:HopfTheta}
\end{enumerate}
\end{prop}
\noindent {As these quantities are computable, in principle, we can rule out fixed points for $\Mod{J} = \Mod{C}_{\lambda}$ or $\Mod{W}_{[\lambda]}$ and thereby deduce a multiplicity-free decomposition.}  We shall illustrate this proposition below in a rational example.

\subsection{Examples} \label{sec:EasyExamples}

Here, we give a selection of simple examples involving the so-called parafermion cosets \cite{FZ,GepNew87} to illustrate the theory developed in this \lcnamecref{sec:SWDuals}. 
Let $\AffVOA{k}{\alg{g}}$ denote the simple \voa{} of level $k$ associated with the affine Kac-Moody (super)algebra $\affine{\alg{g}}$.  Given a Cartan subalgebra $\alg{h} \subset \alg{g}$, let $\VOA{H} \subset \AffVOA{k}{\alg{g}}$ be the corresponding Heisenberg \vosa{}.  The commutant $\VOA{C} = \Com{\VOA{H}}{\AffVOA{k}{\alg{g}}}$ is called the level $k$ parafermion \voa{} of type $\alg{g}$.

\begin{example} \label{ex:Parak=2}
For $\alg{g} = \SLA{sl}{2}$ and $k=2$, the parafermion coset is the Virasoro minimal model $\MinMod{3}{4}$, also known as the Ising model.  The decompositions \eqref{eq:DecompVFC} and \eqref{eq:DecompVWC} become
\begin{equation} \label{eq:Parak=2}
\AffVOA{2}{\SLA{sl}{2}} = \bigoplus_{\lambda \in 4 \ZZ} \bigl[ \Mod{F}_{\lambda} \otimes \Mod{K}_0 \oplus \Mod{F}_{\lambda + 2} \otimes \Mod{K}_{1/2} \bigr] = \Mod{W}_{[0]} \otimes \Mod{K}_0 \oplus \Mod{W}_{[2]} \otimes \Mod{K}_{1/2},
\end{equation}
where $\Mod{K}_h$ denotes the simple $\MinMod{3}{4}$-module of highest weight $h$, the lattice of $\VOA{H}$-weights of $\AffVOA{2}{\SLA{sl}{2}}$ is $\grp{L} = 2 \ZZ$, and the sublattice of $\VOA{H}$-weights giving isomorphic coset modules is $\grp{N} = 4 \ZZ$.  
The convention here for $\Mod{F}_{\lambda}$ is that $\lambda$ indicates the $\SLA{sl}{2}$-weight so that the conformal dimension of this Heisenberg module is $\frac{\lambda^2}{8}$.
The lattice \voa{} $\VOA{W}$ is thus obtained by extending $\VOA{H}$ by the group of simple currents generated by $\Mod{F}_4$.

The representation theory of $\AffVOA{2}{\SLA{sl}{2}}$ is semisimple and it has three simple modules $\Mod{M}^{\omega}$, $\omega=0,1,2$, which are distinguished by the Dynkin labels $(k-\omega,\omega)$ of their highest weights.  $\AffVOA{2}{\SLA{sl}{2}}$ is identified with $\Mod{M}^0$ and the decomposition corresponding to \eqref{eq:Parak=2} for $\Mod{M}^2$ is obtained by swapping $\Mod{K}_0$ with $\Mod{K}_{1/2}$.  In particular, the $\grp{L}$-orbit for $\Mod{M}^2$ is also $\grp{M} = 2 \ZZ$.  The situation for $\Mod{M}^1$ is, however, slightly different:
\begin{equation}
\Mod{M}^1 = \bigoplus_{\mu \in 2 \ZZ + 1} \Mod{F}_{\mu} \otimes \Mod{K}_{1/16} = \Mod{W}_{[1]} \otimes \Mod{K}_{1/16} \oplus \Mod{W}_{[-1]} \otimes \Mod{K}_{1/16}.
\end{equation}
Here, $\grp{M} = 2 \ZZ + 1$ and $\grp{N}' = 2 \ZZ \neq \grp{N}$ (the non-isomorphic lattice modules are paired with isomorphic coset modules).  In other words, this decomposition fails to be multiplicity-free.

To see that this is consistent with the criterion of \cref{sec:CritChar}, recall that $\AKMA{sl}{2}$ admits a family $\sfaut^{\ell}$, $\ell \in \ZZ$, of \emph{spectral flow} automorphisms that lift to automorphisms of the corresponding affine vertex algebras.  The latter may be used to twist the action on an $\AffVOA{k}{\SLA{sl}{2}}$-module $\Mod{M}$ and thereby construct new modules $\sfmod{\ell}{\Mod{M}}$.  Using the conventions of \cite{R1}, the characters of $\Mod{M}$ and $\sfmod{\ell}{\Mod{M}}$ are related by
\begin{equation}
	\fch{\sfmod{\ell}{\Mod{M}}}{z;q} = z^{\ell k} q^{\ell^2 k/4} \fch{\Mod{M}}{z q^{\ell/2}; q}.
\end{equation}

For $k=2$, spectral flow acts on the simple modules as $\sfmod{}{\Mod{M}^{\omega}} = \Mod{M}^{2-\omega}$, $\omega=0,1,2$.  Identifying the weight space of $\SLA{sl}{2}$ with $\CC$ and noting that the scalar product on this space is then $\bilin{\lambda}{\mu} = \frac{1}{4} \lambda \mu$, the criterion of \cref{sec:CritChar} asks us to check which $\lambda \in \CC$ satisfy the relation
\begin{equation}
	\fch{\Mod{M}^{\omega}}{z;q} = z^{\lambda} q^{\lambda^2 / 8} \fch{\Mod{M}^{\omega}}{z q^{\lambda / 4}; q} = \fch{\sfmod{\lambda/2}{\Mod{M}^{\omega}}}{z;q},
\end{equation}
for a given $\Mod{M}^{\omega}$.  Since $\sfaut^2$ acts as the identity, this relation holds for each $\omega$ if $\lambda \in \grp{N} = 4 \ZZ$.  If $\omega \neq 1$, then it does not hold for $\lambda = 2$, hence $\grp{N}' = 4 \ZZ$ and both $\Mod{M}^0$ and $\Mod{M}^2$ have multiplicity-free decompositions in terms of lattice modules.  However, this relation does hold for $\omega = 1$ and $\lambda = 2$, so we cannot conclude that the lattice decomposition of $\Mod{M}^1$ is multiplicity-free (consistent with our explicit calculation that it is not).

With a little more work, we can also see how this failure is consistent with the criterion of \cref{sec:CfitHopf}.  Let $\Mod{X}=\Mod{K}_{1/16}$ and let $\Mod{J}$ be the simple current $\Mod{K}_{1/2}$, so that $\Mod{X}$ is a
fixed point for $\Mod{J}$: $\Mod{J}\fus{}\Mod{X}\cong\Mod{X}$.
Since $\AffVOA{2}{\SLA{sl}{2}}$ is a unitary \voa{}, $\dim(\Mod{J})=1$. Also, as recalled above, 
$\theta$ is given by $e^{2 i\pi L_0}$, hence, 
in our notation, it acts on $\Mod{K_t}$ by $e^{2 i \pi t}$,
where $t=0,1/2,1/16$.
Further, it is easy to check that the category $\categ{C}$ of $\MinMod{3}{4}$-modules has no non-trivial ideals except for $\categ{C}$ itself.

We now verify that for every indecomposable $\Mod{P}$ in $\categ{C}$, either condition \ref{it:HopfS} or \ref{it:HopfTheta} of our Hopf link criterion is satisfied.
\begin{itemize}[leftmargin=\widthof{$\Mod{P}={\Mod{K}}_{1/16}$:\ },labelwidth=\widthof{$\Mod{P}={\Mod{K}}_{1/16}$:\ }]
\item[$\Mod{P}={\Mod{K}}_0$:] In this case,
$\theta_{\Mod{J}\fus{}\Mod{P}}\circ(\theta^{-1}_{\Mod{J}}\fus{}\theta^{-1}_{\Mod{P}})=
\theta_{{\Mod{K}}_{1/2}}\circ(\theta^{-1}_{\Mod{{\Mod{K}}_{1/2}}}\fus{}\theta^{-1}_{{\Mod{K}}_{0}})=
 1.$
\item[$\Mod{P}={\Mod{K}}_{1/2}$:] In this case,
$\theta_{\Mod{J}\fus{}\Mod{P}}\circ(\theta^{-1}_{\Mod{J}}\fus{}\theta^{-1}_{\Mod{P}})=
\theta_{{\Mod{K}}_{0}}\circ(\theta^{-1}_{\Mod{{\Mod{K}}_{1/2}}}\fus{}\theta^{-1}_{{\Mod{K}}_{1/2}})
=1.$
\item[$\Mod{P}={\Mod{K}}_{1/16}$:] In this case,
$\theta_{\Mod{J}\fus{}\Mod{P}}\circ(\theta^{-1}_{\Mod{J}}\fus{}\theta^{-1}_{\Mod{P}})=
\theta_{{\Mod{K}}_{1/16}}\circ(\theta^{-1}_{\Mod{{\Mod{K}}_{1/2}}}\fus{}\theta^{-1}_{{\Mod{K}}_{1/16}})
=-1$, but the modular S-matrix of $\MinMod{3}{4}$ has entry $S_{\Mod{K}_{1/16},\Mod{K}_{1/16}}=0$.
\label{it:Sl2fixedpointinteresting}
\end{itemize}
So we see that in the first two cases condition (2) is satisfied while condition (1) holds in the last.  This is, of course, consistent with the fact that the decomposition is not multiplicity-free.  As an aside, we remark that if we had only known that ${\Mod{K}}_{1/16}$ was a fixed-point of the simple current (which implies that the decomposition is not multiplicity-free), then
we could have instead deduced that $S_{\Mod{K}_{1/16},\Mod{K}_{1/16}}$ must vanish, as above.
\end{example}

\begin{example} \label{ex:Parak=-4/3}
	A more interesting example is the parafermion coset with $\alg{g} = \SLA{sl}{2}$ at level $k=-\frac{4}{3}$.  In \cite{AdaCon05}, Adamovi\'{c} showed that the resulting coset \voa{} is the (simple) singlet algebra $\Sing{1,3}$ of central charge $c=-7$.  This is strongly generated by the energy-momentum tensor and a single conformal primary of weight $5$.  We can revisit and extend this study using the results of this \lcnamecref{sec:SWDuals}.  However, we stress that the parent \voa{} $\AffVOA{-4/3}{\SLA{sl}{2}}$ does not satisfy the conditions of \cref{sec:HLZ} that would allow us to apply the theory of Huang-Lepowsky-Zhang.  Nevertheless, we shall proceed with the analysis, assuming that this theory may be applied.  The results suggest that this assumption is, in this case, not unreasonable.

	Let $\Lambda_0$ and $\Lambda_1$ denote the fundamental weights of $\AKMA{sl}{2}$.  The \voa{} $\AffVOA{-4/3}{\SLA{sl}{2}}$ admits precisely three \hwms{}, namely the simple modules $\Mod{M}^{\omega}$ whose highest weights have the form $(k-\omega) \Lambda_0 + \omega \Lambda_1$, where $\omega \in \set{0,-\frac{2}{3},-\frac{4}{3}}$, as well as an uncountable number of simple non-\hwms{} \cite{AM1,GabFus01,RidRel15}.  In particular, $\Sing{1,3}$ is not a rational \voa{}.  As the level is negative and these \hwms{} have conformal weights that are bounded below, the criterion of \cref{sec:CritConfWt} applies and we conclude that their decompositions are multiplicity-free.

	Explicitly, the decomposition \eqref{eq:DecompVFC} takes the form
	\begin{equation}\label{eq:sl2decomp}
		\AffVOA{-4/3}{\SLA{sl}{2}} = \bigoplus_{\lambda \in 2\ZZ} \Mod{F}_{\lambda} \otimes \Mod{C}_{\lambda},
	\end{equation}
	where $\Mod{C}_{\lambda}$ is a simple \hw{} $\Sing{1,3}$-module whose \hwv{} has conformal weight $\Delta_{\lambda} = \frac{1}{16} \abs{\lambda} (3 \abs{\lambda} + 8)$.  The convention here for $\Mod{F}_{\lambda}$ is again  that $\lambda$ indicates the $\SLA{sl}{2}$-weight so that the conformal dimension of this Heisenberg module is $-\frac{3}{16} \lambda^2$.	
	Of course, $\Mod{C}_{\lambda}$ and $\Mod{C}_{-\lambda}$ are not isomorphic for $\lambda \neq 0$ because the decomposition \eqref{eq:sl2decomp} is multiplicity-free --- they must therefore be distinguished by the action of the zero mode of the weight $5$ conformal primary.

	The theory of \cref{sec:HCosets} shows that the $\Mod{C}_{\lambda}$, with $\lambda \in  2\ZZ$, are all (non-isomorphic) simple currents.  This had been previously deduced \cite{RidMod13,CM1} from the (conjectural) \emph{standard} Verlinde formula of \cite{CreLog13,RidVer14} for non-rational \voas{}.  Noting that $\Delta_{\pm 4} = 5$, we remark \cite{CRW, RidMod13} that the simple current extension of $\Sing{1,3}$ by the $\Mod{C}_{\lambda}$, with $\lambda \in 4 \ZZ$, is the triplet algebra $\Trip{1,3}$ of Kausch \cite{Ka}.

	Consider now the $\AffVOA{-4/3}{\SLA{sl}{2}}$-modules $\sfmod{-2}{\Mod{M}^{-2/3}}$ and $\sfmod{}{\Mod{M}^{-2/3}}$, obtained by twisting the action on $\Mod{M}^{-2/3}$ by the spectral flow automorphisms $\sfaut^{\ell}$, $\ell \in \ZZ$.  Whilst both these modules have conformal weights that are unbounded below, their decompositions into $\VOA{H} \otimes \Sing{1,3}$-modules are nevertheless multiplicity-free:
	\begin{equation}
		\sfmod{-2}{\Mod{M}^{-2/3}} = \sum_{\mu \in 2 \ZZ} \Mod{F}_{\mu} \otimes \Mod{D}^{(-2)}_{\mu}, \qquad
		\sfmod{}{\Mod{M}^{-2/3}} = \sum_{\mu \in 2 \ZZ} \Mod{F}_{\mu} \otimes \Mod{D}^{(1)}_{\mu}.
	\end{equation}
	Here, the $\Mod{D}^{(-2)}_{\mu}$ and $\Mod{D}^{(1)}_{\mu}$ are simple \hw{} $\Sing{1,3}$-modules whose \hwvs{} have conformal weights given by
	\begin{equation}
		\Delta^{(-2)}_{\mu} =
		\begin{cases*}
			\frac{1}{16} \mu (3 \mu + 8) & if \(\mu \le -2\), \\
			\frac{1}{16} (\mu + 4) (3 \mu + 4) & if \(\mu \ge -2\)
		\end{cases*}
		\quad \text{and} \quad
		\Delta^{(1)}_{\mu} =
		\begin{cases*}
			\frac{1}{16} (\mu - 4) (3 \mu - 4) & if \(\mu \le 2\), \\
			\frac{1}{16} \mu (3 \mu - 8) & if \(\mu \ge 2\),
		\end{cases*}
	\end{equation}
	respectively.

	The interesting thing about the $\AffVOA{-4/3}{\SLA{sl}{2}}$-modules $\sfmod{-2}{\Mod{M}^{-2/3}}$ and $\sfmod{}{\Mod{M}^{-2/3}}$ is that they appear, together with two copies of the vacuum module $\Mod{M}^0$, as the composition factors of an indecomposable $\AffVOA{-4/3}{\SLA{sl}{2}}$-module $\Mod{P}^0$.  This module was first constructed as a fusion product in \cite{GabFus01} and was structurally characterised in \cite{CR1} (see \cite{AdaLat09} for a construction and characterisation of a different indecomposable $\AffVOA{-4/3}{\SLA{sl}{2}}$-module).  The action of the Virasoro zero mode $L_0$ on $\Mod{P}^0$ is non-semisimple.  The Loewy diagram for $\Mod{P}^0$ has the form
	\begin{equation}
		\begin{tikzpicture}[>=latex,baseline={(right1.base)},
		                    nom/.style={circle,draw=black!20,fill=black!20,inner sep=1pt}]
			\node (top1) at (5,1.5) [] {\(\Mod{M}^0\)};
			\node (left1) at (3.5,0) [] {\(\sfmod{-2}{\Mod{M}^{-2/3}}\)};
			\node (right1) at (6.5,0) [] {\(\sfmod{}{\Mod{M}^{-2/3}}\)};
			\node (bot1) at (5,-1.5) [] {\(\Mod{M}^0\)};
			\node at (5,0) [nom] {\(\Mod{P}^0\)};
			\draw [->] (top1) -- (left1);
			\draw [->] (top1) -- (right1);
			\draw [->] (left1) -- (bot1);
			\draw [->] (right1) -- (bot1);
		\end{tikzpicture}
		\ ,
	\end{equation}
	where our convention is that the socle appears at the bottom.  An immediate consequence of \cref{thm:SW-Mod} is that there exists a countably-infinite number of mutually non-isomorphic indecomposable $\Sing{1,3}$-modules $\Mod{P}^0_{\mu}$, $\mu \in 2 \ZZ$, on which the $\Sing{1,3}$ Virasoro zero mode acts non-semisimply.  The Loewy diagrams of these indecomposables are
	\begin{equation}
		\begin{tikzpicture}[>=latex,baseline={(right1.base)},
		                    nom/.style={circle,draw=black!20,fill=black!20,inner sep=1pt}]
			\node (top1) at (5,1.5) [] {\(\Mod{C}_{\mu}\)};
			\node (left1) at (3.5,0) [] {\(\Mod{D}^{(-2)}_{\mu}\)};
			\node (right1) at (6.5,0) [] {\(\Mod{D}^{(1)}_{\mu}\)};
			\node (bot1) at (5,-1.5) [] {\(\Mod{C}_{\mu}\)};
			\node at (5,0) [nom] {\(\Mod{P}^0_{\mu}\)};
			\draw [->] (top1) -- (left1);
			\draw [->] (top1) -- (right1);
			\draw [->] (left1) -- (bot1);
			\draw [->] (right1) -- (bot1);
		\end{tikzpicture}
		\ .
	\end{equation}
	The existence of such $\Sing{1,3}$-modules was predicted in \cite{RidMod13} from the fact that similar indecomposables have been constructed \cite{AdaLat09,TW} for a simple current extension, the triplet algebra $\Trip{1,3}$.
\end{example}

\section{Properties of Heisenberg Cosets}\label{sec:properties}

Recall from the introduction that {one of} our main applications {for Heisenberg cosets is to construct} 
new, potentially $C_2$-cofinite, \voas{} as extensions: 
\[
\VOA{V} \xrightarrow{\quad \VOA{H}-\text{coset}\quad } \VOA{C} \xrightarrow{\quad \text{extension}\quad } \VOA{E}.
\]
So far, we understand how $\VOA{V}$-modules decompose as $\VOA{H}\otimes\VOA{C}$-modules. The remaining tasks are to {identify when} 
$\VOA{C}$ may be extended by certain abelian intertwining algebras to a larger algebra $\VOA{E}$. This 
will be {stated in} \cref{thm:extension}. 
Since abelian intertwining algebra extensions are mild generalizations of simple current extensions, analogous arguments to \cite{CKL} allow us to give precise criteria for the lifting of $\VOA{H}\otimes\VOA{C}$-modules to $\VOA{V}$-modules, see \cref{thm:lift}. An analogous criterion for the lifting of $\VOA{C}$-modules to $\VOA{E}$-modules is given in \cref{cor:liftext}.

\subsection{Extended Algebras}
If certain Fock modules involved in the \voa{} decomposition yield a lattice (super) \voa{},
then the corresponding coset modules form a (super) \voa{} as well. Thus, we get
extensions of the coset.
\begin{thm}\label{thm:extension}
Let \[\VOA{V} = \bigoplus_{\lambda\in\grp{L}} \Mod{F}_\lambda\otimes\Mod{C}_\lambda.\]
If $\grp{E}$ is a sub-lattice of $\grp{L}$, such that $\bigoplus_{\lambda\in\grp{E}}\Mod{F}_\lambda$ forms
a lattice \voa, 
then $\VOA{E}=\bigoplus_{\lambda\in\grp{E}}\Mod{C}_\lambda$
has a natural \voa{} structure.
\end{thm}
\begin{proof}
This result is an immediate corollary of \cite[Thm.\ 3.1, 3.2]{Li} with  $\ell=1$, see also \cite{DL}.
\cite[Thm.\ 3.1, 3.2]{Li} in fact guarantee a generalized vertex algebra structure on 
$\bigoplus_{\lambda\in\grp{L}}\VOA{C}_\lambda$. Note
that no restrictions with regards to vertex tensor category theory are needed on $\VOA{V}$ or $\VOA{C}$.
\end{proof}	
For a more general scenario involving mirror extensions, see \cite{Lin}.

\begin{example} \label{ex:sl2decompcont}
Let $\alg{g}$ be a simple simply laced Lie algebra and let $k=\frac{p}{q}\neq 0$ be a rational number ($p, q$ co-prime). 
We do not require it to be an admissible level. 
Then $\AffVOA{k}{\alg{g}}$ is graded  by $\frac{1}{\sqrt{k}}\grp{Q}=\sqrt{\frac{q}{p}}\grp{Q}$ with $\grp{Q}$ the root lattice, that is
\[
\AffVOA{k}{\mathfrak{g}} = \bigoplus_{\lambda \in \sqrt{\frac{q}{p}}\grp{Q}} \Mod{F}_{\lambda} \otimes \Mod{C}_{\lambda}.
\] 
The sublattice $p \sqrt{\frac{q}{p}}\grp{Q}= \sqrt{pq}\,\grp{Q}$ is an even sublattice so that 
\[
\VOA{V}_{\sqrt{pq}\grp{Q}} = \bigoplus_{\lambda \in \sqrt{pq}Q} \Mod{F}_{\lambda}
\]
is a lattice \voa. It follows by \cref{thm:extension} that
\[
\VOA{E}_{k, \alg{g}} := \bigoplus_{\lambda \in \sqrt{pq}\grp{Q}} \Mod{C}_{\lambda}
\]
is also a \voa. 

We believe that these extended \voas{} have a good chance to be $C_2$-cofinite. 
The main outcome of \cite{ACR} is that in the case $\alg{g}=\SLA{sl}{2}$ and $k+2\in \QQ\setminus \{ \frac{1}{n} | n\in \ZZ_{>0}\}$ 
the characters of modules of the extended \voa{} are modular if supplemented by pseudotraces.

In two specific examples $C_2$-cofiniteness is already known. One of them is  $\AffVOA{-4/3}{\SLA{sl}{2}}$.
This is then a continuation of  \cref{ex:Parak=-4/3}. 
Recall that
	\begin{equation}\nonumber
		\AffVOA{-4/3}{\SLA{sl}{2}} = \bigoplus_{\lambda \in 2 \ZZ} \Mod{F}_{\lambda} \otimes \Mod{C}_{\lambda},
	\end{equation}
	where $\Mod{C}_{\lambda}$ is a simple \hw{} $\Sing{1,3}$-module whose \hwv{} has conformal weight $\Delta_{\lambda} = \frac{1}{16} \abs{\lambda} (3 \abs{\lambda} + 8)$ and the Heisenberg Fock module $\Mod{F}_{\lambda}$ has conformal dimension 
	$\frac{-3}{16}\lambda^2$.
It follows that
\begin{equation}
 \VOA{V}_{\grp{L}}=\bigoplus_{\lambda \in 4 \ZZ} \Mod{F}_{\lambda} 
\end{equation}
is the lattice \voa{} of the lattice $\grp{L}=\sqrt{-6}\,\ZZ$ and hence
\begin{equation}
 \Trip{1,3}=\bigoplus_{\lambda \in 4 \ZZ}  \Mod{C}_{\lambda}
\end{equation}
is also a \voa. It is actually the $\Trip{1,3}$-triplet  that is well-known to be $C_2$-cofinite \cite{AM3}. 
This relation between singlet \voa{} and 
$\AffVOA{-4/3}{\SLA{sl}{2}}$
has been first realized by Adamovi\'c \cite{AdaCon05} and has a nice generalization to a relation between singlet \voas{} and certain $\mathcal{W}$-algebras \cite{CRW}.
\end{example}

\begin{example}\label{ex:singletgl11} \emph{$\Sing{2}$-singlet algebra  and  $\UAffVOA{k}{\SLSAp{gl}{1}{1}}$} 

We first illustrate how well-known somehow archetypical logarithmic VOAs are related via simple current extensions and Heisenberg cosets thus nicely illustrating the picture advocated in this work together with \cite{CKL}. The picture is as follows:
\begin{equation*}
\xymatrix{
\UAffVOA{k}{\SLSAp{gl}{1}{1}}
	%V_k\left(\alg{gl}\left(1|1\right)\right)
	\ar[rr]^(0.45){ \text{extension}}\ar[d]^{\text{coset}}
	&& 
	\AffVOA{-1/2}{\SLSAp{sl}{2}{1}}
	%L_{-1/2}\left(\alg{sl}\left(2|1\right)\right) 
	\ar[rr]^(0.55){ \text{extension}}\ar[d]^{\text{coset}}  && \beta\gamma \otimes \VOA{V}_{\mathbb Z}\ar[d]^{\text{coset}}\\
	\VOA{H} \otimes \Sing{2} \ar[rr]^(.5){\text{extension}}\ar[rrd]^{\text{coset}}&&
	\AffVOA{-1/2}{\SLA{sl}{2}}
	\ar[rr]^(0.55){ \text{extension}}\ar[d]^{\text{coset}} &&  \beta\gamma \ar[lld]_{\text{coset}} \\
	&& \Sing{2} \ar[rr]^(0.5){ \text{extension}} && \Trip{2}  }
\end{equation*}
$\Sing{2}$ is the $p=2$ singlet VOA \cite{AM1, CM1} and $\Trip{2}$ is its $C_2$-cofinite but non-rational infinite order simple current extensions, called the triplet. See e.g. \cite{AM3}.

These and other extensions have been worked out in \cite{CR3, CR4, AC} while the coset picture has been part of \cite{CR3, CRo, CRW}.
Here, the situation of the singlet algebra $\Sing{2}$ is that $C_1$-cofiniteness of all known admissible modules is established \cite{CMR}, fusion 
coefficients are known \cite{AM4}
and the category of $C_1$-cofinite modules is a vertex tensor category in the sense of \cite{HLZ} provided that every $C_1$-cofinite $\mathbb N$-gradable module is of finite length \cite[Thm.\ 17]{CMR}.

For reference on $\Sing{2}$-modules we refer to \cite{AM1, CM1}. As reference on $\UAffVOA{k}{\SLSAp{gl}{1}{1}}$  we refer to \cite{CR3}.
$\Sing{2}$ has simple typical modules $\Mod{F}_\lambda$ of conformal weight $\frac{1}{2}\lambda(\lambda-1)$ for $\lambda\in \mathbb R\setminus \mathbb Z$.
For $\lambda=1-r$ integer, we have
\[
\ses{\Mod{M}_{r, 1}}{\Mod{F}_{1-r}}{\Mod{M}_{r+1, 1}}
\]
for simple atypical modules $\Mod{M_{r, 1}}$ and $r$ integer.
Similarly, $\UAffVOA{k}{\SLSAp{gl}{1}{1}}$ %$V_k(\alg{gl}(1|1))$ 
has simple highest-weight modules $\Mod{V_{e, n}}$ where the real numbers $e, n$ are the weight labels, and $e/k$ not integer. If $e/k$ is integer say $\ell$, then the higest-weight-module decomposes as
\[
\ses{\Mod{A}_{n-1, \ell k}}{\Mod{V}_{n, \ell k}}{\Mod{A}_{n, \ell k}}
\]
with simple atypical modules $\Mod{A_{n, \ell k}}$ parameterized by real $n$ and integers $\ell$.
The projective covers $\Mod{P_{n, \ell k}}$ have the form
\[
\ses{\Mod{V}_{n+1, \ell k}}{\Mod{P}_{n, \ell k}}{\Mod{V}_{n, \ell k}}.
\]

The commutant of $\Sing{2}$ in $\UAffVOA{k}{\SLSAp{gl}{1}{1}}$ is a rank two Heisenberg \voa{}, and we denote their Fock-modules by $\Mod{F}_{e, n}$ where we take the notation of \cite{CRo}.
Using the explicit realization of $\UAffVOA{k}{\SLSAp{gl}{1}{1}}$-modules of \cite{CRo} we can compute the decomposition of modules. The answer is as follows
\begin{equation}
\begin{split}
\UAffVOA{k}{\SLSAp{gl}{1}{1}} &=  \bigoplus_{m\in \ZZ} \Mod{F}_{0, m} \otimes  \Mod{M}_{m+1, 1},  \qquad
\Mod{A}_{n, \ell k}=  \bigoplus_{m\in \ZZ} \Mod{F}_{-\ell, m-n} \otimes  \Mod{M}_{m, 1} \qquad\text{and}\\
\Mod{V}_{-e, -n+1}&=  \bigoplus_{m\in \ZZ} \Mod{F}_{\frac{e}{k}, n+m} \otimes  \Mod{F}_{\frac{e}{k}-m}.
\end{split}
\end{equation}
It follows with \cref{thm:SW-Mod} that
\[
\Mod{P}_{n, \ell k}=  \bigoplus_{m\in \ZZ} \Mod{F}_{-\ell, -n+m} \otimes  \Mod{S}_{m},
\]
where $\Mod{S}_m$ is an indecomposable $\Sing{2}$-module that has non-split short-exact sequence
\[
\ses{\Mod{F}_{1-m}}{\Mod{S}_m}{\Mod{F}_{2-m}}.
\]
In terms of Loewy diagrams, we have the following:
\begin{align*}
\Mod{P}_{n,\ell k}
&=
\begin{matrix}
\xymatrixcolsep{0.5pc}
\xymatrixrowsep{0.75pc}
\xymatrix{
	& \Mod{A}_{n,\ell k} \ar@{-}[dl] \ar@{-}[dr] & \\
	\Mod{A}_{n+1,\ell k} \ar@{-}[dr]& & \Mod{A}_{n-1,\ell k}\ar@{-}[dl] \ar@{-}[dl]\\
	& \Mod{A}_{n,\ell k} &
}\end{matrix}
%\\
%&=
= \bigoplus_{m\in\ZZ}\Mod{F}_{-\ell,m-n}\otimes
\begin{bmatrix}
\xymatrixcolsep{0.5pc}
\xymatrixrowsep{0.75pc}
\xymatrix{
	& \Mod{M}_{m, 1} \ar@{-}[dl] \ar@{-}[dr] & \\
	\Mod{M}_{m+1,1} \ar@{-}[dr]& & \Mod{M}_{m-1, 1}\ar@{-}[dl] \ar@{-}[dl]\\
	& \Mod{M}_{m,1} &
}\end{bmatrix}.
\end{align*}
The triplet algebra $\Trip{2}$ is known to be $C_2$-cofinite but non-rational. It is a simple current extension of $\Sing{2}$, namely
\[
\Trip{2} = \bigoplus_{m\in \ZZ} \Mod{M}_{1+2m, 1}.
\]
\end{example}

\subsection{Lifting Coset Modules}

In this subsection, we show that whether certain generalized $\VOA{C}$-modules $\Mod{D}$ could be tensored with
appropriate Fock modules so that the product can be induced (lifted) to a $\VOA{V}$-module is essentially
decided by the monodromy
\[
M_{\VOA{C}_\lambda, \Mod{D}} =
R_{\Mod{D}, \VOA{C}_\lambda}\circ R_{\VOA{C}_\lambda, \Mod{D}} :
\Mod{C}_\lambda\fus{}\Mod{D} \rightarrow \Mod{C}_\lambda\fus{}\Mod{D}.
\]
For properties of the monodromy used here, we refer to \cite{CKL}.

The following lemma could be easily proved
as in \cite{CKL} and will be used frequently below.
For a \voa{} $\VOA{V}$, and its vertex tensor category $\categ{C}$.
Let $\mathrm{Pic}_{\categ{C}}(\VOA{V})$ denote the Picard groupoid (see \cite{C, FRS}).
That is,
$\mathrm{Pic}_{\categ{C}}(\VOA{V})$ is the full subcategory of simple currents.
Clearly $\mathrm{Pic}_{\categ{C}}(\VOA{V})$ is closed under tensor product.
\begin{lemma}
\label{lem:pic}
Let $\Mod{X}\in\categ{C}$ be such that for $\Mod{J}_i\in\mathrm{Pic}_{\categ{C}}(V)$,
$M_{\Mod{J}_i,\Mod{X}}=\lambda_{\Mod{J}_i,\Mod{X}}\Id_{\Mod{J}\fus{}\Mod{X}}$
where $\lambda_{\Mod{J}_1,\Mod{X}}\in\CC$ for $i=1,2$.
Then, $\lambda_{\Mod{J}_1,\Mod{X}}\lambda_{\Mod{J}_2,\Mod{X}}=\lambda_{\Mod{J}_1\fus{}\Mod{J}_2,\Mod{X}}$.
\end{lemma}

\begin{thm}\label{thm:lift}
Let $\VOA{V}$, $\VOA{H}$, $\VOA{C}$, %$\VOA{W}$
$\grp{L}$
be as in \cref{thm:SW-VOA},
let $\grp{L}'$ be the dual lattice, $U=\grp{L} \otimes_\ZZ \RR$
and let
$\Mod{D}$ be a generalized $\VOA{C}$-module that appears as a subquotient
of fusion product of some simple $\VOA{C}$-modules.
Then, there exists $\alpha$ in $U$, such that for all $\lambda\in\grp{L}$,
\[
M_{{\VOA{C}}_\lambda, \Mod{D}} = e^{-2\pi i \langle \alpha, \lambda \rangle} \Id_{\Mod{C}_\lambda\fus{}\Mod{D}}
\]
and $\Mod{F}_\beta\otimes\Mod{D}$ lifts to a $\VOA{V}$-module if and only if $\beta\in\alpha+\grp{L}'$.
\end{thm}
\begin{proof}
Recall that we are working with categories of $\Mod{C}$ and $\VOA{H}$ that have real weights
for the respective $L_0$s.
Additionally, recall that we are working over semi-simple category for $\VOA{H}$ and a category for
$\VOA{C}$ each object of which has globally bounded $L_0$-Jordan blocks. 

We know that $\grp{L}$ is equipped with a symmetric nondegenerate bilinear form $\langle\cdot,\cdot\rangle$,
and this form takes real values since the conformal weights with respect to the Heisenberg are real.
By non-degeneracy of $\langle\cdot,\cdot\rangle$, given a
homomorphism $f : \grp{L}\rightarrow S^1$,
there exists an $\alpha\in U$ such that
\begin{align}
f(\lambda)=e^{2i\pi\langle\alpha,\lambda\rangle}
\label{eqn:findalpha}
\end{align}
 for all $\lambda\in \grp{L}$.
 Moreover, $\beta\in U$ satisfies \cref{eqn:findalpha} if and only if
$\beta\in \alpha + \grp{L}'$.

Since each of the ${\VOA{C}}_\lambda$ is a simple current, by results in \cite{CKL},
we know that the monodromy $M_{{\VOA{C}}_\lambda,\,\Mod{D}}=M_\lambda\Id_{\Mod{C}_\lambda\fus{}\Mod{D}}$
for some scalar, say $M_\lambda\in\CC^\times$.
Since $M_{\Mod{C}_\lambda,\Mod{D}}$ is semi-simple and $\Mod{C}_\lambda,\Mod{D},\Mod{C}_\lambda\fus{}\Mod{D}$
have globally bounded $L_0$-Jordan blocks,
proceeding as in the proof of \cite[Eq.~(3.10)]{CKL}, we gather that
$M_{\Mod{C}_\lambda,\Mod{D}}=(\theta_{\Mod{C}_\lambda\fus{}\Mod{D}})_{ss}\circ( (\theta_{\Mod{C}_\lambda}^{-1})_{ss}\fus{}
(\theta_{\Mod{D}}^{-1})_{ss})$, where $ss$ denotes the semi-simple part.
Since each of the modules involved has real conformal weights, we get that $M_\lambda=e^{2i \pi r_\lambda}$ for some $r_\lambda\in\RR$.
So, $M_\lambda\in S^1$ for all $\lambda\in\grp{L}$.
Using \cref{lem:pic} we deduce that $\lambda\mapsto M_\lambda$ is a homomorphism $\grp{L}\rightarrow S^1$
and so is $\lambda\mapsto M_\lambda^{-1}$ since $S^1$ is abelian.

Now, in \cref{eqn:findalpha}, we take $f(\lambda)=M_\lambda^{-1}$ and we get an $\alpha\in U$ such that
$M_\lambda^{-1}=e^{2i\pi\langle\alpha,\lambda\rangle}=M_{\Mod{F}_\lambda,\Mod{F}_\alpha}$.
Using \cref{prop:FusingTensorProducts}
we conclude that $(\Mod{F}_\lambda\otimes{\VOA{C}}_\lambda)\fus{}(\Mod{F}_\alpha\otimes\Mod{D})\cong
\Mod{F}_{\lambda+\alpha}\otimes({\VOA{C}}_\lambda\fus{}\Mod{D})$,
and therefore, monodromy factors over the $\otimes$ tensorands.
We conclude that
$M_{\Mod{F}_\lambda\otimes {\VOA{C}}_\lambda,\, \Mod{F}_\alpha\otimes D}=
M_{\Mod{F}_\lambda,\,\Mod{F}_\alpha}\otimes M_{{\VOA{C}}_\lambda,\,\Mod{D}}=1$.
It now follows that $\Mod{F}_\alpha\otimes\Mod{D}$ lifts.
Moreover, from the arguments above we can conclude that $\Mod{F}_\beta\otimes\Mod{D}$ lifts if and only if $\beta\in\alpha+\grp{L}'$.
\end{proof}

We now combine this with extensions of $\VOA{C}$ as in \cref{thm:extension} to deduce the following.
\begin{cor}\label{cor:liftext}
Assume the setup of \cref{thm:lift}.
Let $\grp{E}$ be a sublattice of $\grp{L}$ such that $\VOA{E}=\oplus_{\lambda\in \grp{E}}\VOA{C}_\lambda$
has a \voa{} structure inherited from $\VOA{V}$ exactly as in \cref{thm:extension}. Then $\Mod{D}$ lifts to a $\VOA{E}$-module
$\oplus_{\lambda\in \grp{E}}\VOA{C}_\lambda\fus{}\Mod{D}$ iff
$\alpha\in\grp{E}'$, where $\grp{E}'$ is the dual lattice of $\grp{E}$.
\end{cor}
\begin{proof}
Recall that each $\VOA{C}_\lambda$ is a simple current for $\VOA{C}$.
Therefore, using \cite{HKL} (for the ``if'' direction)
and \cite{CKL} (for the ``only if'' direction), we know that $\oplus_{\lambda\in \grp{E}}\VOA{C}_\lambda\fus{}\Mod{D}$
is an $\VOA{E}$-module iff
$M_{\VOA{C}_\lambda,\VOA{C}_\mu\fus{}\Mod{D}}=\Id_{\VOA{C}_\lambda\fus{}(\VOA{C}_\mu\fus{}\Mod{D})}$, for all $\lambda,\mu\in \grp{E}$.
Since $\VOA{E}$ is a \voa{}, we know that $M_{\VOA{C}_\lambda,\VOA{C}_\mu}=\Id_{\VOA{C}_\lambda\fus{}\VOA{C}_\mu}$
for all $\lambda,\mu\in\grp{E}$.
By properties of monodromy, we gather that  $M_{\VOA{C}_\lambda,\VOA{C}_\mu\fus{}\Mod{D}}=\Id_{\VOA{C}_\lambda\fus{}(\VOA{C}_\mu\fus{}\Mod{D})}$ for $\lambda,\mu\in\grp{E}$
iff $M_{\VOA{C}_\lambda,\Mod{D}}=\Id_{\VOA{C}_\lambda\fus{}\VOA{D}}$ for all $\lambda\in\grp{E}$,
which in turn holds iff $\alpha\in\grp{E}'$.
\end{proof}

\begin{remark}
Since $\VOA{E}$ is a simple current extension of $\VOA{C}$,
we can utilize arguments similar to \cite[Thm.~4.4]{Lam} in order to analyze certain simple $\VOA{E}$-modules.
Let $\Mod{X}$ be a simple $\VOA{E}$-module such that there exists a simple $\VOA{C}$-module
$\Mod{X}_0\subset\Mod{X}$. (In the notation of \cite{Lam}, the role of group $G$ is played by $\grp{P}$ and
the $V^\chi$ are $\VOA{C}_\lambda$ for $\lambda\in\grp{P}$.) Then, $\mathcal{F}(\Mod{X}_0)=\oplus_{\lambda\in\grp{P}}{\VOA{C}_\lambda}\fus{}\Mod{X}_0$ has a natural structure of an (induced) $\VOA{E}$-module and it surjects onto $\Mod{X}$.
\end{remark}

\begin{example}\label{ex:N2} We now illustrate the lifting properties with unitary minimal models of the $N=2$ super Virasoro algebra.
We refer the reader to \cite{AdaRepN2}, \cite{AdaVoaN2}, \cite{DPYZ} and \cite{SatoConformal} for additional information on these minimal models.

We start with some well-known results whose proofs can be found e.g. in \cite{CLI}.
Let $k$ be a positive integer, then $\VOA{L}_k(\alg{sl}_2)$ contains the lattice \voa $\VOA{V}_{\grp{L}_\alpha}$ 
with $\grp{L}_\alpha=\alpha\ZZ$ and $\alpha^2=2k$, so $\grp{L}_\alpha\cong \sqrt{2k}\ZZ$.
The $bc$-ghost \voa{} $\VOA{E}(1)$ is isomorphic to $\VOA{V}_{\grp{L}_\beta}$ with $\grp{L}_\beta=\beta\ZZ$ and $\beta^2=1$, so 
$\grp{L}_\beta\cong \ZZ$.
Then the lattice $\grp{L}_\alpha\oplus \grp{L}_\beta$ contains the lattice $\grp{L}_\gamma=\gamma\ZZ$ with $\gamma=\alpha+k\beta$ as sublattice. The orthogonal complement is $\grp{N}=\mu\ZZ$ with $\mu=\alpha-2\beta$. Note, that $\gamma^2=k(k+2)$ and $\mu^2=2(k+2)$.
In \cite[Sec.\ 8]{CLI} it is proved that
\[
\VOA{S}_k:= \text{Com}\left(\VOA{V}_{\grp{L}_\mu}, \VOA{L}_k(\alg{sl}_2)\otimes \VOA{E}(1)\right)
\]
is the simple and rational $N=2$ super Virasoro algebra at central charge $c=3k/(k+2)$.

We will now explain how to obtain simple $\VOA{S}_k$-modules.
For this let $\lambda$ be an integer with $0\leq \lambda\leq k$. Further let $\Lambda_0$ and $\Lambda_1$ be the usual fundamental weights of $\widehat{\alg{sl}}_2$. Then the simple $\VOA{L}_k(\alg{sl}_2)$-modules are the integrable highest weight modules $\Mod{L}(\lambda)$ of weight $(k-\lambda) \Lambda_0+\lambda\Lambda_1$.
$\VOA{V}_{\frac{n}{2k}\alpha+\grp{L}_\alpha}$ appears in $\Mod{L}(\lambda)$ 
if and only if $\lambda+n$ is even. This follows directly since $\VOA{V}_{\frac{n}{2k}\alpha+\grp{L}_\alpha}$ appears in the decomposition 
of $\VOA{L}_k(\alg{sl}_2)$ if and only if $n$ is even.
We now express lattice vectors of $\grp{L}'_\alpha\oplus \grp{L}_\beta$ in terms of those of $\grp{L}'_\gamma\oplus\grp{L}'_\mu$, namely,
\[
\frac{a}{2k}\alpha +b\beta = (a+bk) \frac{\gamma}{k(k+2)}+(a-2b)\frac{\mu}{2(k+2)}\qquad \qquad a, b \in \mathbb Z.
\]
It follows that $\VOA{V}_{\frac{n}{2(k+2)}+\grp{N}'}$ is contained in $\Mod{L}(\lambda)\otimes \VOA{V}_{\grp{L}_\beta}$ 
if and only if $\lambda+n$ is even as well. We thus get
\[
\Mod{L}(\lambda) \otimes \VOA{V}_{\grp{L}_\beta} \cong \begin{cases}
\bigoplus\limits_{\nu\in 2\grp{N}'/\grp{N}} \VOA{V}_{\nu+\grp{N}} \otimes M(\lambda, \nu) &\qquad \text{if} \ \nu+\lambda \ \text{is even} \\
\bigoplus\limits_{\nu\in \frac{1}{2(k+2)}+2\grp{L}'/\grp{L}} \VOA{V}_{\nu+\grp{L}} \otimes M(\lambda, \nu) &\qquad \text{if} \ \nu+\lambda \ \text{is odd}
\end{cases}
\]
as $\VOA{V}_{\grp{L}_\mu}\otimes\VOA{S}_k$-modules. By \cref{thm:SW-Mod} (2) all $M(\lambda, \nu)$ are simple $\VOA{S}_k$-modules.
On the other hand, by \cref{thm:lift} for every $\mathcal L_k$-module $M$ there exists a $V_N$-module $V_{\nu+N}$ such that
\[
V_{\nu+\rho+N}\otimes M
\]
lifts to a $V_N\otimes \VOA{S}_k$-module if and only if $\rho \in (2\grp{N}')'/\grp{N}=\frac{1}{2}\grp{N}/\grp{N}$.
Finally, we announce that the relation between the tensor category of a \voa{} and its extensions can be made quite explixit \cite{CKM} and that these results imply that every simple $\VOA{S}_k$-module appears in the decomposition of at least one of the $\Mod{L}(\lambda) \otimes V_{\grp{L}_\beta}$ and moreover
\[
M(\lambda, \nu) \cong M(\lambda', \nu') \qquad \text{if and only if} \qquad \lambda'=k-\lambda\qquad\text{and}\qquad\nu'=\nu +\frac{\mu}{2} \ \ \text{mod} \ \grp{L}_\mu.
\]
\end{example}

\subsection{Rationality}
In this section we prove an interesting rationality result: Let $\VOA{V}$ be simple, rational,
	CFT-type (that is,
	conformal weights of $\VOA{V}$ are non-negative
	and the zeroth weight space is spanned by vacuum) and {$C_2$-cofinite}.
	Then, \cref{thm:rational} states that every grading-restricted generalized $\VOA{C}$-module is completely reducible.

We work with the following setup: Let $\VOA{C}=\text{Com}(\VOA{H},\VOA{V})$.
Assume that $\text{Com}(\VOA{C},\VOA{V})=\VOA{V}_{\grp{L}}$, where
$(\grp{L},\langle\cdot,\cdot\rangle)$ is a positive definite even
lattice. With this, $(\VOA{V}_\grp{L},\VOA{C})$ form a commuting pair and $\VOA{C}$ is simple.
We now collect well-known results from the literature that guarantee
that we can invoke vertex tensor category theory for $\VOA{C}$, under suitable assumptions on $\VOA{V}$.

\begin{lemma}
If $\VOA{V}$ is {$C_2$-cofinite} then so is $\VOA{C}$.
In particular, if $\VOA{V}=\VOA{L}_k(\widehat{\alg{g}})$ with $k\in\NN$ then $\VOA{C}$ is $C_2$-cofinite.
\end{lemma}
\begin{proof}
The proof of the most general statement can be found in \cite{M1}.
For the case of $\VOA{V}=\VOA{L}_k(\widehat{\alg{g}})$ with $k\in\NN$, see \cite{ALY}.
\end{proof}

\begin{lemma}
If $\VOA{V}$ is simple and CFT-type, 
then so is $\VOA{C}$.
\end{lemma}
\begin{proof}
Firstly, since $\VOA{V}_{\grp{L}}$ and $\VOA{C}$ form a commuting pair, there exists a non-zero map
$\VOA{V}_{\grp{L}}\otimes \VOA{C}\rightarrow \VOA{V}$.
Since $\VOA{V}_{\grp{L}}$ and $\VOA{C}$ are both simple, so is $\VOA{V}_{\grp{L}}\otimes \VOA{C}$
and hence this map is an injection. 
Now, $\mathbf{1}\otimes \VOA{C}_n\subset \VOA{V}_n$ for any $n$,
in particular, we conclude that $\VOA{C}_n=0$ for $n<0$ and $\VOA{C}_0=\CC\mathbf{1}_\VOA{C}$.
\end{proof}
\begin{lemma}
If $\VOA{V}$ is simple, CFT-type and self-contragredient, then so is $\VOA{C}$.
\end{lemma}
\begin{proof}
Note that $\VOA{C}$ is simple and
 we have an injection $\VOA{V}_{\grp{L}}\otimes \VOA{C}\hookrightarrow \VOA{V}$.
Since $\VOA{V}'\cong \VOA{V}$, there exists an invariant bilinear form on $\VOA{V}$
\cite{FHL}. Any invariant form on $\VOA{V}$ is automatically symmetric, by
\cite[Prop.~2.6]{Li} (see also \cite{FHL}).  Moreover, the space
of symmetric invariant forms on $\VOA{V}$ is naturally isomorphic to
$(\VOA{V}_{0}/L_1\VOA{V}_1)^*$ \cite[Thm.~3.1]{Li}.  Since
$\VOA{V}_0=\CC\mathbf{1}$, we conclude that $L_1\VOA{V}_1=0$.
Now, $L_1\VOA{V}_1=0$ implies that $L_1(\mathbf{1}\otimes \VOA{C}_1)=\mathbf{1}\otimes((L_{\VOA{C}})_1\VOA{C}_1)=0$.
This implies that $(L_{\VOA{C}})_1\VOA{C}_1=0$. This implies that $\VOA{C}_0/(L_{\VOA{C}})_1\VOA{C}_1\neq 0$, and
hence there exists a symmetric invariant bilinear form on $\VOA{C}$, by \cite[Cor.~3.2]{Li}.
In other words, $\VOA{C}'\cong \VOA{C}$. 
\end{proof}

\begin{lemma}
If $\VOA{V}$ is simple, {$C_2$-cofinite} and CFT-type, then
\begin{enumerate}
\item The category
of grading-restricted generalized modules for $\VOA{V}$ and $\VOA{C}$ satisfy the conditions needed
to invoke Huang, Lepowsky and Zhang's tensor category theory.
\item
Denoting the finite abelian group $\grp{L}'/\grp{L}$ by $\grp{G}$, there exists a subgroup $\grp{H}$ of $\grp{G}$
such that
\[\VOA{V} = \bigoplus_{\lambda\in \grp{H}}\VOA{V}_{\lambda}\otimes \VOA{C}_\lambda.\]
\item Each $\VOA{C}_\lambda$ appearing above is a simple current for $\VOA{C}$.
\end{enumerate}
\end{lemma}
\begin{proof}
(1) follows from \cite{H-projcov} and previous lemmas. (2) and (3) follow from our results above.
\end{proof}

\begin{lemma}\label{lem:GGhat}
Let $(\grp{L},\langle\cdot,\cdot\rangle)$ be a postive definite even lattice,
$\grp{L}'$ be the dual lattice 
and let $\grp{G}=\grp{L}'/\grp{L}$.
Then, $f:\mu\mapsto Q_\mu$ where $Q_\mu(\nu)=\mathrm{exp}(2\pi i \langle \mu,\nu\rangle)$
for $\mu, \nu\in \grp{G}$ is an isomorphism $\grp{G}\cong \widehat{\grp{G}}$.
\end{lemma}
\begin{proof}
It is clear that the image of $f$ is in $\widehat{\grp{G}}$.
Let $\lambda$ be in the kernel of $f$. Then, we see that
$\langle \lambda, \grp{L}'\rangle\subset\ZZ$, therefore, $\lambda\in \grp{L}''=\grp{L}$,
hence $\lambda=0$ in $\grp{G}$.
\end{proof}

\begin{lemma}
Let $\VOA{C}$ be {$C_2$-cofinite} and CFT-type. Then the endomorphism
space of any grading-restricted generalized module for $\VOA{C}$ is finite
dimensional. Moreover, each grading-restricted generalized module
has finite length and has $L_0$-Jordan blocks of bounded length.
\end{lemma}
\begin{proof}
These are the results \cite[Thm.~3.24, Prop.~4.1 and Prop.~4.7]{H-projcov}.
In fact, the conclusions hold under weaker hypotheses.
\end{proof}

\begin{thm}\label{thm:rational}
Let $\VOA{V}$ be simple, rational, {$C_2$-cofinite} and  CFT-type.
Then, every grading-restricted generalized $\VOA{C}$-module is completely reducible.
\end{thm}
\begin{proof}
We shall freely use the lemmas above.
Let $\Mod{W}$ be a grading-restricted generalized $\VOA{C}$-module.
We know that $\Mod{W}$ decomposes as a finite direct sum of indecomposable modules.
Therefore, without loss of generality, let $\Mod{W}$ be indecomposable.

Since $\Mod{W}$ is indecomposable and $\VOA{C}_\lambda$ are  finite order simple currents for every $\lambda\in\grp{H}$, by \cite[Lem.\ 3.17]{CKL},
we know that $M_{\VOA{C}_\lambda,\Mod{W}}$ is a scalar multiple, say $M_\lambda\in\CC^\times$, of identity morphism.
Let us assume that $\Mod{W}$ is such that for some non-zero 
$\VOA{C}$-modules $\Mod{R}$ and $\Mod{S}$,
we have an exact sequence:
\[0\rightarrow \Mod{R} \rightarrow \Mod{W}\rightarrow \Mod{S}\rightarrow 0.\]
We know from \cite[Lem.\ 3.19(b)]{CKL} that 
$M_{\VOA{C}_\lambda,\Mod{R}}=M_\lambda\id_{\VOA{C}_\lambda\fus{}\Mod{R}}$ and $M_{\VOA{C}_\lambda,\Mod{S}}=M_\lambda\id_{\VOA{C}_\lambda\fus{}\Mod{S}}$.
From \cref{lem:pic}, we know that $\lambda \mapsto M_\lambda^{\pm 1}$
are homomorphisms $\grp{H}\rightarrow S^1$.

We now seek a $\mu\in \grp{L}'$ such that for the $\VOA{V}_\grp{L}$ module $\VOA{V}_{\mu+\grp{L}}$,
the monodromy of $\VOA{V}_{\lambda+\grp{L}}\otimes \VOA{C}_\lambda$  with $\VOA{V}_{\mu+\grp{L}}\otimes \Mod{X}$ is trivial,
for $\Mod{X}=\Mod{R},\Mod{S},\Mod{W}$  and for all $\lambda\in \grp{H}$.
In other words, we want to find a $\mu$ such that for all $\lambda\in \grp{H}$,
\[M_{\VOA{V}_{\mu+\grp{L}},\, \VOA{V}_{\lambda+\grp{L}}}=M_\lambda^{-1}.\]
Since $\grp{H}\leq \grp{G}$ are finite abelian groups, every character of $\grp{H}$ can be extended to a character of $\grp{G}$.
Pick a $\chi\in\hat{\grp{G}}$ that
extends $\lambda\mapsto M_\lambda^{-1}$. We will be done if we can find a $\mu$
such that for each $\lambda\in \grp{G}= \grp{L}'/\grp{L}$,
\[\mathrm{exp}(2\pi i\langle\mu,\lambda\rangle)=M_{\VOA{V}_{\mu+\grp{L}},\, \VOA{V}_{\lambda+\grp{L}}}=\chi(\lambda).\]
By \cref{lem:GGhat}, we know that there indeed exists a $\mu\in \grp{L}'$ such that
$Q_{\mu}=\chi$.

For $\Mod{X}=\Mod{R},\Mod{S},\Mod{W}$, denote $V_{\mu+\grp{L}}\otimes \Mod{X}$ by $\widetilde{\Mod{X}}$  and
let
\[\widetilde{\Mod{X}}_e = \bigoplus_{\lambda\in \grp{H}}(\VOA{V}_{\lambda+\grp{L}}\boxtimes \VOA{V}_{\mu+\grp{L}})\otimes (\VOA{C}_\lambda\boxtimes \Mod{X})
=\bigoplus_{\lambda\in \grp{H}}\VOA{V}_{\lambda+\mu+\grp{L}}\otimes (\VOA{C}_\lambda\boxtimes \Mod{X}).
\]
We now invoke \cite[Thm.\ 3.4]{HKL} to get that $\widetilde{\Mod{X}}_e$ is indeed a generalized (untwisted) module for
$\VOA{V}$ when $\Mod{X}=\Mod{R},\Mod{S},\Mod{W}$.

Using flatness of simple currents, we deduce the
exact sequence of $\VOA{V}$-modules
\[ 0 \rightarrow \widetilde{\Mod{R}}_e\rightarrow \widetilde{\Mod{W}}_e\rightarrow \widetilde{\Mod{S}}_e\rightarrow 0.\]
However, every such exact sequence splits by rationality of $\VOA{V}$. Note that any morphism of $\VOA{V}$-modules
must preserve Heisenberg weights. 
Hence, we get that $0\rightarrow \Mod{R} \rightarrow \Mod{W}\rightarrow \Mod{S}\rightarrow 0$ splits.

\end{proof}

Now we can combine our results with those of \cite{H-rigidity, H-projcov} to obtain the following corollary.
\begin{cor}
	If $\VOA{V}$ is simple, rational, CFT-type and self-contragredient then we have the following:
\begin{enumerate}
\item Finite reductivity: Every $\VOA{C}$-module is completely reducible, there exist
finitely many inequivalent irreducible modules, fusion coefficients amongst irreducible
modules are finite.

\item Each finitely generated generalized $\VOA{C}$-module is a $\VOA{C}$-module.

\item The category $\Mod{C}$-modules has a structure of a modular tensor category.
\end{enumerate}
\end{cor}

\begin{example}
The Bershadsky-Polyakov algebra \cite{Ber, Pol} is the quantum Hamiltonian reduction of $\AffVOA{\ell-\frac{3}{2}}{\SLA{sl}{3}}$ for the non-principal nilpotent embedding of $\SLA{sl}{2}$ in $\SLA{sl}{3}$. This \voa\ is strongly generated by four fields of conformal dimension $1, 2, \frac{3}{2}$ and $\frac{3}{2}$. We denote its simple quotient by  $\mathcal{W}_\ell$. This \voa\ is rational provided $\ell$ is a positive integer \cite{Ar2}.
In this case it contains the lattice \voa\ $V_L$ of the lattice $L=\sqrt{6(\ell-1)}\ZZ$ as sub\ \voa \cite{ACL}. Furthermore the coset is rational, since it is isomorphic to the principal $\mathcal W$-algebra $\mathcal{W}(\SLA{sl}{2\ell})$ at level $k=-2\ell+\frac{2\ell+3}{2\ell+1}$ and central charge $c=-\frac{3(2\ell-1)^2}{2\ell+3}$ \cite{ACL}, but the latter is rational \cite{Ar3}. Our results give thus another more direct proof of rationality of this coset.  
\end{example}

\section{Heisenberg cosets inside free field algebras and 
	$\AffVOA{-1}{\SLSAp{sl}{m}{n}}$
	%$L_{-1}\left(\alg{sl}(m|n)\right)$
	}\label{sec:betagamma}

We use the opportunity to prove that 	$\AffVOA{-1}{\SLSAp{sl}{m}{n}}$
arise as certain Heisenberg cosets inside free field algebras, i.e. tensor products of $bc$ and $\beta\gamma$ systems. It had been known for a while that the affine \vosa{} is a sub-\voa{} of the coset \cite{KW2}. Moreover this gives a different proof to a recent result on the case $n=0$ and $m\geq 3$ \cite{AP}.
As simple affine \vosas{} are poorly understood at present we hope that one can use this realization to clarify the structure of  	$\AffVOA{-1}{\SLSAp{sl}{m}{n}}$-modules. 

Let $\VOA{S}$ denote the $\beta\gamma$-system, which has even generators $\beta,\gamma$ and OPE relations
$$\beta(z) \gamma(w) \sim (z-w)^{-1},\qquad \gamma(z) \beta(w) \sim -(z-w)^{-1},\qquad \beta(z) \beta(w) \sim 0,\qquad \gamma(z) \gamma(w) \sim 0.$$

Let $\VOA{H}$ be the copy of the Heisenberg algebra with generator $h = \ :\beta\gamma:$, and let $\VOA{C} = \text{Com}(\VOA{H}, \VOA{S})$. By a theorem of Wang \cite{Wa}, $\VOA{C}$ is isomorphic to the simple Zamolodchikov $\cW_3$-algebra with $c = -2$. The explicit generators, suitably normalized, are as follows:
$$ L = \ :\beta\beta\gamma\gamma: + 2: \beta \partial \gamma: - 2 :(\partial \beta) \gamma:,$$
$$W =  \ : \beta\beta\beta\gamma \gamma \gamma: + 3 : \beta \beta (\partial \gamma) \gamma:  - 6: (\partial \beta) \beta \gamma \gamma:  - 6 :(\partial \beta) \partial \gamma:  + 3 :(\partial^2 \beta)\gamma: .$$
Now let $\VOA{S}(n)$ denote the rank $n$ $\beta\gamma$-system, which has generators $\beta^i$, $\gamma^j$ for $i=1,\dots, n$ satisfying
$$\beta^i(z) \gamma^j(w) \sim \delta_{i,j} (z-w)^{-1},\qquad \gamma^i(z) \beta^j(w) \sim -\delta_{i,j} (z-w)^{-1},$$ $$ \beta^i(z) \beta^j(w) \sim 0,\qquad \gamma^i(z) \gamma^j(w) \sim 0.$$ Let $\VOA{H}$ be the Heisenberg algebra with generator $$h = \sum_{i=1}^n :\beta^i \gamma^i:,$$ and let $\VOA{C}(n) = \text{Com}(\VOA{H}, \VOA{S}(n))$. Note that $\VOA{C}(n)$ contains $n$ commuting copies of $\cW_{3}$ with generators $L^i, W^i$, obtained from $L$ and $W$ above by replacing $\beta$ and $\gamma$ with $\beta^i$ and $\gamma^i$. Moreover, $\VOA{C}(n)$ contains the fields
$$X^{jk} = -:\beta^j \gamma^k: , \qquad j,k = 1,\dots, n, \qquad j \neq k,$$
$$H^{\ell} = - :\beta^1 \gamma^1: \ + \ :\beta^{\ell+1} \gamma^{\ell+1}:,\qquad 1\leq \ell < n,$$
which generate a homomorphic image of the affine vertex algebra $V^{-1}(\alg{sl}_n)$.

A consequence of Theorem 7.3 of \cite{L} is

\begin{lemma} \label{lem:cngenerators} $\VOA{C}(n)$ is generated as a vertex algebra by $\{L^i, W^i, X^{jk}, H^{\ell} \}$ for $i,j,k,\ell$ as above.
\end{lemma}

\begin{proof} In the notation of \cite{L}, the lattice $A \subset \ZZ^n$ is spanned by $(1,1,\dots, 1)$ so $A^{\perp}$ is precisely the root lattice of $\alg{sl}_n$.
\end{proof}

By a recent theorem of Adamovi\'c and Per\v{s}e \cite{AP}, for $n\geq 3$ $\VOA{C}(n)$ is precisely the image of the map 
$\UAffVOA{-1}{\SLA{sl}{n}} %V_{-1}(\alg{sl}_n) 
\ra \VOA{C}(n)$, 
and is therefore isomorphic to the {\it simple} affine vertex algebra 
$\AffVOA{-1}{\SLA{sl}{n}}$
%$L_{-1}(\alg{sl}_n)$. 
Using Lemma \ref{lem:cngenerators}, we now provide a much shorter proof of this result. It suffices to show that $L^i$ and $W^i$ lie in the image of the map $\UAffVOA{-1}{\SLA{sl}{n}} %V_{-1}(\alg{sl}_n) 
\ra \VOA{C}(n)$, and by symmetry it is enough to prove this for $L^1$ and $W^1$. This is immediate from the following calculations:
$$L^1 =\  : H^1 H^2:  + :X^{12} X^{21}:  + :X^{13} X^{31}: - :X^{23} X^{32}: - \partial H^1 .$$
$$W^1 =  -  : H^1 H^2 H^2 :   -  : X^{12} X^{21} H^2:    -  : X^{13} X^{31} H^1 : -
: X^{13} X^{31} H^2 :  $$ $$ +  : X^{23} X^{32} H^2 : -  :X^{13} X^{32} X^{21}: +
 \frac{1}{2} :X^{12} \partial X^{21}:  -  \frac{3}{2} :(\partial X^{12}) X^{21}:  $$ $$+  \frac{7}{2} : X^{13} \partial X^{31}:  -
 \frac{9}{2} :(\partial X^{13}) X^{31}: - \frac{1}{2} :X^{23} \partial X^{32}:  + \frac{3}{2} :(\partial X^{23}) X^{32}:  $$ $$-
 \frac{1}{2} :H^1 \partial H^2: +  \frac{1}{2} : (\partial H^1) H^2:  + \frac{1}{2} \partial^2 H^1.$$

Next, we find a minimal strong generating set for the remaining case $\VOA{C}(2)$. In this case, it is readily verified that $L^1$ and $W^1$ do {\it not} lie in the affine vertex algebra generated by $X^{12}, X^{21}, H^1$. However, consider the following elements of $\VOA{C}(2)$:

$$P = -\frac{1}{2} L^2_{(0)} X^{12} + \frac{1}{3} :H^1 X^{12}:  +\frac{2}{3} \partial X^{12} $$ $$ =  \ : \beta^1 \partial \gamma^2:  - :(\partial \beta^1) \gamma^2:  + \frac{1}{3} :\beta^1 \beta^1 \gamma^1 \gamma^2: + \frac{2}{3} :\beta^1 \beta^2 \gamma^2 \gamma^2: ,$$
$$Q = -\frac{1}{2} L^1_{(0)} X^{21}  - \frac{2}{3} :H^1 X^{21}:  +\frac{1}{3} \partial X^{21} $$ $$= \ :\beta^2 \partial \gamma^1: - :(\partial \beta^2) \gamma^1: + \frac{1}{3} :\beta^1 \beta^2 \gamma^1 \gamma^1: + \frac{2}{3} :\beta^2 \beta^2 \gamma^1 \gamma^2:,$$
$$R = L^1 - L^2,$$
$$L = \ :X^{12} X^{21}:  + \frac{1}{4}  :H^1 H^1: - \frac{1}{2} \partial H^1.$$

Here $L$ is the Sugawara Virasoro field of the affine vertex algebra of 
%$V_{-1}(\alg{sl}_2)$, 
$\UAffVOA{-1}{\SLA{sl}{2}}$,
which has central charge $1$, and $X^{12}, X^{21}, H^1$ are primary of weight one with respect of $L$. It is easily verified that $P,Q,R$ are primary of weight $2$ with respect to $L$, and that $\{X^{12}, X^{21}, H^1, P,Q, R\}$ close under operator product expansion, so they strongly generate a vertex subalgebra $\VOA{C}'(2) \subset \VOA{C}(2)$. Moreover, we have
$$L^1 =   \frac{1}{2} R + :X^{12} X^{21}:  + \frac{1}{2} :H^1H^1:  -\frac{1}{2} \partial H^1, $$
$$L^2 = -\frac{1}{2} R + :X^{12} X^{21}: + \frac{1}{2} :H^1H^1: -\frac{1}{2} \partial H^1,$$
$$W^1 = - \frac{1}{2} : RH^1:  - :P X^{21}: - \frac{1}{2}  :H^1H^1H^1:  - \frac{5}{3} :X^{12} X^{21} H^1:  - \frac{13}{3}  :(\partial X^{12})X^{21}:  + \frac{10}{3} :X^{12} \partial X^{21}:  $$ $$- \frac{1}{6}  :(\partial H^1)H^1: + \frac{1}{3} \partial^2 H^1,$$
 $$W^2 = -\frac{1}{2} :RH^1: - :P X^{21}: + \frac{1}{2} :H^1H^1H^1: + \frac{4}{3}  :X^{12} X^{21} H^1: + \frac{19}{6} :(\partial X^{12}  X^{21}:  - \frac{25}{6} :X^{12} \partial X^{21}: $$ $$- \frac{5}{3} :(\partial H^1)H^1:  + \frac{3}{4} \partial R  + \frac{7}{12} \partial^2 H^1.$$
Since $\VOA{C}(2)$ is generated by $L^1, L^2, W^1, W^2,X^{12}, X^{21}, H^1$, this shows that $\VOA{C}'(2)= \VOA{C}(2)$. We obtain

\begin{thm} $\VOA{C}(2)$ is of type $\cW(1,1,1,2,2,2)$. In fact, it is the simple quotient of an algebra of type $\cW(1,1,1,2,2,2,2)$ where the Virasoro field in weight $2$ coincides with the Sugawara field. \end{thm}

\begin{remark}
Recall that each embedding of $\alg{sl}_2$ inside a reductive Lie super algebra $\alg{g}$ gives an associated affine $\mathcal W$-super algebra from the affine vertex super algebra of $\alg{g}$ at level $k$ \cite{KW}.
Denote by $W^k(\alg{sl}_4)$ the universal affine $\mathcal{W}$-algebra of $\alg{sl}_4$ for the embedding of $\alg{sl}_2$ such that $\alg{sl}_4$ decomposes into four copies of the adjoint representation of $\alg{sl}_2$ plus three copies of the trivial one. This implies that $W^k(\alg{sl}_4)$ is of type $(1, 1, 1, 2, 2, 2, 2)$ and in fact the three fields of dimension one generate the sub \voa{} $\UAffVOA{2k+2}{\SLA{sl}{2}}$. % $V_{2k+4}(\alg{sl}_2)$. 
Let $k=-5/2$ then the central charge of $W^k(\alg{sl}_4)$ is $-3$ and it contains $\AffVOA{-1}{\SLA{sl}{2}}$ %$V_{-1}(\alg{sl}_2)$ 
as sub \voa{}. 
A free field realization of $W^k(\alg{sl}_4)$ is given in Example 3.3 of \cite{AM}. A computation then reveals that the simple quotient $W_{-5/2}(\alg{sl}_4)$ is isomorphic to $\VOA{C}(2)$.
\end{remark}

Next we consider Heisenberg cosets inside $bc$-systems and $bc\beta\gamma$-system. First, consider the rank $n$ $bc$-system $\VOA{E}(n)$ with odd generators $b^i, c^i$ satisfying $$b^i(z) c^j(w) \sim \delta_{i,j} (z-w)^{-1},\qquad c^i(z) b^j(w) \sim \delta_{i,j} (z-w)^{-1},$$ $$ b^i(z) b^j(w) \sim 0,\qquad c^i(z) c^j(w) \sim 0.$$
Consider the Heisenberg algebra $\VOA{H}$ with generators $h = -\sum_{i=1}^n :b^i c^i:$, and let $\VOA{D}(n) = \text{Com}(\VOA{H}, \VOA{E}(n))$. It is well-known to be trivial for $n=1$ and isomorphic to $\AffVOA{1}{\SLA{sl}{n}}$ %$L_1(\alg{sl}_n)$ 
for $n\geq 2$.

Now we consider the Heisenberg algebra $\VOA{H}$ inside $\VOA{S}(n) \otimes \VOA{E}(m)$ with generator $$h = \sum_{i=1}^n : \beta^i \gamma^i: - \sum_{j=1}^m : b^i c^i:.$$ Let $\VOA{C}(n,m) = \text{Com}(\VOA{H}, \VOA{S}(n) \otimes \VOA{E}(m))$. It is easy to verify that $\VOA{C}(n,m)$ contains the following fields:

$$X^{jk} = -:\beta^j \gamma^k: ,\qquad j,k = 1,\dots, n, \qquad j \neq k,$$
$$H^{\ell} = - :\beta^1 \gamma^1: \ + \ :\beta^{\ell+1} \gamma^{\ell+1}:,\qquad 1\leq \ell < n,$$
$$\bar{X}^{rs} = \ :b^r c^s:, \qquad r,s = 1,\dots, m, \qquad r\neq s,$$
$$\bar{H}^{u} = \ :b^1 c^1: - :b^{u+1} c^{u+1}:,\qquad 1\leq u <m,$$
$$J^{i,r} = \ :\beta^i \gamma^i: - :b^r c^r:,\qquad 1\leq i \leq n,\qquad 1< r <m,$$
$$\phi^{r,k} = \ :b^r \gamma^k:,\qquad \psi^{j,s} = \ :\beta^j c^s:,\qquad j,k = 1,\dots, n,\qquad r,s = 1,\dots, m.$$

Moreover, these generate a homomorphic image of  $\UAffVOA{1}{\SLSAp{sl}{n}{m}}$. %$V_1(\alg{sl}(n|m))$. 
By a similar argument to the proof of Lemma \ref{lem:cngenerators}, we obtain

\begin{lemma} For all $n\geq 1$ and $m\geq 1$, $\VOA{C}(n,m)$ is generated as a vertex algebra by $L^i, W^i$ for $i=1,\dots, n$, together with the image of the map $\UAffVOA{1}{\SLSAp{sl}{n}{m}} %V_1(\alg{sl}(n|m))
	 \ra \VOA{C}(n,m)$.
\end{lemma}

\begin{thm} For all $n\geq 1$ and $m\geq 1$, $\VOA{C}(n,m)$ is isomorphic to the simple affine vertex superalgebra $\AffVOA{1}{\SLSAp{sl}{n}{m}}$. %$L_1(\alg{sl}(n|m))$.
\end{thm}

\begin{proof} Since $\VOA{C}(n,m)$ is simple, it suffices to show that $L^i, W^i$ lie in the image of the map 
	$\UAffVOA{1}{\SLSAp{sl}{n}{m}} %V_1(\alg{sl}(n|m)) 
	\ra \VOA{C}(n,m)$. By symmetry it is enough to show this for $L^1$ and $W^1$. 
	Consider the following fields in the image of 
	$\UAffVOA{1}{\SLSAp{sl}{n}{m}}$: $$J^{1,1} =\ : \beta^1 \gamma^1: - :b^1c^1:,\qquad \psi^{1,1} = \ :\beta^1 c^1:,\qquad \phi^{1,1} =\ : b^1 \gamma^1:.$$ A straightforward calculation shows that
$$L^1 =\ :J^{1,1}  J^{1,1} : -  2 :\psi^{1,1}\phi^{1,1}: +  \partial J^{1,1} ,$$
$$W^1 =\  :J^{1,1} J^{1,1} J^{1,1} :  - 3 :J^{1,1} \psi^{1,1} \phi^{1,1}: + 3:( \partial \psi^{1,1})\phi^{1,1}:  - \frac{1}{2} \partial^2 J^{1,1} .$$ \end{proof}

\section{Some $C_1$-cofiniteness results}\label{sec:C1}

In this section, we show that the simple parafermion algebra of $\alg{sl}_2$, as well as the coset of the Heisenberg algebra inside the Bershadsky-Polyakov algebra, both admit large categories of $C_1$-cofinite modules.

\subsection{The $\alg{sl}_2$ parafermion algebra}
We work with the usual generating set $X, Y, H$ for the universal affine vertex algebra $\UAffVOA{k}{\SLA{sl}{2}}$. Let $\VOA{I}_k \subset \UAffVOA{k}{\SLA{sl}{2}}$ denote the maximal proper ideal graded by conformal weight, so that the simple affine vertex algebra $\AffVOA{k}{\SLA{sl}{2}}$ is isomorphic to $\UAffVOA{k}{\SLA{sl}{2}} / \VOA{I}_k$. By abuse of notation, we use the same symbols  $X,Y,H$ for the generators of $\AffVOA{k}{\SLA{sl}{2}}$. Let $\VOA{N}_k(\alg{sl}_2) = \text{Com}(\VOA{H}, \AffVOA{k}{\SLA{sl}{2}})$ denote the simple parafermion algebra of $\alg{sl}_2$. We will prove the following.

\begin{thm} \label{thm:parafermion}For all $k\neq 0$, every irreducible $\VOA{N}_k(\alg{sl}_2)$-module appearing in $\AffVOA{k}{\SLA{sl}{2}}$ has the $C_1$-cofiniteness property according to Miyamoto's definition.
\end{thm}

In the case where $k$ is a positive integer, $\VOA{N}_k(\alg{sl}_2)$ is rational, so the $C_1$-cofiniteness of the above modules is already known. Therefore we will assume for the rest of this discussion that $k$ is not a positive integer. Since $\VOA{I}_n$ is generated by either  $:(X^{n+1}):$ or $(Y^{n+1}):$ for any positive integer $n$, it follows that if $k$ is not a positive integer, $\VOA{I}_k$ does not contain $:(X^n):$ or $:(Y^n):$ for any $n$.

Recall that $\AffVOA{k}{\SLA{sl}{2}}^{U(1)} \cong \VOA{H} \otimes \VOA{N}_k(\alg{sl}_2)$ where the $U(1)$ action is infinitesimally generated by the zero mode of the field $H$. Since each irreducible $\AffVOA{k}{\SLA{sl}{2}}^{U(1)}$-module $M$ appearing in $\AffVOA{k}{\SLA{sl}{2}}$ 
is isomorphic to $\VOA{H} \otimes N$ where $N$ is an irreducible $\VOA{N}_k(\alg{sl}_2)$-module, it suffices to prove the $C_1$-cofiniteness of the irreducible modules $M$.

Recall that for all $k \in \CC$, $\UAffVOA{k}{\SLA{sl}{2}}^{U(1)}$ has a strong generating set $$\{H, U_{0,i} = \ :X \partial^i Y:|\ i \geq 0\}.$$ For all $k\neq 0$ and $i \geq 4$, there is a relation of weight $i+2$ for the form
$$U_{0,i} = P_i(H, U_{0,0}, U_{0,1}, U_{0,2}, U_{0,3}),$$ where $P_i$ is a normally ordered polynomial in $H, U_{0,0}, U_{0,1}, U_{0,2}, U_{0,3}$, and their derivatives. Therefore $\UAffVOA{k}{\SLA{sl}{2}}^{U(1)}$ is strongly generated by $\{H, U_{0,0}, U_{0,1}, U_{0,2}, U_{0,3}\}$ and hence is of type $\cW(1,2,3,4,5)$ for all $k\neq 0$. Moreover, since the map $\UAffVOA{k}{\SLA{sl}{2}}^{U(1)} \ra \AffVOA{k}{\SLA{sl}{2}}^{U(1)}$ is surjective, the same strong generating set works for $\AffVOA{k}{\SLA{sl}{2}}^{U(1)}$.

Since $U(1)$ is compact and $\AffVOA{k}{\SLA{sl}{2}}$ is simple, we have a decomposition $$\AffVOA{k}{\SLA{sl}{2}} = \bigoplus_{n \in \ZZ} L_n \otimes M_n,$$ where $L_n$ is the irreducible, one-dimensional $U(1)$-module indexed by $n\in \ZZ$ and the $M_n$'s are inequivalent, irreducible $\AffVOA{k}{\SLA{sl}{2}}^{U(1)}$-modules. Here $M_n$ consists of elements where $H(0)$ acts by $2n$. Since $:(X^n):\  \neq 0$ and $:(Y^n):\  \neq 0$ in $\AffVOA{k}{\SLA{sl}{2}}$, and these elements lie in $M_n$ and $M_{-n}$ and have minimal conformal weight $n$, it follows that $M_n$ and $M_{-n}$ are generated as $\AffVOA{k}{\SLA{sl}{2}}^{U(1)}$-modules by $:(X^n):$ and $:(Y^n):$, respectively. Note that we have a similar decomposition $$\UAffVOA{k}{\SLA{sl}{2}} = \bigoplus_{n \in \ZZ} L_n \otimes \tilde{M}_n,$$ where the $\tilde{M}_n$'s are $\UAffVOA{k}{\SLA{sl}{2}}^{U(1)}$-modules which are no longer irreducible when $\UAffVOA{k}{\SLA{sl}{2}}$ is not simple.

Recall that a module $\Mod{M}$ for a vertex algebra $\VOA{V}$ is called {\it $C_1$-cofinite} if $\Mod{M}/ C_1(\Mod{M})$ is finite-dimensional, where $C_1(\Mod{M})$ is spanned by $$\{\alpha(k) m|\ m \in \Mod{M}, \ k <0, \ \text{wt}(\alpha) >0\}.$$ To prove the $C_1$-cofiniteness property of $M_n$ as a $\AffVOA{k}{\SLA{sl}{2}}^{U(1)}$-module for all $n$, it suffices to prove the $C_1$-cofiniteness of $M_{\pm 1}$. In fact, we shall prove a stronger statement: $\tilde{M}_{\pm 1}$ are $C_1$-cofinite as $\UAffVOA{k}{\SLA{sl}{2}}^{U(1)}$-modules. Since the map $\tilde{M}_{\pm 1} \ra M_{\pm 1}$ is surjective and compatible with the actions of $\UAffVOA{k}{\SLA{sl}{2}}^{U(1)}$ and $\AffVOA{k}{\SLA{sl}{2}}^{U(1)}$, this implies the $C_1$-cofiniteness of $M_{\pm 1}$. We only prove the $C_1$-cofiniteness of $\tilde{M}_{-1}$; the proof for $\tilde{M}_1$ is the same.

Since $\UAffVOA{k}{\SLA{sl}{2}}$ is {\it freely} generated by $X,Y,H$, it has a good increasing filtration
$$\UAffVOA{k}{\SLA{sl}{2}}_{(0)} \subset \UAffVOA{k}{\SLA{sl}{2}}_{(1)}  \subset \cdots, \qquad \UAffVOA{k}{\SLA{sl}{2}}_{(0)} = \bigcup_{d\geq 0} \UAffVOA{k}{\SLA{sl}{2}}_{(d)},$$ where $\UAffVOA{k}{\SLA{sl}{2}}_{(d)}$ is spanned by iterated Wick products of $X,Y,H$ and their derivatives, of length at most $d$. Then $\tilde{M}_{-1}$ inherits this filtration, and $(\tilde{M}_{-1})_{(d)}$ has a basis consisting of
\begin{equation} \label{basis:para} :(\partial^{i_1} H) \cdots (\partial^{i_r} H)( \partial^{j_1} X)  \cdots (\partial^{j_s} X)( \partial^{k_1} Y) \cdots (\partial^{k_s} Y)(\partial^{k_{s+1}} Y):,\end{equation} where $$i_1\geq \cdots \geq i_r \geq 0,\qquad j_1 \geq \cdots \geq j_s\geq 0,\qquad k_1\geq \cdots \geq k_s \geq  k_{s+1} \geq 0,\qquad d \geq r+2s+1.$$ In particular, $(\tilde{M}_{-1})_{(1)}$ has a basis
$$\{\partial^j Y|\ j\geq 0\}.$$

\begin{lemma} Any $\omega \in \tilde{M}_{-1}$ of weight $m>0$ is equivalent to a scalar multiple of $\partial^{m-1} Y$, modulo $C_1(\tilde{M}_{-1})$.
\end{lemma}

\begin{proof} It suffices to assume that $\omega$ is a monomial of the form \eqref{basis:para} with $r+2s >0$, which has filtration degree $r+2s+1$. Let $$\nu = \ :(\partial^{i_1} H) \cdots (\partial^{i_r} H) (U_{i_1, j_1}) \cdots (U_{i_s, j_s} )(\partial^{s+1} Y):,\qquad U_{a,b} = \ :\partial^a X \partial^b Y:.$$ and observe that $\nu$ has weight $m$ and lies in $C_1(\tilde{M}_{-1})$, and $\omega - \nu$ has filtration degree $r+2s$. Therefore by induction on filtration degree, $\omega$ is equivalent to an element of filtration degree one and weight $m$. The only such element up to scalar multiples is $\partial^{m-1}Y$. \end{proof}

Now we are ready to prove Theorem \ref{thm:parafermion}. By the preceding lemma, is enough to prove that $$\partial^i Y \in C_1(\tilde{M}_{-1}),$$ for $i$ sufficiently large. For this purpose, we compute
$$(U_{0,4})_{(0)} (\partial^i Y) = \big(k+2/5\big)\partial^{i+5}Y + \cdots,$$
where the remaining terms are of the form $$:(\partial^r H)( \partial^{i+4-r}Y):,\qquad 0 \leq r \leq i,$$ and hence lie in $C_1(\tilde{M}_{-1})$. Recall that for all $k\neq 0$, we have a relation $$U_{0,4} =P_4(H,U_{0,0}, U_{0,1}, U_{0,2},U_{0,3}).$$ We claim that $$P_4(H,U_{0,0}, U_{0,1}, U_{0,2},U_{0,3})_{(0)} (\partial^i Y) \in C_1(\tilde{M}_{-1}).$$ To see this, let $\omega$ be a term appearing in $P_4(H,U_{0,0}, U_{0,1}, U_{0,2},U_{0,3})$ of the form $:\alpha_1\dots \alpha_t:$ where $t>1$ and each $\alpha_j$ is one of the fields $H,U_{0,0}, U_{0,1}, U_{0,2},U_{0,3}$ or their derivatives. Then $\omega_{(0)} (\partial^i Y) \in C_1(\tilde{M}_{-1})$ because the zero mode of such an operator cannot consist of only annihilation operators (i.e., non-negative modes of $\alpha_j$). If $t=1$, then $\omega$ is a total derivative by weight considerations, so $\omega_{(0)} (\partial^i Y) = 0$. It follows that for all $k\neq -2/5$, $\partial^i Y \in C_1(\tilde{M}_{-1})$ for all $i \geq 5$.

Finally, suppose that $k = -2/5$. A similar computation shows that $$(U_{0,5})_{(0)} (\partial^i Y) = -\frac{1}{15} \partial^{i+6}Y + \cdots,$$ where the remaining terms are of the form $$:(\partial^r H)( \partial^{i+5-r}Y):,\qquad 0 \leq r \leq i,$$ and hence lie in $C_1(\tilde{M}_{-1})$. The same argument using the relation $U_{0,5} = P_5(H,U_{0,0}, U_{0,1}, U_{0,2},U_{0,3})$ shows that $\partial^i Y \in C_1(\tilde{M}_{-1})$ for all $i \geq 6$.

\subsection{Bershadsky-Polyakov algebras}

Let $\cW^k$ denote the universal Bershadsky-Polyakov algebra which is freely generated by fields $J, T, G^{\pm}$ of weights $1,2,\frac{3}{2}, \frac{3}{2}$, respectively, and whose OPE structure can be found in \cite{FS}. This algebra appeared originally in \cite{Ber}\cite{Pol}, and it coincides with the Feigin-Semikhatov algebra $\cW^{(2)}_3$ \cite{FS} as well as the minimal $\cW$-algebra of $\cW^k(\alg{sl}_3, f_{\text{min}})$ \cite{KW}. Let $\VOA{I}_k \subset \cW^k$ denote the maximal proper ideal graded by conformal weight, and let $\cW_k = \cW^k / \VOA{I}_k$ be the simple quotient.

The field $J$ generates a Heisenberg algebra $\VOA{H}$, and we define $$\VOA{C}^k = \text{Com}(\VOA{H}, \cW^k),\qquad \VOA{C}_k = \text{Com}(\VOA{H}, \cW_k).$$ In \cite{ACL} it was shown that $\VOA{C}^k$ is of type $\cW(2,3,4,5,6,7)$ for all $k$ except for $\{-1, -\frac{3}{2}\}$, and since there is a projection $\VOA{C}^k \ra\VOA{C}_k$, the generators of $\VOA{C}^k$ descend to give strong generator for $\VOA{C}_k$ as well.

\begin{thm} \label{thm:bpcoset} For all $k\neq -1, -\frac{3}{2}$, every irreducible $\VOA{C}_k$-module appearing in $\cW_k$ has the $C_1$-cofiniteness property according to Miyamoto's definition.
\end{thm}

The proof of this result is similar to the case of parafermion algebras above. First, suppose that $k = p/2 -3$ for $p = 5, 7, 9, \dots,$. As shown in \cite{ACL}, $\VOA{C}_{p/2-3}$ is isomorphic to the simple, rational $\cW(\alg{sl}_{p-3})$-algebra with central charge $c=-\frac{3}{p}(p-4)^2$, and $\cW_{p/2-3}$ is a simple current extension of $\VOA{C}_{p/2-3} \otimes V_L$ where $V_L$ is the lattice vertex algebra for $L = \sqrt{3p-9} \ZZ$. From this result, it is immediate that Theorem \ref{thm:bpcoset} holds in these cases, so from now on we assume that $k$ is not of this form. Since $\VOA{I}_{p/2-3}$ is generated by $:(G^+)^{p-2}):$ for $ p = 5, 7, 9, \dots$, it follows that if $k \neq p/2-3$, $\VOA{I}_{k}$ does not contain $:(G^{\pm})^n:$ for any $n>0$.

Recall that $(\cW_k)^{U(1)} \cong\VOA{H} \otimes \VOA{C}_k$ where the $U(1)$ action is infinitesimally generated by the zero mode of $J$. Since each irreducible $(\cW_k)^{U(1)}$-module $M$ appearing in $\cW_k$ is isomorphic to $\VOA{H} \otimes N$ where $N$ is an irreducible $\VOA{C}_k$-module, it suffices to prove the $C_1$-cofiniteness of the irreducible modules $M$.

By Theorem 5.3 of \cite{ACL}, for all $k\neq -1, -\frac{3}{2}$, $(\cW^k)^{U(1)}$ has a strong generating set $$\{J, L, U_{0,i} = \ :G^+\partial^i G^-:|\ i \geq 0\}.$$ For all $k\neq -1, -\frac{3}{2}$ and $i \geq 5$, there is a relation of weight $i+3$ for the form
$$U_{0,i} = P_i(J, L, U_{0,0}, U_{0,1}, U_{0,2}, U_{0,3},U_{0,4}),$$ where $P_i$ is a normally ordered polynomial in $J,L, U_{0,0}, U_{0,1}, U_{0,2}, U_{0,3},U_{0,4}$, and their derivatives. Therefore $(\cW_k)^{U(1)}$ is strongly generated by $\{J, L,  U_{0,0}, U_{0,1}, U_{0,2}, U_{0,3}, U_{0,4}\}$ and hence is of type $\cW(1,2,3,4,5,6,7)$ for all $k\neq -1, -\frac{3}{2}$. Since the map $(\cW^k)^{U(1)} \ra (\cW_k)^{U(1)}$ is surjective, the same strong generating set works for $(\cW_k)^{U(1)}$.

We have a decomposition $$\cW_k = \bigoplus_{n \in \ZZ} L_n \otimes M_n,$$ where $L_n$ is the irreducible, one-dimensional $U(1)$-module indexed by $n\in \ZZ$ and the $M_n$'s are inequivalent, irreducible $(\cW_k)^{U(1)}$-modules. Here $M_n$ consists of elements where $J(0)$ acts by $n$. This contains a unique up to scalar element $\omega_n$ of minimal weight $\frac{3n}{2}$. Here $\omega_0 = 1$, $\omega_n = (:G^-)^{-n}:$ for $n<0$, and $\omega_n =  (:G^+)^{n}:$ for $n >0$. It follows that so $M_n$ is generated as a $(\cW_k)^{U(1)}$-module by $\omega_n$ for all $n$.

As usual, to prove the $C_1$-cofiniteness of $M_n$ as a $(\cW_k)^{U(1)}$-module for all $n$, it suffices to prove the $C_1$-cofiniteness of $M_{\pm 1}$. For this purpose, it is enough to prove that $\tilde{M}_{\pm 1}$ are $C_1$-cofinite as $(\cW^k)^{U(1)}$-modules. We only prove the $C_1$-cofiniteness of $\tilde{M}_{-1}$; the proof for $\tilde{M}_1$ is the same.

Recall from \cite{ACL} that $\cW^k$ has a weak filtration $$(\cW^k)_{(0)} \subset (\cW^k)_{(1)}  \subset \cdots, \qquad (\cW^k) = \bigcup_{d\geq 0} (\cW^k)_{(d)},$$ where $(\cW^k)_{(d)}$ is spanned by iterated Wick products of $J, L, G^{\pm}$ and their derivatives, where at most $d$ of the fields $G^{\pm}$ and their derivatives appear. Then $\tilde{M}_{-1}$ inherits this filtration, and $(\tilde{M}_{-1})_{(d)}$ has a basis consisting of
\begin{equation} \label{basis:para'} :(\partial^{a_1} L) \cdots (\partial^{a_i} L) (\partial^{b_1} J )\cdots (\partial^{b_j} J) (\partial^{c_1} G^+) \cdots (\partial^{c_r} G^+)( \partial^{d_1} G^-) \cdots (\partial^{d_{r+1}} G^-):,\end{equation} where $r\geq 0$ and $0\leq a_1\leq \cdots \leq a_i$, $0\leq b_1\leq \cdots \leq b_j$, $0\leq c_1\leq \cdots \leq c_r$, and $0\leq d_1\leq \cdots \leq d_{r+1}$.

\begin{lemma} Any $\omega \in \tilde{M}_{-1}$ of weight $m+\frac{3}{2}>0$ is equivalent to a scalar multiple of $\partial^{m} G^-$, modulo $C_1(\tilde{M}_{-1})$.
\end{lemma}

\begin{proof} By the same argument as previous, $\omega$ is equivalent modulo $C_1(\tilde{M}_{-1})$ to a linear combination of terms of the form $$:(\partial^{a_1} L) \cdots (\partial^{a_i} L) (\partial^{b_1} J )\cdots (\partial^{b_j} J) (\partial^{c} G^-):.$$ All such terms except possibly $\partial^m G^-$ clearly lie in $C_1(\tilde{M}_{-1})$. \end{proof}

To prove Theorem \ref{thm:bpcoset}, it is enough to show that $$\partial^i G^- \in C_1(\tilde{M}_{-1}),$$ for $i$ sufficiently large. For this purpose, we compute
$$(U_{0,5})_{(0)} (\partial^i G^-) =  \bigg( k^2+ \frac{2}{21} k + \frac{1}{28}\bigg) \partial^{i+7}G^-  + \cdots,$$
where the remaining terms lie in $C_1(\tilde{M}_{-1})$. Recall that for all $k\neq -1,-\frac{3}{2}$, we have a relation $$U_{0,5} =P_5(J, L, U_{0,0}, U_{0,1}, U_{0,2},U_{0,3}, U_{0,4}).$$ We claim that $$P_5(J, L, U_{0,0}, U_{0,1}, U_{0,2},U_{0,3},U_{0,4})_{(0)} (\partial^i G^-) \in C_1(\tilde{M}_{-1}).$$ To see this, let $\omega$ be a term appearing in $P_5(J, L,U_{0,0}, U_{0,1}, U_{0,2},U_{0,3},U_{0,4})$ of the form $:\alpha_1\dots \alpha_t:$ where $t>1$ and each $\alpha_j$ is one of the fields $J,L,U_{0,0}, U_{0,1}, U_{0,2},U_{0,3},U_{0,4}$ or their derivatives. Then $\omega_{(0)} (\partial^i G^-) \in C_1(\tilde{M}_{-1})$ because the zero mode of such an operator cannot consist of only annihilation operators. If $t=1$, then $\omega$ is a total derivative by weight considerations, so $\omega_{(0)} (\partial^i G^-) = 0$. It follows that if $k$ is not a root of  $x^2+ \frac{2}{21} x + \frac{1}{28}$, $\partial^i G^- \in C_1(\tilde{M}_{-1})$ for all $i \geq 7$.

Finally, suppose that $k$ is a root of $x^2+ \frac{2}{21} x+ \frac{1}{28}$. A similar computation shows that $$(U_{0,6})_{(0)} (\partial^i G^-) = \bigg( k^2 +\frac{1}{56}k+\frac{3}{112} \bigg) \partial^{i+8}G^- + \cdots,$$ where the remaining terms lie in $C_1(\tilde{M}_{-1})$. Since $k$ is not a root of $x^2 +\frac{1}{56} x +\frac{3}{112}$, the same argument using the relation $U_{0,6} = P_6(J,L,U_{0,0}, U_{0,1}, U_{0,2},U_{0,3},U_{0,4})$ shows that $\partial^i G^- \in C_1(\tilde{M}_{-1})$ for all $i \geq 8$.

\appendix

\section{A proof of \cref{thm:MiySimpleCurrent}} \label{app:Miyamoto}

Let $\VOA{V}$ be a simple vertex operator algebra and let $\grp{G}$ be a finitely generated
abelian group of semi-simple automorphisms of $\VOA{V}$. 
Assume that $\VOA{V}=\bigoplus_{\lambda\in\grp{L}}\Mod{V}_{\lambda}$ for some subgroup $\lambda$
of $\widehat{\grp{G}}$. Assume that we are working with a category of $\Mod{V}_0$-modules that satisfies
the conditions required to invoke the Huang, Lepowsky, Zhang's tensor category theory.

We denote the vertex operator map of $\VOA{V}$ by $Y$.
Fix an $i \in\grp{L}$. We shall prove that $\Mod{V}_{-i}\fus{}\Mod{V}_i\cong \Mod{V}_0$. In other words, we shall prove that $\Mod{V}_i$
is a simple current. 
The proof we provide below is essentially the proof given in \cite{M2},\cite{CarMi}. 

We break the proof in several steps.
\begin{enumerate}

	\item Let us think of $Y$ as a $\Mod{V}$-intertwining operator of type $\binom{\Mod{V}}{\Mod{V}\,\Mod{V}}$. We have already
	assumed that $\Mod{V}$ is a simple VOA, i.e., $\Mod{V}$ is simple as a $\Mod{V}$-module. Using Proposition 11.9 of \cite{DL},
	we see that for any $t_1,t_2\in \Mod{V}$, $Y(t_1,x)t_2\neq 0$. This implies that coefficients of $Y(t_1,x)t_2$
	as $t_1$ runs over $\Mod{V}_j$ and $t_2$ runs over $\Mod{V}_k$ span a non-zero $\Mod{V}_0$-submodule of $\Mod{V}_{j+k}$.
	Since $\Mod{V}_{j+k}$ is a simple $\Mod{V}_0$-module,
	we get that coefficients of $Y(t_1,x)t_2$ for $t_1\in \Mod{V}_j$ and $t_2\in \Mod{V}_k$ span $\Mod{V}_{j+k}$.

	\item Given generalized $\Mod{V}_0$-modules $\Mod{A}, \Mod{B}$, we denote by $\cY^{\fus{}}_{\Mod{A},\Mod{B}}$ 
	the ``universal'' intertwining operator
	of type $\Mod{A}\fus{}\Mod{B}\choose \Mod{A}\,\Mod{B}$ furnished by the universal property of tensor products.
	If $\Mod{V}_0$ is a direct summand of $\Mod{A}$, then we assume that $\cY^{\fus{}}_{\Mod{A},\Mod{B}}$ is normalized so that
	$\cY^{\fus{}}_{\Mod{A},\Mod{B}}(v_0,x)b=Y_{\Mod{B}}(v_0,x)b$ for all $v_0\in \Mod{V}_0$ and $b\in \Mod{B}$, 
	where $Y_{\Mod{B}}$ is the module map for the $\Mod{V}_0$-module
	$\Mod{B}$.
	Moreover, for finite direct sums, $\Mod{A}=\bigoplus \Mod{A}_i$, we will assume that $\cY^{\fus{}}_{\Mod{A},\Mod{B}}\big|_{\Mod{A}_i,\Mod{B}}=\cY^{\fus{}}_{\Mod{A}_i,\Mod{B}}$.

	\item In what follows, we will often make the identification $\Mod{V}_0\fus{} \Mod{V}_r \cong \Mod{V}_r$.
	
	\item Recall that we have fixed an $i\in\grp{L}$. By \cite{HLZ}, we have the associativity of intertwining
	operators, and hence, there exists a logarithmic intertwining
	operator $\cY_{r,s;i}$ of type $\binom{\Mod{V}_s\fus{}\Mod{V}_i}{\Mod{V}_{r+s}\fus{}\Mod{V}_i\,\, \Mod{V}_r}$ such that for
	complex numbers $x, y$ with $|x|>|y|>|x-y|>0$,
	\begin{align}
	\langle
	w', \cY^{\fus{}}_{\Mod{V}_{r+s},\Mod{V}_i}(Y(u_r,x-y)u_s,y)v_i\rangle
	=
	\langle
	w',\cY_{r,s;i}(u_r,x)\cY^{\fus{}}_{\Mod{V}_s,\Mod{V}_i}(u_s,y)v_i
	\rangle,\label{eqn:cYdefn}
	\end{align}
	for any $u_r\in \Mod{V}_r, u_s\in \Mod{V}_s$ and $v_i\in \Mod{V}_i$, $w'\in (\Mod{V}_{r+s}\fus{}\Mod{V}_{i})'$.
	\item Taking $u_r=\vac$,  we get:
	\begin{align}
	\langle
	w', \cY^{\fus{}}_{\Mod{V}_s,\Mod{V}_i}(u_s,y)v_i\rangle
	=
	\langle
	w',\cY_{0,s;i}(\vac,x)\cY^{\fus{}}_{\Mod{V}_s,\Mod{V}_i}(u_s,y)v_i
	\rangle,
	\end{align}
	combining with the observation that coefficients of
	$\cY^{\fus{}}_{\Mod{V}_s,\Mod{V}_i}(t_s,y)v_i$ span $\Mod{V}_s\fus{}\Mod{V}_i$, we get that:
	\begin{align}
	\cY_{0,s;i}(\vac,x)v^e = v^e
	\end{align}
	for all $v^e\in \Mod{V}_s\fus{}\Mod{V}_i$.
	Now, using Jacobi identity we get that $\cY_{0,s+i}(u_0,x)v^e$,
	where $u_0\in \Mod{V}_0$ and $v^e\in \Mod{V}_s\fus{} \Mod{V}_i$ equals the action of $u_0$
	by the $\Mod{V}_0$-module map.

	\item Taking $u_s=\vac$ in \eqref{eqn:cYdefn}, and identifying $\Mod{V}_0\fus{}\Mod{V}_i$ with $\Mod{V}_i$, we get that:
	\begin{align}
	\langle w',\cY_{r,0;i}(u_r,x)v_i\rangle
	&=  \langle w',\cY_{r,0;i}(u_r,x)\cY^{\fus{}}_{\Mod{V}_0,\Mod{V}_i}(\vac,y)v_i\rangle\nonumber
	=  \langle w',\cY^{\fus{}}_{\Mod{V}_r,\Mod{V}_i}(Y(u_1,x-y)\vac,y)v_i\rangle\nonumber\\
	&=  \langle w',\cY^{\fus{}}_{\Mod{V}_r,\Mod{V}_i}(e^{(x-y)L_{-1}}u_1,y)v_i\rangle\nonumber
	=  \langle w',\cY^{\fus{}}_{\Mod{V}_r,\Mod{V}_i}(u_1,y+x-y)v_i\rangle,
	\end{align}
	where all the equalities hold for complex numbers $x,y$ with $|x|>|y|>|x-y|>0$.
	We may now choose $y = \frac{2}{3}x$, as this satisfies the required constraints,
	and deduce that
	\begin{align}
	\cY_{r,0;i}(u_r,x)v_i = \cY^{\fus{}}_{\Mod{V}_r,\Mod{V}_i}(u_r,x)v_i
	\label{eqn:YequalsYunivsometimes}
	\end{align}
	for all $t_r\in \Mod{V}_r$ and $v_i\in \Mod{V}_i$.

	\item 
	For complex numbers $|x|>|y|>|z|>|x-z|>|y-z|>|x-y|>0$ we have that:
	\begin{equation}\begin{split}
	\langle w', \cY_{r,s+t;i}(u_r,x)\cY_{s,t;i}(u_s,y)\cY^{\fus{}}_{\Mod{V}_t,\Mod{V}_i}(u_t,z)v_i\rangle
	&=\langle w', \cY_{r,s+t;i}(u_r,x)\cY^{\fus{}}_{\Mod{V}_s\fus{}\Mod{V}_t,\Mod{V}_i}(Y(u_s,y-z)u_t,z)v_i\rangle
	\nonumber\\
	&=\langle w', \cY^{\fus{}}_{\Mod{V}_{r+s+t},\Mod{V}_i}(Y(u_r,x-z)Y(u_s,y-z)u_t,z)v_i\rangle
	\nonumber\\
	&=\langle w', \cY^{\fus{}}_{\Mod{V}_{r+s+t},\Mod{V}_i}(Y(Y(u_r,x-y)u_s,y-z)u_t,z)v_i\rangle
	\nonumber\\
	&=\langle w', \cY_{r+s,t;i}(Y(u_r,x-y)u_s,y)\cY^{\fus{}}_{\Mod{V}_t,\Mod{V}_i}(u_t,z)v_i\rangle.
	\end{split}
	\end{equation}
	Again, since coefficients of $\cY^{\fus{}}_{\Mod{V}_t,\Mod{V}_i}$ span $\Mod{V}_t\fus{}\Mod{V}_i$,
	we get that for all $u_r \in \Mod{V}_r$ and $u_s\in \Mod{V}_s$ and $v^e\in \Mod{V}_t\fus{}\Mod{V}_i$,
	\begin{align}
	\cY_{r,s+t;i}(u_r,x)\cY_{s,t;i}(u_s,y)v^e = \cY_{r+s,t;i}(Y(u_r,x-y)u_s,z)v^e.
	\label{eqn:YYassoc}
	\end{align}
	\item Now we consider $\Mod{V}_{-i}\fus{}\Mod{V}_i$.
	Since the $Y$ map for the vertex operator algebra $\Mod{V}$ furnishes a $\Mod{V}_0$-intertwining operator
	of type $\Mod{V}_0\choose \Mod{V}_{-i}\, \Mod{V}_i$,
	by universal property of tensor products, there exists a morphism from $\Mod{V}_{-i}\fus{}\Mod{V}_i$ to $\Mod{V}_0$.
	Since the coefficients of $Y(u_{-i},x)u_i$ for $u_{-i}\in \Mod{V}_{-i}$ and $u_i\in \Mod{V}_i$ span $\Mod{V}_0$,
	$\Mod{V}_{-i}\fus{}\Mod{V}_i$ in fact surjects onto $\Mod{V}_0$. Since the latter is simple, proving simplicity of $\Mod{V}_{-i}\fus{}\Mod{V}_i$
	will give us that $\Mod{V}_{-i}\fus{}\Mod{V}_i\cong \Mod{V}_0$.
	\item Let $\Mod{B}$ be a non-zero $\Mod{V}_0$ submodule of $\Mod{V}_{-i}\fus{}\Mod{V}_i$ ($\Mod{V}_{-i}\fus{}\Mod{V}_i$ 
	is non-zero since it surjects onto $\Mod{V}_0$) and let
	\begin{align}\nonumber
	\Mod{E} = \mathrm{Span}\{\mathrm{Coefficients\,\, of\,\, } \cY_{i,-i ; i}(u_i,x)b\, | \, u_i\in \Mod{V}_i, b \in \Mod{B} \}.
	\label{eqn:defE}
	\end{align}
	Since the type of $\cY_{i,-i;i}$ is $\Mod{V}_0\fus{}\Mod{V}_i \cong \Mod{V}_i \choose \Mod{V}_i \,\,\, \Mod{V}_{-i}\fus{}\Mod{V}_i$, $\Mod{E}$ can be regarded as a $\Mod{V}_0$-submodule of $\Mod{V}_i$.
	\item $\Mod{E}$ is in fact a non-zero submodule of $\Mod{V}_i$. Indeed, if it were $0$, then, the left-hand side of
	\eqref{eqn:YYassoc} with $r=t=-i, s=i$ would be 0 and hence we would get that
	$\cY_{0,-i;i}(Y(u_{-i},x-y)u_i,y)b$ is $0$ for all $u_{-i}\in \Mod{V}_{-i}$, $u_i\in \Mod{V}_i$ and $b\in \Mod{B}$.
	However,
	in this case, coefficients of $Y(u_{-i},x-y)u_i$ span $\Mod{V}_0$ and
	$\cY_{0,-i;i}(u_0,x)b$ for $u_0\in \Mod{V}_0$  is equal to $Y_{\Mod{B}}(u_0,x)b$ where $Y_{\Mod{B}}$ is the module
	map for the $\Mod{V}_0$-module $\Mod{B}$. Since the coefficients of the module map span the entire module,
	we have a contradiction.
	\item Since $0\subsetneq \Mod{E} \subset \Mod{V}_i$ and $\Mod{V}_i$ is simple, $\Mod{E}=\Mod{V}_i$.
	\item Using $\Mod{E}=\Mod{V}_i$ and using equation \eqref{eqn:YequalsYunivsometimes},
	\begin{align*}
	\text{Span}&\{\text{Coefficients of } \cY_{-i,0;i}(v_{-i},x)\cY_{i,-i;i}(v_i,y)b\,|\, v_{-i}\in \Mod{V}_{-i}, v_i\in \Mod{V}_i, b\in \Mod{B}\}\\
	&= \text{Span}\{\text{Coefficients of }\cY_{-i,0;i}(v_{-i},x)\epsilon\,|\, v_{-i}\in \Mod{V}_{-i}, \epsilon\in \Mod{E}\}\\
	&= \text{Span}\{\text{Coefficients of }\cY_{-i,0;i}(v_{-i},x)v_i\,|\, v_{-i}\in \Mod{V}_{-i}, v_i\in \Mod{V}_i\}\\
	&= \text{Span}\{\text{Coefficients of }\cY^{\fus{}}_{\Mod{V}_{-i},\Mod{V}_i}(v_{-i},x)v_i\,|\, v_{-i}\in \Mod{V}_{-i}, v_i\in \Mod{V}_i\}\\
	&= \Mod{V}_{-i}\fus{}\Mod{V}_i
	\end{align*}
	However, using the right-hand side of equation \eqref{eqn:YYassoc},
	\begin{align*}
	\text{Span}&\{\text{Coefficients of } \cY_{-i,0;i}(v_{-i},x)\cY_{i,-i;i}(v_i,y)b\,|\, v_{-i}\in \Mod{V}_{-i}, v_i\in \Mod{V}_i, b\in \Mod{B}\}\\
	&=\text{Span}\{\text{Coefficients of } \cY_{0,-i;i}(Y(v_{-i},x-y)v_{i},y)b\,|\, v_{-i}\in \Mod{V}_{-i}, v_i\in \Mod{V}_i, b\in \Mod{B}\}\\
	&=\text{Span}\{\text{Coefficients of } \cY_{0,-i;i}(v_0,x-y)b\,|\, v_{0}\in \Mod{V}_0, b\in \Mod{B}\}\\
	&=\text{Span}\{\text{Coefficients of } Y_B(v_0,x-y)b\,|\, v_{0}\in \Mod{V}_0, b\in \Mod{B}\}\\
	&= \Mod{B}.
	\end{align*}
	This shows that $\Mod{V}_{-i}\fus{}\Mod{V}_i=\Mod{B}$ for any non-zero submodule $\Mod{B}$ of $\Mod{V}_{-i}\fus{}\Mod{V}_i$.
	We conclude that $\Mod{V}_{-i}\fus{}\Mod{V}_i$ is simple. Hence, it equals $\Mod{V}_0$.
\end{enumerate}

\flushleft
\singlespacing

\end{document}